\documentclass[a4paper,10pt,leqno]{amsart}
\title[Relative assembly maps and the  $K$-theory  of Hecke algebras \ldots]
{Relative assembly maps and the  $K$-theory  of Hecke algebras in prime characteristic}

        \author{L\"uck, W.}
        \address{Mathematicians Institut der Universit\"at Bonn\\
                Endenicher Allee 60\\
                53115 Bonn, Germany}
         \email{wolfgang.lueck@him.uni-bonn.de}
          \urladdr{http://www.him.uni-bonn.de/lueck}
         \date{July, 2024}
     \makeatletter
     \@namedef{subjclassname@2020}{%
  \textup{2020} Mathematics Subject Classification}
\makeatother

\usepackage[utf8]{inputenc}
    \usepackage{comment}
  \usepackage{calc}
  \usepackage{enumerate,amssymb,bm}
  \usepackage[arrow,curve,matrix,tips,2cell]{xy}
    \SelectTips{eu}{10} \UseTips
    \UseAllTwocells
  \usepackage{tikz}
  \usetikzlibrary{decorations.pathreplacing}
  \usepackage{color}
  \usepackage{scalerel}
  \usepackage{braket}
  \usepackage{mathtools}
  \usepackage{bbm}
  \usepackage{enumitem}
  \usepackage{float}
  \DeclareMathAlphabet{\matheurm}{U}{eur}{m}{n}
  \usepackage{hyperref}

\DeclareMathAlphabet{\matheurm}{U}{eur}{m}{n}

\newcommand{\addcat}{\matheurm{Add\text{-}Cat}}

\newcommand{\Chcat}{\matheurm{Ch}}

\newcommand{\MODcat}[1]{#1\text{-}\matheurm{Mod}}

\newcommand{\Or}{\matheurm{Or}}
\newcommand{\OrG}[1]{\matheurm{Or}(#1)}

\newcommand{\Spectra}{\matheurm{Spectra}}

\newcommand{\Sub}{\matheurm{Sub}}
\newcommand{\SubGF}[2]{\matheurm{Sub}_{#2}(#1)}

\newcommand{\contc}{\EC}

\DeclareMathOperator{\aut}{aut}

\DeclareMathOperator{\cok}{cok}
\DeclareMathOperator{\colim}{colim}
\DeclareMathOperator*{\colimunder}{colim}

\DeclareMathOperator{\conhom}{conhom}

\DeclareMathOperator{\homendo}{end}
\DeclareMathOperator{\End}{End}

\DeclareMathOperator{\fgp}{fgp}

\DeclareMathOperator*{\hocolimunder}{hocolim}

\DeclareMathOperator{\id}{id}

\DeclareMathOperator{\Idem}{Idem}

\DeclareMathOperator{\Kgroup}{K}

\DeclareMathOperator{\mor}{mor}
\DeclareMathOperator{\Nil}{Nil}

\DeclareMathOperator{\ob}{ob}

\DeclareMathOperator{\pr}{pr}
\DeclareMathOperator{\res}{res}

\DeclareMathOperator{\Sw}{Sw}

\DeclareMathOperator{\Wh}{Wh}

\newcommand{\COM}{{\calc\hspace{-1pt}\mathrm{om}}}

\newcommand{\COP}{{\calc\hspace{-1pt}\mathrm{op}}}

\newcommand{\CVCYC}{{\calc\hspace{-1pt}\mathrm{vcy}}}

\newcommand{\FIN}{{{\mathcal F}\mathrm{in}}}

\newcommand{\VCYC}{{{\mathcal V}\mathrm{cyc}}}

  \newcommand{\IN}{\mathbb{N}}

  \newcommand{\IQ}{\mathbb{Q}}
  \newcommand{\IR}{\mathbb{R}}

  \newcommand{\IZ}{\mathbb{Z}}

  \newcommand{\EA}{\matheurm{A}}  
    
  \newcommand{\EC}{\matheurm{C}}

  \newcommand{\EP}{\matheurm{P}}

  \newcommand{\cala}{\mathcal{A}}
  \newcommand{\calb}{\mathcal{B}}
  \newcommand{\calc}{\mathcal{C}}

  \newcommand{\calf}{\mathcal{F}}
   
  \newcommand{\calg}{\mathcal{G}}
  \newcommand{\calh}{\mathcal{H}}

  \newcommand{\calm}{\mathcal{M}}

  \newcommand{\calp}{\mathcal{P}}

  \newcommand{\calt}{\mathcal{T}}
    \newcommand{\calu}{\mathcal{U}}
  \newcommand{\calv}{\mathcal{V}}

  \newcommand{\bfE}{\mathbf{E}}

    \newcommand{\bfF}{\mathbf{F}}
  \newcommand{\bff}{\mathbf{f}}

  \newcommand{\bfj}{\mathbf{j}}
  \newcommand{\bfK}{\mathbf{K}}

  \newcommand{\bfp}{\mathbf{p}}

  \newcommand{\bfT}{\mathbf{T}}

 \newcommand{\bfKNilinfty}{\mathbf{K}_{\Nil}^{\infty}}

\newcommand{\EGF}[2]{E_{#2}(#1)}

\newcommand{\SETS}[1]{#1\text{-}\matheurm{SETS}}

\newcommand{\SETSf}[1]{#1\text{-}\matheurm{SETS}_{\operatorname{f}}}

\newcounter{commentcounter}

\theoremstyle{plain}
\newtheorem{theorem}{Theorem}[section]

\newtheorem{lemma}[theorem]{Lemma}
\newtheorem{corollary}[theorem]{Corollary}

\newtheorem{conjecture}[theorem]{Conjecture}
\newtheorem{assumption}[theorem]{Assumption}

\newtheorem*{theorem*}{Theorem}
\newtheorem*{theoremA*}{Theorem A}
\newtheorem*{theoremB*}{Theorem B}

\theoremstyle{definition}
\newtheorem{definition}[theorem]{Definition}

\newtheorem{remark}[theorem]{Remark}
\newtheorem{notation}[theorem]{Notation}

\newtheorem*{definition*}{Definition}

\theoremstyle{remark}

\makeatletter\let\c@equation=\c@theorem\makeatother

\theoremstyle{definition}

\newcounter{othercommentcounter}

\newcommand{\contcover}{\overline{\EC}}

\newcommand{\allG}{+}
\newcommand{\nowedge}{\sharp}

\newcommand{\version}[1]              
{\begin{center} last edited on #1\\
last compiled on \today\\
name of tex-file: \jobname
\end{center}
}

  \newcommand{\trf}{\operatorname{trf}}

\begin{document}

  \begin{abstract}
    We investigate the relative assembly map from the family of finite subgroups to the
    family of virtually cyclic subgroups for the algebraic $K$-theory of twisted group rings of a group $G$ 
    with coefficients in a regular ring $R$ or, more generally, with coefficients in a regular  additive category.
    They are known to be isomorphisms
    rationally. We show that it suffices to invert only those primes $p$ for which
    $G$  contains a non-trivial finite $p$-group and $p$ is not invertible
    in $R$. The key ingredient is the detection of Nil-terms of a  twisted group ring of a finite group
    $F$  after localizing at $p$ in terms of the $p$-subgroups of $F$ using
    Verschiebungs and Frobenius operators. We construct and exploit
    the structure of a module over the ring of big Witt vectors
    on the Nil-terms.    We analyze  the algebraic $K$-theory of the
    Hecke algebras of subgroups 
    of reductive $p$-adic groups  in prime characteristic.
  \end{abstract}

  \maketitle

  \newlength{\origlabelwidth} \setlength\origlabelwidth\labelwidth

  \typeout{------------------- Introduction -----------------}
  \section{Introduction}\label{sec:introduction}

  We first state and discuss the main results of this paper. In this introduction
  we only consider rings as coefficients for simplicity.
  Many of the results  will extend to  additive categories as coefficients. Moreover,
  in all cases one can allow a twisting by a $G$-action on $R$ or a $G$-action on the additive category.
  Groups are understood to be discrete, unless explicitly stated otherwise.


  \subsection{On the $K$-theoretic relative assembly map from $\FIN$ to $\VCYC$ for regular coefficient rings}%
  \label{subsec:On_the_K_theoretic_relative_assembly_map_from_FIN_VCYC_for_regular_coefficient_rings}

  If $\calp$ is a set of primes and $f \colon A \to B$ is a homomorphisms of abelian groups, we call
  $f$ a \emph{$\calp$-isomorphism} if the map $\id_{\IZ[\calp^{-1}]} \otimes_ {\IZ} f \colon
  \IZ[\calp^{-1}] \otimes_{\IZ} A \to \IZ[\calp^{-1}] \otimes_{\IZ} B$ is bijective, where
  the ring $\IZ[\calp^{-1}]$ satisfies $\IZ \subseteq \IZ[\calp^{-1}] \subseteq \IQ$ and
  is obtained from $\IZ$ by inverting all primes in $\calp$.

  \begin{notation}\label{calp(G,R)}
    For a group $G$ and a ring $R$, let $\calp(G,R)$ be the set of primes,
     which are not invertible in $R$ and for which $G$ contains a non-trivial finite
     $p$-subgroup.
   \end{notation}

   If $G$ is torsionfree or, more generally, the order of any finite subgroup of $G$ is
   invertible in $R$, then $\calp(G,R)$ is empty and $\calp$-isomorphism means just
   isomorphism.  If $R = \IZ$, then $\calp(G,\IZ)$ is the set of primes $p$ for which $G$
   contains an element of order $p$.

\begin{theorem}\label{the:passage_from_Fin_to_Vcyc}
  Let $R$ be a regular ring coming with a
  group homomorphism $\rho \colon G \to \aut(R)$ to the group of ring automorphisms of $R$.
  Then the relative assembly map 
  \[
    H_n^G(\EGF{G}{\FIN};\bfK_R) \to H_n^G(\EGF{G}{\VCYC};\bfK_R)
  \]
  is a $\calp(G,R)$-isomorphism for all $n \in \IZ$.
\end{theorem}

Theorem~\ref{the:passage_from_Fin_to_Vcyc}
improves~\cite[Theorem~0.3]{Lueck-Steimle(2016splitasmb)}, where $\calp(G,R)$ is required
to be the set of all primes. See also~\cite[Theorem~5.11]{Grunewald(2008Nil)} for
$R = \IZ$. If  the order of any finite subgroup of
$G$ is invertible in $R$ and the action $\rho$ is trivial, then Theorem~\ref{the:passage_from_Fin_to_Vcyc}
has already been proved in~\cite[Proposition~2.6 on page~686]{Lueck-Reich(2005)}.
Theorem~\ref{the:passage_from_Fin_to_Vcyc} is a special case of
Theorem~\ref{the:passage_from_Fin_to_Vcyc_additive_categories}.  Note that the relative
assembly map appearing in Theorem~\ref{the:passage_from_Fin_to_Vcyc} is always split
injective,
see~\cite[Theorem~1.3]{Bartels(2003b)},~\cite[Theorem~0.1]{Lueck-Steimle(2016splitasmb)}.

For more information about the  relative assembly map
appearing in Theorem~\ref{the:passage_from_Fin_to_Vcyc} and the $K$-theoretic Farrell-Jones Conjecture,
we refer to Remark~\ref{rem:FJ-assembly} and~\cite{Lueck(2022book)}.

The \emph{subgroup category of $G$ for the family $\FIN$ of finite subgroups} $\SubGF{G}{\FIN}$
  has as objects finite  subgroups $H$ of
$G$. For finite subgroups $H$ and $K$ of $G$, denote by $\conhom_G(H,K)$ the set of group
homomorphisms $f\colon H \to K$, for which there exists an element $g \in G$ with
$gHg^{-1} \subset K$ such that $f$ is given by conjugation with $g$, i.e.
$f = c(g): H \to K, \hspace{3mm} h \mapsto ghg^{-1}$.  Note that $c(g) = c(g')$ holds for
two elements $g,g' \in G$ with $gHg^{-1} \subset K$ and $g'Hg'^{-1} \subset K$ if and
only if $g^{-1}g'$ lies in the centralizer
$C_GH = \{g \in G \mid gh=hg \mbox{ for all } h \in H\}$ of $H$ in $G$. The group of inner
automorphisms $\operatorname{Inn}(K)$ of $K$ acts on $\conhom_G(H,K)$ from the left by
composition. Define the set of morphisms
\[
  \mor_{\SubGF{G}{\FIN}}(H,K) := \operatorname{Inn}(K) \backslash \conhom_G(H,K).
\]
Equivalently, $\mor_{\Sub_\COP}(H,K)$ is the double coset
$K \backslash \{ g \in G \mid gHg^{-1} \subseteq K\} / C_G(H)$ where the left $K$-action and
the right $C_G(H)$-action come from the multiplication in $G$.

\begin{remark}[The Full Farrell-Jones Conjecture]\label{rem:staus_of_FJC}
  The Full Farrell-Jones Conjecture is stated in~\cite[Conjecture~13.27]{Lueck(2022book)}.  Here we only need  to know
  that it implies that the assembly map
  \begin{equation}
    H_n^G(\EGF{G}{\VCYC};\bfK_R) \to K_n(R_{\rho}[G])
    \label{FJS-assembly_for_rings}
  \end{equation}
  is bijective for all $n \in \IZ$ and any ring $R$ coming with a group homomorphism
  $\rho \colon G \to \aut(R)$, see~\cite[Theorem~13.61~(i)]{Lueck(2022book)}.
 
  Note that the Full Farrell-Jones Conjecture is known to be true for a large class of
  groups including hyperbolic groups, CAT(0)-groups, lattices in locally compact second
  countable Hausdorff groups, and fundamental groups of manifolds of dimension $\le 3$ and
  has useful inheritance properties, e.g., passing to subgroups and overgroups of finite
  index, see for instance~\cite[Chapter~15]{Lueck(2022book)}.
\end{remark}

\begin{theorem}\label{the:FJ_and_regular}
  Suppose $G$ satisfies the Full Farrell-Jones Conjecture, Let $R$ be a regular ring
  coming with a group homomorphism $\rho \colon G \to \aut(R)$ such that the order of any
  finite subgroup of $G$ is invertible in $R$.
  
  Then  the canonical map
     \[
       \colimunder_{H \in \SubGF{G}{\FIN}} K_0(R_{\rho|_H}[H]) \to K_0(R_{\rho}[G])
      \]
     is an isomorphism and 
  \[
     K_n(R_{\rho}[G])  = 0 \quad \text{for} \;n \le -1.
   \]
   where $R_{\rho}[G]$ denotes the $\rho$-twisted group ring.
 \end{theorem}


  \subsection{On the $K$-theoretic relative assembly map from $\FIN$ to $\VCYC$ for Artinian  coefficient rings}%
  \label{subsec:On_the_K_theoretic_relative_assembly_map_from_FIN_VCYC_for_Artinian_coefficient_rings}

  In the case that $R$ is an Artinian ring,
  we get even an integral result in degree $n \le 0$ without the assumption that
  the order of any finite subgroup of $G$ is invertible in $R$.

  \begin{theorem}\label{the:passage_from_Fin_to_Vcyc_Artinian}
    Let $G$ be a discrete group which satisfies the Full Farrell-Jones Conjecture.  Let $R$
    be an Artinian ring coming with a group homomorphism $\rho \colon G \to \aut(R)$.

    Then  the canoncial map
    \[
      \colimunder_{H \in \SubGF{G}{\FIN}} K_0(R_{\rho|_H}[H]) \to K_0(R_{\rho}[G])
    \]
    is an isomorphism and
    \[
      K_n(R_{\rho}[G]) = 0 \quad \text{for} \;n \le -1.
    \]
  \end{theorem}
  \begin{proof} This is proved in~\cite[Theorem~13.61~(v)]{Lueck(2022book)}
    for trivial $\rho$.  It directly extends to the case where $\rho$ is non-trivial.
  \end{proof}

  Note that the prototype of results such as Theorem~\ref{the:passage_from_Fin_to_Vcyc_Artinian} is due
  to Moody~\cite{Moody(1987)}, who proved the bijectivity of the canonical map
  $\colimunder_{H \in \SubGF{G}{\FIN}} K_0(RH) \to K_0(RG)$ for $R$ a field of characteristic
  zero and $G$ a virtually polycyclic group.


  \subsection{On the $K$-theoretic relative assembly map from $\FIN$ to $\VCYC$ for fields as coefficients}%
  \label{subsec:On_the_K_theoretic_relative_assembly_map_from_FIN_VCYC_for_fields_as_coefficients}

  For the reader's convenience we summarize what happens in the special case, where $R$ is
  a skew-field $F$.  Note that a skew-field of characteristic zero is regular and
  satisfies $\calp(G,F) = \emptyset$ and any skew-field is an Artinian ring. Hence we
  conclude from Theorem~\ref{the:passage_from_Fin_to_Vcyc},
  Theorem~\ref{the:FJ_and_regular}, and
  Theorem~\ref{the:passage_from_Fin_to_Vcyc_Artinian}

  \begin{theorem}\label{the:field_coefficients}
    Let $G$ be a group and $F$ be a skew-field coming
    with a group homomorphism $\rho \colon G \to \aut(F)$
    into the group of field automorphism of $F$. Then:

    \begin{enumerate}

    \item\label{the:field_coefficients:relative_assembly_integral}
        The relative assembly map
        \[
        H_n^G(\EGF{G}{\FIN};\bfK_F) \to H_n^G(\EGF{G}{\VCYC};\bfK_F)
      \]
      is bijective for every $n \in \IZ$  if one of the following condition holds:

      \begin{itemize}
      \item $F$ has characteristic zero;
      \item There exists a prime $p$ such that $F$ has characteristic $p$
        and $G$ contains no non-trivial finite $p$-group;
      \end{itemize}

    \item\label{the:field_coefficients:relative_assembly_characteristic_p}
        If $p$ is a prime and $F$ is a skew-field of characteristic $p$, then the map
        \[
        H_n^G(\EGF{G}{\FIN};\bfK_F)[1/p]\to H_n^G(\EGF{G}{\VCYC};\bfK_F)[1/p]
      \]
      is bijective for every $n \in \IZ$;

      \item\label{the:field_coefficients:K_0}
      The canonical map
      \[
       \colimunder_{H \in \SubGF{G}{\FIN}} K_0(F_{\rho|_H}[H]) \to K_0(F_{\rho}[G])
      \]
      is bijective  if $G$ satisfies the Full Farrell-Jones Conjecture;
        
    \item\label{the:field_coefficients:negative}
      We have $K_n(F_{\rho}[G])$ for $n \le -1$,  if $G$ satisfies the Full Farrell-Jones Conjecture.
  \end{enumerate}
\end{theorem}


  \subsection{On the $K$-theoretic relative assembly map from $\FIN$ to $\VCYC$ for $\IZ$ as coefficients}%
  \label{subsec:On_the_K_theoretic_relative_assembly_map_from_FIN_VCYC_for_Z_as_coefficients}

  For the reader's convenience we summarize what is known in the special case
  where $R$ is  the ring $\IZ$ of integers.

  \begin{theorem}\label{the:coefficients_in_Z}
    Let $G$ be a group. Then

    \begin{enumerate}
       \item\label{the:coefficients_in_Z_relative_assembly_n_arbitrary}
      The relative assembly map
      \[
        H_n^G(\EGF{G}{\FIN};\bfK_\IZ) \to H_n^G(\EGF{G}{\VCYC};\bfK_\IZ)
      \]
      is a $\calp(G;\IZ)$-isomorphism for every $n \in \IZ$, where $\calp(G,\IZ)$ is the set of primes $p$,
      for which $G$ contains a non-trivial finite $p$-group;
      
    \item\label{the:coefficients_in_Z_relative_assembly_n_le_-1}
      The canonical map
      \[
       \colimunder_{H \in \SubGF{G}{\FIN}} K_{-1}(\IZ_{\rho|_H}[H]) \to K_{-1}(\IZ_{\rho}[G])
      \]
      is bijective and $K_n(\IZ[G]) = 0$ for $n \le -2$ if $G$ satisfies the Full Farrell-Jones Conjecture;

    \item\label{the:coefficients_in_Z:torsionfree}
        Suppose  that $G$ is torsionfree and  satisfies the Full Farrell-Jones Conjecture.
       Then $K_n(\IZ G)$ for $n \le -1$, the reduced projective class group $\widetilde{K}_0(\IZ G)$,
       and the Whitehead group $\Wh(G)$ are trivial.
     \end{enumerate}
   \end{theorem}
   \begin{proof}~\ref{the:coefficients_in_Z_relative_assembly_n_arbitrary}
     This follows directly from Theorem~\ref{the:passage_from_Fin_to_Vcyc}.
     \\[1mm]~\ref{the:coefficients_in_Z_relative_assembly_n_le_-1}
     See~\cite[Theorem~13.61~(vi)]{Lueck(2022book)}.
     \\[1mm]~\ref{the:coefficients_in_Z:torsionfree}
     See~\cite[Theorem~13.61~(iii) and~(iv)]{Lueck(2022book)}.
   \end{proof}
   

   \subsection{Totally disconnected groups}\label{subsec:totally_disconnected_groups}

   So far $G$ has been a discrete group. Now we want to deal with td-groups, i.e.,
   locally compact second countable totally disconnected topological Hausdorff groups, and
   the algebraic $K$-theory of  their Hecke algebras.  In some special cases or in a weaker form,
   we  extend the main results
   of~\cite{Bartels-Lueck(2023K-theory_red_p-adic_groups)}
   from characteristic zero to prime characteristic.
   
   Let $R$ be a (not necessarily commutative) ring.  We will need the following assumption
   to make sense of the notion of a Hecke algebra. It is taken from~\cite[page~9]{Blondel(2011)}.

    \begin{assumption}\label{ass:existence_of_widetilde(U)_G-Z-categories}
   There exists a compact open subgroup $\widetilde{U}$ of $G$ such that for any compact
   open subgroup $\widetilde{U}'\subseteq \widetilde{U}$ of $\widetilde{U}$ the index
   $[\widetilde{U}: \widetilde{U}']$ is invertible in $R$.
 \end{assumption}
 This assumption is automatically satisfied if $G$ is discrete, since then we can take
 $\widetilde{U} = \{1\}$, or if $\IQ \subseteq R$.  If $p$ is a prime number which is invertible in $R$,
 then Assumption~\ref{ass:existence_of_widetilde(U)_G-Z-categories}
 is satisfied for any subgroup of a reductive $p$-adic group $G$
 by~\cite[Lemma~1.1]{Meyer-Solleveld(2010)} and Lemma~\ref{lem:cal_and_subgroups}.

   Suppose that
   Assumption~\ref{ass:existence_of_widetilde(U)_G-Z-categories} is satisfied.
   Then the Hecke algebra $\calh(G;R)$ is defined as the algebra of locally constant functions $G \to R$
   with compact support and  multiplication given by  convolution, see for
   instance~\cite[Section~11]{Bartels-Lueck(2023forward)}.

One may consider the $K$-groups $K_n(\calh(G;R))$ of $\calh(G;R)$ for $n \in \IZ$.  There
is an assembly map
\begin{equation}\label{eq:COP-assembly-homolgy-theory-R}
  H_n^G(\EGF{G}{\COP};\bfK_R) \to H_n^G(G/G;\bfK_R) = K_n(\calh(G;R)).
\end{equation}
Here $H_*^G$ is a $G$-homology theory digesting smooth $G$-$CW$-complexes, which satisfies
$\calh^G_n(G/U;\bfK_R) \cong K_n(\calh(U;R))$ for every open subgroup $U$ of $G$, the
$G$-$CW$-complex $\EGF{G}{\COP}$ is any model for the classifying space for proper smooth
$G$-actions, and the map~\eqref{eq:COP-assembly-homolgy-theory-R} is induced by the
projection $\EGF{G}{\COP}  \to G/G$.  We say that  $R$ is \emph{$l$-uniformly 
    regular} for the natural number $l$ if $R$ is noetherian and
     every $R$-module admits a projective resolution
     of length at most $l$. We call $R$ \emph{uniformly regular}
     if $R$ is $l$-uniformly regular for  some natural number $l$.
     If $G $ is modulo a compact subgroup isomorphic to a
  closed subgroup of a reductive $p$-adic group and $R$ is uniformly regular and satisfies $\IQ \subseteq R$, e.g., $R$
  is a field of characteristic zero, then we conclude
  from~\cite[Corollary~1.18]{Bartels-Lueck(2023K-theory_red_p-adic_groups)}  that the
map~\eqref{eq:COP-assembly-homolgy-theory-R} is bijective, 
$K_n(\calh(G;R))$ vanishes for $n \le -1$, and the canonical map
\begin{equation}
  \colimunder_{U \in \Sub_\COP(G)} \Kgroup_0 (\calh(U;R)) \to \Kgroup_0 (\calh(G;R))
  \label{colim_K_0}
\end{equation}
is bijective. Here $\Sub_\COP(G)$ is analogously defined as
  $\SubGF{G}{\FIN}$ but now for $\COP$ the family of compact open subgroups of $G$.

  Next we want to explain what we can say in the case, where the condition that $R$ is
  uniformly regular and satisfies $\IQ \subseteq R$ is weakened to the condition that $R$ is
  uniformly regular   and only certain primes have to be invertible in $R$ or to the condition that
  $N \cdot 1_R = 0$ holds in $R$ for some natural number $N$.

\begin{theorem}\label{the:passage_from_COP_to_CVcyc_in_characteristic_N}
  Let $p$ be a prime. Assume that $G $ is modulo a compact subgroup isomorphic to a
  closed subgroup of a reductive $p$-adic group. Let $N$ be a natural number and let $R$ a
  be ring with unit $1_R$.

  \begin{enumerate}
  \item\label{the:passage_from_COP_to_CVcyc_in_characteristic_N:general_char_p} Suppose
    that $N \cdot 1_R = 0$ and that
   Assumption~\ref{ass:existence_of_widetilde(U)_G-Z-categories} is satisfied.
    Then the assembly map~\eqref{eq:COP-assembly-homolgy-theory-R}
    induces an isomorphism
    \[H_n^G(\EGF{G}{\COP};\bfK_R)[1/N] \to K_n \big(\calh(G;R)\big)[1/N]
    \]
    for every $n \in \IZ$;

  \item\label{the:passage_from_COP_to_CVcyc_in_characteristic_N:Artinian} Suppose that
    $N \cdot 1_R = 0$, and that $R$ is Artinian, e.g.,
    $R$ is a field of  prime characteristic $q$ for $p \not = q$ and we take $N = q$. Suppose that
   Assumption~\ref{ass:existence_of_widetilde(U)_G-Z-categories} is satisfied.
    Then
    \[ K_n (\calh(G;R))[1/N] = 0 \quad \text{for}\; n \le -1
    \]
    and the map induced by~\eqref{colim_K_0}
    \[
      \colimunder_{U \in \Sub_\COP(G)} \Kgroup_0 (\calh(U;R))[1/N] \to \Kgroup_0
      (\calh(G;R))[1/N]
    \]
    is bijective;

  \item\label{the:passage_from_COP_to_CVcyc_in_characteristic_N:general} Suppose that $R$
    is uniformly regular and for any two compact open subgroups $U_0$ and $U_1$ of $G$
    with $U_0 \subseteq U_1$ the index $[U_1 : U_0]$ is invertible in $R$. Then:

    \begin{enumerate}
    \item    Assumption~\ref{ass:existence_of_widetilde(U)_G-Z-categories} is satisfied;

    \item The  map~\eqref{eq:COP-assembly-homolgy-theory-R} is bijective for all $n \in \IZ$;

    \item We have     $K_n (\calh(G;R))= 0$ for $n \le 1$;

    \item The canoncial map
    \[
      \colimunder_{U \in \Sub_\COP(G)} \Kgroup_0 (\calh(U;R))\to \Kgroup_0 (\calh(G;R))
    \]
    is bijective.
  \end{enumerate}
\end{enumerate}
\end{theorem}

One can define more general Hecke algebras $\calh(G,R,\rho,\omega)$
allowing a $G$-action $\rho$ on $R$ and central
character $\omega$  and everything carries over to this more general setting,
see Remark~\ref{rem:calp(underline(R))}. 

The proof of Theorem~\ref{the:passage_from_COP_to_CVcyc_in_characteristic_N}
will be given in Subsection~\ref{subsec:Proof_of_Theorem_ref(the:passage_from_COP_to_CVcyc_in_characteristic_N)}
and is based on a version of the Farrell-Jones Conjecture
for totally disconnected groups with categories with $G$-support as coefficients.


  \subsection{On the twisted Nil-terms of finite groups}%
  \label{subsec:On_the_twisted_Nil-terms_of_finite_groups}

  The proof of some of the results above relies on the following theorem.

  Let $F$ be  a finite group and $\alpha \colon  F \to F$ be a group automorphism.
  Let $F \rtimes_{\alpha} \IZ$ be the semidirect product associated to $\alpha$, where for the standard
   generator $t \in \IZ$ we have $tft^{-1} = \alpha(f)$ for $f \in F$.
  Let $R$ be a  ring coming with a group homomorphism $\mu \colon F \rtimes_{\alpha} \IZ \to \aut(R)$.
  Let $\rho \colon F \to \aut(R)$ be the restriction of $\mu$ to $F$.
  Our goal is to get information about the structure of the Nil-groups
  \[
    N\!K_n(R_{\rho}[F],\Psi) = \overline{K}_{n-1}(\Nil(R_{\rho}[F],\Psi))
  \]
  with  respect  to   the ring automorphism
  $\psi  \colon  R_{\rho}[F]  \xrightarrow{\cong} R_{\rho}[F]$ sending $r\cdot f$ for $r \in R$ and $f \in F$ to
  $\mu(t)(r) \cdot \alpha(f)$, where $\overline{K}_{n-1}(\Nil(R_{\rho}[F],\Psi))$
  is defined in Notation~\ref{not:overline(K)(Nil))} taking $\cala = \underline{R_{\rho}[F]}$
  for the $\IZ$-category $ \underline{R_{\rho}[F]}$, which has precisely one object 
and whose $\IZ$-module of endomorphisms is $R_{\rho}[F]$. 
  These  Nil-groups  appear in  the  twisted  Bass-Heller-Swan decomposition  for
  $R_{\mu}[F \rtimes_{\alpha}\IZ]   =    (R_{\rho}[F])_{\psi}[\IZ]$,
  see~\eqref{twisted_BHS_on_homotopy_groups}.

  Fix a prime number $p$.  Let $T_p$ be the set of triples $(P,k,y)$ consisting of a
  $p$-subgroup $P$ of $F$, an integer $k$ with $k \ge 1$, and an element $y \in F$ such
  that $c_y \circ \alpha^k(P) = P$ holds for the automorphism $c_y \colon F \to F$ sending
  $z$ to $yzy^{-1}$. Let
  $\psi_{(P,k,y)} \colon R_{\rho|_P}[P] \xrightarrow{\cong} R_{\rho|_P}[P]$ be the ring
  automorphism sending $r \cdot p$ for $r \in R$ and $p \in P$ to
  $\mu(yt^k)(r) \cdot c_y \circ \alpha^k(p)$.

  Given a triple $(P,k,y) \in T_p$, define a functor of Nil-categories
  \[
    \gamma(P,k,y) \colon \Nil(R_{\rho|_P}[P],\psi_{(P,k,y)}) \to \Nil(R_{\rho}[F],\psi)
  \]
  by sending an object in $\Nil(R[P],\psi|_P)$ given by a nilpotent $R[P]$-endomorphism
  $\varphi \colon (c_y \circ \alpha^k)_* Q = RP \otimes_{c_y \circ \alpha^k} Q \to Q$ for a
  finitely generated projective $R[P]$-module $Q$ to the object in $\Nil(R[F],\alpha^k)$
  given by the nilpotent $R[F]$-endomorphism
  $(\alpha^k)_* \bigl(R[F] \otimes_{R[P]} Q\bigr) = RF \otimes_{\alpha^k} \bigl(R[F]
  \otimes_{R[P]} Q\bigr) \to R[F] \otimes_{R[P]} Q$ sending $f_0 \otimes (f_1 \otimes q)$
  to $(\alpha^{-k}(f_0 ) f_1t^{-k}y^{-1}) \otimes \varphi(1 \otimes q)$. It induces for every
  $n \in \IZ$ a homomorphism
  \[ \gamma(P,k,y)_m \colon N\!K_n(R[P],c_y \circ \alpha^k) \to N\!K_n(R[F],\alpha^k).
  \]
  Let
  \[(V_k)_n \colon N\!K_n(R[F],\alpha^k) \to N\!K_n(R[F],\alpha)
  \]
  be the homomorphism induced by the Verschiebungs operator $V_k$, see~\eqref{Verschiebung_V_k}.
  
  \begin{theorem}\label{the:compuation_of_Nil-Terms_for_autos_of_finite_groups}
    The homomorphism
    \[\bigoplus_{(P,k,y) \in T_p}\bigl((V_K)_n \circ \gamma(P,k,y)_n\bigr)_{(p)} \colon
      \bigoplus_{(P,k,y) \in T_p} N\!K_n(R[P],c_y \circ \alpha^k)_{(p)} \to
      N\!K_n(R[F],\alpha)_{(p)}
    \]
    is surjective for every $n \in \IZ$, where the subscript $(p)$ stands for localization at the prime $p$.
  \end{theorem}

  The untwisted version of Theorem~\ref{the:compuation_of_Nil-Terms_for_autos_of_finite_groups}, i.e.,
  $\alpha = \id_F$ and trivial $\mu$,
  appears already in~\cite[Theorem~A]{Hambleton-Lueck(2012)}. 

  Note that this does not mean that the Nil-groups are computable after localizing at $p$
  by $p$-subgroups groups, since the maps $\gamma(P,k,y)_n$ are not given just by
  induction with the inclusion $P \to F$.  One can check by inspecting
  Lemma~\ref{p-hyper_elementary_implies_p_elementary} and
  Lemma~\ref{lem:Computabilty_in_terms_of_p-hyperlelementary_subgroups_final}
  that the Nil-groups are computable by $p$-elementary groups.

  \begin{corollary}\label{cor:_vanishing_of_Nil-Terms_for_autos_of_finite_groups}
    Let $R$ be a regular ring. Then
    $\IZ[\calp(F,R)^{-1}] \otimes_{\IZ} N\!K_n(R[F],\alpha)$ vanishes for $\calp(F,R)$
    defined in Notation~\ref{calp(G,R)}.
  \end{corollary}

  We mention that the  second Nil-group of $F_2[\IZ/2]$ is non-trivial,
  see~\cite{Kallen(1971)}.  So one needs to invert certain primes in
  Corollary~\ref{cor:_vanishing_of_Nil-Terms_for_autos_of_finite_groups}.
  
  On the other hand, given a prime $p$, we get $N\!K_n(\IZ[\IZ/p])=0$ for $n \le 1$,
  see~\cite[Theorem~10.6  on page~695]{Bass(1968)},~\cite{Bass-Murthy(1967)},~\cite[Theorem~6.21]{Lueck(2022book)}.
  So Theorem~\ref{the:compuation_of_Nil-Terms_for_autos_of_finite_groups} implies that for a
  finite group $G$, for which $p^2$ does not divide the order of $G$, we have
  $N\!K_n(\IZ G)_{(p)} = 0$ for $n \le 1$.  As an application we get a new proof of the
  result of Harmon~\cite{Harmon(1987)} that $N\!K _n(\IZ G) = 0$ for $n \le 1$ if the order of
  $G$ is square-free.

  \begin{remark}\label{rem:all_groups_of_order_p}
    We mention without giving the details that the proof appearing
    in~\cite[Theorem~6.21]{Lueck(2022book)} can be generalized to the twisted setting showing
    that Harmon's result extends to twisted group rings and that in
    Theorem~\ref{the:compuation_of_Nil-Terms_for_autos_of_finite_groups} the terms
    $N\!K_n(R[P],c_y \circ \alpha^k)_{(p)}$ vanish if $|P| \le p$, $R= \IZ$, and
    $n \le 1$ hold.  This implies that for a group $G$ for which the order of any finite
    subgroup is squarefree the relative assembly map
    \[
      H_n^G(\EGF{G}{\FIN};\bfK_\IZ) \to H_n^G(\EGF{G}{\VCYC};\bfK_\IZ)
    \]
    is an isomorphism for every $n \in \IZ$ with $n \le 1$. If $G$ satisfies the Full
    Farrell-Jones Conjecture and the order of any finite subgroup is squarefree, then the
    assembly map
    \[
      H_n^G(\EGF{G}{\FIN};\bfK_\IZ) \to H_n^G(G/G;\bfK_\IZ) = K_n(\IZ[G])
    \]
    is an isomorphism for every $n \in \IZ$ with $n \le 1$. Examples for such $G$ are
    given by extensions $1 \to \IZ^n \to G \to \IZ/m \to 1$ for a squarefree
    natural number $m$. 
    
  \end{remark}

  \begin{remark}[K-theory of stable $\infty$-categories]\label{rem:infty-categories}
    One may ask whether the results of this paper can be extended from the $K$-theory of  additive categories
    to the $K$-theory of stable $\infty$-categories.  For the extension of the statement and proofs in some cases
    of the Full Farrell-Jones Conjecture,
    we refer to
    Bunke-Kasprowski-Winges~\cite{Bunke-Kasprowski-Winges(2021)}. Dominik Kirstein and
    Christian Kremer are working on a twisted Bass-Heller-Swan decomposition in this
    setting generalizing~\cite{Lueck-Steimle(2016BHS)} and~\cite{Saunier(2023)}.
    Efimov has announced that the
    Nil-terms are modules over TR on the spectrum level, which would yield a module
    structure of Nil-groups over the ring of big Witt vectors also for stable 
    $\infty$-categories.  However, algebraic $K$-theory for additive categories can be
    viewed as a Mackey functor over the Green functor given by the Swan ring, see
    Subsection~\ref{subsec:Sw_F}. This is very unlikely to be the case for the algebraic
    $K$-theory of stable $\infty$-categories, where an $A$-theoretic version of the Swan group
    is needed, see~\cite{Ullmann-Winges(2019)}.  Therefore the induction theorems, which we
    use  here for instance to prove
    Theorem~\ref{the:compuation_of_Nil-Terms_for_autos_of_finite_groups}, are  not
    available in the setting of stable $\infty$-categories. It is completely unclear
    in the setting of stable $\infty$-categories
    whether Theorem~\ref{the:passage_from_Fin_to_Vcyc} or
    Theorem~\ref{the:compuation_of_Nil-Terms_for_autos_of_finite_groups} are still  true
    and  how one can formulate the statement of
    Theorem~\ref{the:passage_from_COP_to_CVcyc_in_characteristic_N}.
  \end{remark}


\subsection{Acknowledgments}\label{subsec:Acknowledgements}

The paper is funded by the ERC Advanced Grant \linebreak ``KL2MG-interactions'' (no.
662400)  granted by the European Research Council and  by the Deutsche
Forschungsgemeinschaft (DFG, German Research Foundation) under Germany's Excellence
Strategy \--- GZ 2047/1, Projekt-ID 390685813, Hausdorff Center for Mathematics at Bonn.
The author thanks  Ian Hambleton, Dominik Kirstein, and Christian Kremer  for helpful conversations
and the (unknown) referee for his detailed and very helpful report.

On behalf of all authors, the corresponding author states that there is no conflict of interest.

My manuscript has no associated data.

The paper is organized as follows:
\tableofcontents


\typeout{---------- Section 2: Basics about (additive) $\Lambda$-categories
  ---------------}

\section{Basics about (additive) $\Lambda$-categories}%
\label{sec:Basics_about_(additive)_Lambda-categories}

Consider a commutative ring $\Lambda$ and a group $G$. A \emph{$\Lambda $-category} is a
small category $\cala$ enriched over the category of $\Lambda$-modules, i.e., for every
two objects $A$ and $A'$ in $\cala$ the set of morphisms $\mor_{\cala}(A,A')$ has the
structure of a $\Lambda$-module such that composition is a $\Lambda$-bilinear map.  A
\emph{$G$-$\Lambda$-category} is a $\Lambda$-category, which comes with a $G$-action by
automorphisms of $\Lambda$-categories.

If a $\Lambda$-category comes with an appropriate notion of a finite direct sums, it is
called an \emph{additive $\Lambda$-category}.  An \emph{additive $G$-$\Lambda$-category}
is an additive $\Lambda$-category, which comes with a $G$-action by automorphisms of
additive $\Lambda$-categories. If $\Lambda$ is $\IZ$, we often omit $\Lambda$ and talk
just about an additive category or additive $G$-category.

One can associate to a $\Lambda$-category category $\cala$ an additive $\Lambda$-category
$\cala_{\oplus}$ as follows.  Objects in $\cala_{\oplus}$ are pairs $(S,\EA)$ consisting
of a finite set $S$ and a map $\EA \colon S \to \ob(A)$ and a morphism
$\psi \colon (S,\EA) \to (S',\EA')$ is given by a collection
$\psi_{s,s'} \colon \EA(s) \to \EA(s')$ of morphisms in $\cala$ for $s \in S$ and
$s' \in S'$.  The direct sum $(S,\EA) \oplus (S',\EA')$ is given by
$(S \amalg S', \EA \amalg \EA')$.

Note that a $\Lambda$-category and an additive $\Lambda$-category respectively is in
particular a $\IZ$-category and an additive category respectively thanks to the canonical
ring homomorphism $\IZ \to \Lambda$.

One can assign to an additive category $\cala$ its non-connective $K$-theory spectrum
$\bfK(\cala)$. We denote $K_n(\cala) = \pi_n(\bfK(\cala))$ for $n \in \IZ$.

Given a ring $R$, define $\underline{R}$ to be the $\IZ$-category, which has precisely one
object $\ast_R$ and whose $\IZ$-module of endomorphisms is $R$. Composition is given by
the multiplication in $R$.

For a ring $R$ let $K_n(R)$ for $n \in \IZ$ be its $n$-algebraic $K$-group, which can be
defined for instance as the (non-connective) $K$-theory of the exact category of finitely
generated projective $R$-module.  It can be identified with $K_n(\underline{R}_{\oplus})$.
Note that we do not have to pass to the idempotent completion of $\underline{R}_{\oplus}$,
as we are working with non-connective $K$-theory.

All these classical notions are summarized with references to the relevant papers
in~\cite[Section~2 and~3]{Bartels-Lueck(2020additive)}.

We fix some conventions concerning matrices of morphisms. For an object $A$ in $\cala$
we denote by $A^m$ the direct sum $\bigoplus_{i=1}^m A$.  For  two finite direct
         sums $\bigoplus_{i= 1}^m A_i$ and $\bigoplus_{j= 1}^n B_j$ a morphism
         $U \colon \bigoplus_{i= 1}^m A_i \to \bigoplus_{j= 1}^n B_j$ is the same as a
         $(m,n)$-matrix
         \[
           U =
           \begin{pmatrix}
             u_{1,1} & u_{2,1} & u_{3,1} & \cdots & u_{m-2,1} & u_{m-1,1} & u_{m,1}
             \\
             u_{1,2} & u_{2,2} & u_{3,2} & \cdots & u_{m-2,2} & u_{m-1,2} & u_{m,1}
             \\
             u_{1,3} & u_{2,3} & u_{3,3} & \cdots & u_{m-2,3} & u_{m-1,3} & u_{m,3}
             \\
             \vdots &  \vdots &  \vdots  & \ddots  &\vdots &  \vdots &  \vdots
             \\
             u_{1,n-1} & u_{2,n-1} & u_{3,n-1} & \cdots & u_{m-2,n-1} & u_{m-1,n-1} & u_{m,n-1}
             \\
             u_{1,n} & u_{2,n} & u_{3,n} & \cdots & u_{m-2,n} & u_{m-1,n} & u_{m,n}
           \end{pmatrix}
         \]
         of morphisms $u_{i,j} \colon A_i \to B_j$. One may think of an element in
         $\bigoplus_{i= 1}^m A_i$ as a $(1,m)$-matrix
         \[
           x := \begin{pmatrix}
             x_1 \\ x_2 \\ \vdots \\ x_m
           \end{pmatrix}
         \]
         and thinks of the image of this element  under $U$ as the $(1,n)$-matrix
           \[U(x)  =
           \begin{pmatrix}
             \sum_{l = 1}^m  u_{l,1}(x_l)
               \\
              \sum_{l = 1}^m  u_{l,2}(x_l) 
                \\
                \vdots
                \\
                \sum_{l = 1}^m  u_{l,n}(x_l)
                \end{pmatrix}.
              \]
              Note  that $m$ is the number of columns and $n$ is the number of rows with these conventions. 
         If   $V \colon \bigoplus_{j= 1}^n B_j \to \bigoplus_{k= 1}^o C_k$ is another
     morphism given by the $(n,o)$-matrix $V$, then the composite
     $V \circ U \colon \bigoplus_{i= 1}^m A_i \to \bigoplus_{k= 1}^o C_k$ is given by
     the $(m,o)$ matrix $W$ whose $(i,k)$-entry is
     \[
       w_{i,k} = \sum_{j = 1}^n v_{j,k} \circ u_{i,j}.
      \]
      So  one could think of $W$ as a kind of   product of matrices $V \cdot U$.


\typeout{------------- Section 3: Frobenius and Verschiebungs operators
  ----------------------}

\section{Frobenius and Verschiebungs operators}%
\label{sec:Frobenius_and_Verschiebungs_operators}

Let $\cala$ be an additive category and $\Phi$ be an automorphism of $\cala$.

\begin{definition}[Nilpotent morphisms and Nil-categories]%
  \label{def:Nilpotent_morphisms_and_Nil-categories}\
  
  \begin{enumerate}
  \item A morphism $\varphi\colon \Phi(A)\to A$ of $\cala$ is called \emph{$\Phi$-nilpotent},
    if for some $n \ge 1 $ the $n$-fold composite
    \[
      \varphi^{(n)}:=\varphi \circ \Phi(\varphi) \circ \cdots \circ \Phi^{n-1}(\varphi) \colon
      \Phi^n(A)\to A.
    \]
    is trivial;
 
  \item The category $\Nil(\cala, \Phi)$ has as objects pairs $(A, \varphi)$, where
    $\varphi\colon \Phi(A)\to A$ is a $\Phi$-nilpotent morphism in $\cala$. A morphism from
    $(A, \varphi)$ to $(A', \varphi')$ is a morphism $u\colon A\to A'$ in $\cala$ such that the diagram 
    \[
      \xymatrix{{\Phi(A)} \ar[r]^-{\varphi} \ar[d]_{\Phi(u)} & A \ar[d]^u\\
        {\Phi(A')} \ar[r]_-{\varphi'} & A'}
    \]
    is commutative.
  \end{enumerate}
\end{definition}

The category $\Nil(\cala, \Phi)$ inherits the structure of an exact category from $\cala$,
a sequence in $\Nil(\cala,\Phi)$ is declared to be exact if the underlying sequence in
$\cala$ is (split) exact.

Next we define for $k \in \{1,2, \ldots\}$ the \emph{Verschiebung operator} $V_k$ and the
\emph{Frobenius operator } $F_k$
\begin{eqnarray}
  V_k \colon \Nil(\cala, \Phi^k) & \to & \Nil(\cala, \Phi);
                                         \label{Verschiebung_V_k}
  \\
  F_k \colon \Nil(\cala, \Phi) & \to & \Nil(\cala, \Phi^k).
                                       \label{Frobenius_F_k}
\end{eqnarray}

Given an object $(A, \varphi)$ in $\Nil(\cala, \Phi^k)$, define $V_k(A, \varphi)$ to be the
object in $\Nil(\cala, \Phi)$ that is given by the object
$\bigoplus_{i =0}^{k-1} \Phi^i(A)$ in $\cala$ together with the $\Phi$-nilpotent morphism
\begin{multline*}
  \begin{pmatrix}
    0 & 0 & 0 & 0 & 0 & \cdots & 0 & 0 & \varphi
    \\
    \id_{\Phi(A)} & 0 & 0 & 0 & 0 & \cdots & 0 & 0 & 0
    \\
    0 & \id_{\Phi^2(A)} & 0 & 0 & 0 &\cdots & 0 & 0 & 0
    \\
    0 & 0 &\id_{\Phi^3(A)} & 0 & 0 & \cdots & 0 & 0 & 0
    \\
    \vdots & \vdots & \vdots &\vdots & \vdots & \vdots & \vdots & \vdots & \vdots
    \\
    0 & 0 & 0 & 0 & 0 & \cdots & \id_{\Phi^{k-2}(A)} & 0 & 0
    \\
    0 & 0 & 0 & 0 & 0 & \cdots & 0 & \id_{\Phi^{k-1}(A)} & 0
  \end{pmatrix}
  \colon
  \\
  \Phi\left(\bigoplus_{i =0}^{k-1} \Phi^i(A)\right) =\bigoplus_{i =1}^k \Phi^i(A) \to
  \bigoplus_{i =0}^{k-1} \Phi^i(A).
\end{multline*}
A morphisms $u \colon (A, \varphi) \to (A', \varphi')$ in $\Nil(\cala, \Phi^k)$ given by a
morphism $u \colon A \to A'$ in $\cala$ is sent to the morphism
$V_k(A,\varphi) \to V_k(A',\varphi')$ in $\Nil(\cala, \Phi)$ given by the morphism
$\bigoplus_{i =0}^{k-1} \Phi^i(u) \colon \bigoplus_{i =0}^{k-1} \Phi^i(A) \to \bigoplus_{i
  =0}^{k-1} \Phi^i(A')$.

Given an object $(A, \varphi)$ in $\Nil(\cala, \Phi)$, define $F_k(A, \varphi)$ to be
$(A, \varphi^{(k)})$, see Definition~\ref{def:Nilpotent_morphisms_and_Nil-categories}. A
morphism $u \colon (A,\varphi) \to (A',\varphi')$ in $\Nil(\cala, \Phi)$ given by a morphism
$u \colon A \to A'$ in $\cala$ is sent to the morphism
$u \colon (A,\varphi^{(k)}) \to (A',\varphi'^{(k)})$ in $\Nil(\cala, \Phi^k)$ given by $u$
again.

The elementary proof of the next lemma is left to the reader.

\begin{lemma}\label{lem:F_k_circ_V_k} The composite
  \[F_k \circ V_k \colon \Nil(\cala, \Phi^k) \to \Nil(\cala, \Phi^k)
  \]
  sends an object $(A,\varphi)$ to the object given by
  \[
    \bigoplus_{i = 0}^{k-1}\Phi^i(\varphi)\colon \Phi^k\left(\bigoplus_{i= 0}^{k-1}
      \Phi^i(A)\right) = \bigoplus_{i= 0}^{k-1} \Phi^{i+k}(A) \to \bigoplus_{i= 0}^{k-1}
    \Phi^i(A).
  \]
  It sends a morphism $u \colon (A,\varphi) \to (A',\varphi')$ in $\Nil(\cala, \Phi^k)$ given by
  a morphism $u \colon A \to A'$ in $\cala$ to the morphism
  $F_k \circ V_k(A,\varphi) \to F_k \circ V_k(A',\varphi')$ in $\Nil(\cala, \Phi^k)$ given by the
  morphism
  $\bigoplus_{i= 0}^{k-1} \Phi^i(u) \colon \bigoplus_{i= 0}^{k-1} \Phi^i(A) \to
  \bigoplus_{i= 0}^{k-1} \Phi^i(A')$ in $\cala$.
\end{lemma}


\typeout{------------- Section 4: Frobenius and Verschiebungs operators and induction and
  restriction ----------------------}

\section{Frobenius and Verschiebungs operators and induction and restriction}%
\label{sec:Frobenius_and_Verschiebungs_operators_and_induction_and_restriction}

Let $\cala$ be an additive category with an action $\rho \colon G \to \aut(\cala)$ of the
(discrete) group $G$ by automorphisms of additive categories. Then we obtain a new
additive category
\begin{equation}
  \cala_{\rho}[G]
  \label{cala_rho(G)}
\end{equation}
as follows. The set of objects of $\cala_{\rho}[G]$ is the set of objects of $\cala$. A
morphism $f \colon A \to A'$ in $\cala_{\rho}[G]$ is a finite formal
$\sum_{g \in G} (f_g \colon gA \to A') \cdot g$, where $f_g \colon gA \to A'$ is a morphism
in $\cala$ from $gA $ to $A'$ and finite means that for only finitely many elements $g$ in
$G$ the morphism $f_g$ is different from the zero-homomorphism. If $f' \colon A' \to A''$
is a morphism in $\cala_{\rho}[G]$ given by the finite formal sum
$\sum_{g' \in G} (f'_{g'} \colon g'A' \to A'') \cdot g'$, then define their composite
$f' \circ f \colon A \to A''$ by the finite formal sum
\[
  f' \circ f = \sum_{g'' \in G} \; \sum_{\substack{g,g' \in G,\\g'' = g'g}} \; (f_{g'}
  \circ g'f_{g} \colon g''A = g'gA \to A'') \cdot g''.
\]
  
If $R$ is a unital ring coming with a $G$-action $\rho_R \colon G \to \aut(R)$ and we take
$\cala$ to be the category of finitely generated free $R$-modules with the obvious
$G$-action $\rho \colon G \to \aut(\cala)$ coming from $\rho_R$ by induction, then
$\cala_{\rho}[G]$ is equivalent to the additive category of finitely generated free modules
over the twisted group ring $R_{\rho}[G]$.
  
If $\Phi \colon \cala \xrightarrow{\cong} \cala$ is an automorphism of an additive
category $\cala$, then we define
\begin{equation}
  \cala_{\Phi}[\IZ] = \cala_{\rho_{\Phi}}[\IZ]
  \label{cal_Phi(Z)}
\end{equation}
for the $\IZ$-action $\rho_{\Phi} \colon \IZ \to \aut(\cala), \; n \mapsto \Phi^n$.

Let $i_k \colon \IZ \to \IZ$ be the group homomorphism given by multiplication with
$k$. Next we define functors of additive categories
\begin{eqnarray}
  (i_k)_* \colon \cala_{\Phi^k}[\IZ]
  & \to &
          \cala_{\Phi}[\IZ];
          \label{(i_k)_ast}
  \\
  i_k^* \colon \cala_{\Phi}[\IZ]
  & \to &
          \cala_{\Phi^k}[\IZ].
          \label{(i_k)_upper_ast}
\end{eqnarray}
The functor $(i_k)_*$ sends an object $A$ in $\cala_{\Phi^k}[\IZ]$, which is given by an
object $A$ in $\cala$, to the object in $\cala_{\Phi}[\IZ]$ given by $A$ again. Consider a
morphism
\[
  f = \sum_{l \in \IZ} \bigl(f_l \colon \Phi^{kl}(A) \to A'\bigr) \cdot t^l \colon A \to
  A'
\]
in $\cala_{\Phi^k}[\IZ]$.  It is sent by $( i_k)_*$ to the morphism
\[
  (i_k)_*(f) := \sum_{l \in \IZ} \bigl(f_l \colon \Phi^{kl}(A) \to A'\bigr) \cdot t^{lk}
  \colon A \to A'
\]
in $\cala_{\Phi}[\IZ]$.

The functor $i_k^*$ sends a object $A$ in $\cala_{\Phi}[\IZ]$, which is given by an object
$A$ in $\cala$, to the object $i_k^*(A)$ in $\cala_{\Phi^k}[\IZ]$ given by the object
$\bigoplus_{i = 0}^{k-1} \Phi^i(A)$ in $\cala$.  Consider a morphism
$ f = \sum_{l \in \IZ} \bigl(f_l \colon \Phi^{l}(A) \to A'\bigr) \cdot t^l \colon A \to
A'$ in $\cala_{\Phi}[\IZ]$.  It is sent by $i_k^*$ to the morphism
\[
  i_k^*(f) \colon \bigoplus_{i = 0}^{k-1} \Phi^i(A) \to \bigoplus_{j = 0}^{k-1} \Phi^j(A')
\]
in $\cala_{\Phi}[\IZ]$ defined as follows. By additivity we have only to specify
$i_k^*(f_l \cdot t^l)$. For this purpose we have to define for every
$i,j \in \{0,1,\ldots , (k-1)\}$ a morphisms
$i_k^*(f_l \cdot t^l)_{i,j} \colon \Phi^i(A) \to \Phi^j(A')$ in $\cala_{\Phi^k}[\IZ]$. It
is given by $\bigl(\Phi^j(f_l) \colon \Phi^{i+mk}(A) \to \Phi^j(A')\bigr) \cdot t^m$ if
there exists an integer $m$ with $i+mk +l =j$, and by zero otherwise.

If $\cala$ is given by $\underline{R}$ for a ring $R$ coming with a ring automorphism
$\Phi \colon R \xrightarrow{\cong} R$, $(i_k)_*$ and $i_k^*$ corresponds to induction and
restriction with respect to the change of ring homomorphism of twisted group rings
$R_{\Phi^k}[\IZ] \to R_{\Phi}[\IZ]$ associated to $i_k$.

In the sequel we use the notation of~\cite{Lueck-Steimle(2016BHS)}. We get by taking
homotopy groups from~\cite[Theorem~0.1]{Lueck-Steimle(2016BHS)} for $n \in \IZ$ an
isomorphism
\begin{equation}
  a_n \oplus c_n^+ \oplus c_n^- \colon \pi_n(\bfT_{\bfK(\Phi^{-1})}) 
  \oplus \overline{K}_{n-1}(\Nil(\cala,\Phi)) \oplus  \overline{K}_{n-1}(\Nil(\cala,\Phi))
  \xrightarrow{\cong} K_n(\cala_{\Phi}[\IZ]).
  \label{twisted_BHS_on_homotopy_groups}
\end{equation}
Here $\bfT_{\bfK(\Phi^{-1})}$ is the mapping torus of the map induced on
non-connective $K$-theory spectra
$\bfK(\Phi^{-1}) \colon \bfK(\cala) \to \bfK(\cala)$. There is a long exact Wang
sequence
\begin{multline*}
  \cdots \xrightarrow{\partial_{n+1}} K_n(\cala) \xrightarrow{K_n(\Phi) - \id} K_n(\cala)
  \xrightarrow{\pi_n(\bfj)} \pi_n(\bfT_{\bfK(\Phi^{-1})})
  \\
  \xrightarrow{\partial_n} K_{n-1}(\cala) \xrightarrow{K_{n-1}(\Phi) - \id} K_{n-1}(\cala)
  \xrightarrow{\pi_{n-1}(\bfj)} \pi_n(\bfT_{\bfK(\Phi^{-1})}) \xrightarrow{\partial_{n-1}}
  \cdots
\end{multline*}
where $\bfj \colon \bfK(\cala) \to \bfT_{\bfK(\Phi^{-1})} $ is the inclusion. If
$\Phi = \id_{\cala}$, this boils down to an isomorphism
\[
  \pi_n(\bfT_{\bfK(\Phi^{-1})}) \cong K_n(\cala) \oplus K_{n-1}(\cala).
\]
We define   $K_{n}(\Nil(A,\Phi^k)) = \pi_n(\bfKNilinfty(\cala,\Phi^k))$ for $n \in \IZ$,
where the non-connective $K$-theory spectrum $ \bfKNilinfty(\cala,\Phi)$ is 
constructed in~\cite[Remark~6.3 and Lemma~6.5]{Lueck-Steimle(2014delooping)}.
It is likely but we have no detailed proof that
the group $K_n(\Nil(\cala,\Phi))$ can be
identified for $n \le -1$ with the $K$-groups associated to the exact category
$\Nil(\cala,\Phi)$ in the sense of Schlichting~\cite{Schlichting(2006)}
see~\cite[Remark~6.11]{Lueck-Steimle(2014delooping)}.  Fortunately, 
we do not need this identification for our purposes.
  There is the inclusion functor $I \colon \cala \to \Nil(\cala,\Phi)$
that sends an object $A$ to the object $(A,0)$ and the projection functor
$P \colon \Nil(\cala,\Phi) \to \cala$ that sends an object $(A,\varphi)$ to $A$. Obviously
$P \circ I = \id_{\cala}$.

\begin{notation}\label{not:overline(K)(Nil))}
We define $\overline{K}_n(\Nil(A,\Phi))$ to be the cokernel of
the split injective homomorphism $K_n(I) \colon K_n(\cala) \to K_n(\Nil(A,\Phi))$ for
$n \in \IZ$.
\end{notation}

Let $T_{\bfK(\Phi^{-1}),k}$ be the $k$-fold mapping torus of
$\bfK(\Phi) \colon \bfK(\cala) \to \bfK(\cala)$, which is a $\IZ/k$-spectrum. There
is a $k$-fold covering
$\bfp_k \colon T_{\bfK(\Phi^{-1}),k} \to T_{\bfK(\Phi^{-1})}$ and a homotopy
equivalence $\bff \colon T_{\bfK(\Phi^{-1}),k} \to \bfT_{\bfK(\Phi^{-k})}$. These
correspond to the following construction on the level of spaces for a map
$\Phi \colon X \to X$.  Namely, $p_k \colon T_{\Phi,k} \to T_{\Phi}$ is the $k$-sheeted
covering obtained by the pull back of the $k$-sheeted covering $S^1 \to S^1$ sending $z$
to $z^k$ with the canonical map $T_{\Phi} \to S^1$.  Explicitly $T_{\Phi,k}$ is obtained
from $\coprod_{i = 1}^{k} X \times [i-1,i]$ by identifying $(x,i) \in X \times [i-1,i]$
with $(\Phi(x),i)$ in $X \times [i,i+1]$ for $i = 1,2 \ldots k-1$ and
$(x,k) \in X \times [k-1,k]$ with $(\Phi(x),0)$ in $X \times [0,1]$. Obviously
$T_{\Phi,1} = T_{\Phi}$.  The map $p_k \colon T_{\Phi,k} \to T_{\Phi}$ sends the class of
$(x, j) \in X \times [i-1,i]$ to the class of $(x,j-i-1) \in X \times [0,1]$ for
$i = 1,2, \ldots, k$. The homotopy equivalence $f \colon T_{\Phi,k} \to T_{\Phi^k}$ sends
the class of $(x, j) \in X \times [i-1,i]$ to the class of
$(\Phi^{k-i}(x),\frac{j}{k}) \in X \times [0,1]$ for $i = 1,2, \ldots, k$. On homotopy
groups we obtain an isomorphism
\begin{equation}
  \pi_n(\bff) \colon \pi_n(\bfT_{\Phi^{-1},k}) \xrightarrow{\cong} \pi_n(T_{\Phi^{-k}})
  \label{pi_n(bbf)}
\end{equation}
and a homomorphism induced by $\bfp_k$
\begin{equation}
  \pi_n(\bfp_k) \colon \pi_n(\bfT_{\Phi^{-1},k}) \xrightarrow{\cong} \pi_n(T_{\Phi^{-1}}).
  \label{pi_n(bfp_k)}
\end{equation}
Since $\bfp_k$ is a $k$-sheeted covering, there is a transfer homomorphism
\begin{equation}
  \trf_n(\bfp_k) \colon \pi_n(T_{\Phi^{-1}}) \to \pi_n(\bfT_{\Phi^{-1},k}).
  \label{trf_n(bfp_k)}
\end{equation}
Note that the Frobenius and the Verschiebungs operator are functors of exact categories
and hence induces homomorphism
\begin{eqnarray}
  K_n(V_k) \colon K_n(\Nil(\cala,\Phi^k) )& \to & K_n(\Nil(\cala,\Phi));
                                                  \label{K_n(V_k)}
  \\
  K_n(F_k) \colon K_n(\Nil(\cala,\Phi)) & \to & K_n(\Nil(\cala,\Phi^k)).
                                                \label{K_n(F_k)}
\end{eqnarray}
Since $K_n(F_k) \circ K_n(I) = K_n(I)$ and $K_n(V_k) \circ K_n(I) = K_n(I)$ holds, they
induce homomorphisms
\begin{eqnarray}
  \overline{K}_n(V_k) \colon \overline{K}_n(\Nil(\cala,\Phi^k) )& \to & \overline{K}_n(\Nil(\cala,\Phi));
                                                                          \label{widetilde(K)_n(V_k)}
  \\
  \overline{K}_n(F_k) \colon \overline{K}_n(\Nil(\cala,\Phi)) & \to & \overline{K}_n(\Nil(\cala,\Phi^k)).
                                                                        \label{widetilde(K)_n(F_k)}
\end{eqnarray}
  
The functors $(i_k)_*$ and $i_k^*$ are functors of additive categories and induce
homomorphisms
\begin{eqnarray}
  K_n((i_k)_*) \colon  K_n(\cala_{\Phi^k}[\IZ]) & \to & K_n(\cala_{\Phi}[\IZ]);
                                                        \label{K_n((i_k)_ast)}
  \\
  K_n(i_k^*) \colon  K_n(\cala_{\Phi}[\IZ]) & \to & K_n(\cala_{\Phi^k}[\IZ]).
                                                    \label{K_n(i_k_upper_ast)}
\end{eqnarray}
  
The main result of this section is
\begin{theorem}\label{F_k_and_V_k_versus_I_k_ast_and_i_k_upper_ast}
  Let $k \ge 1$ be a natural number and $\Phi \colon \cala \xrightarrow{\cong} \cala$ be
  an automorphism of an additive category $\cala$. Then:
  \begin{enumerate}
  \item\label{F_k_and_V_k_versus_I_k_ast_and_i_k_upper_ast:induction} The following
    diagram commutes for $n \in \IZ$
    \[
      \xymatrix@!C=16em{\pi_n(\bfT_{\bfK(\Phi^{-k})})\oplus
        \overline{K}_{n-1}(\Nil(\cala,\Phi^k)) \oplus \overline{K}_{n-1}
        (\Nil(\cala,\Phi^k)) \ar[r]^-{a_n\oplus c_n^+ \oplus c_n^-}_-{\cong}
        \ar[d]_{(\pi_n(\bfp_k) \circ \pi_n(\bff)^{-1}) \oplus \overline{K}_{n-1}(V_k)
          \oplus \overline{K}_{n-1}(V_k)}
        &
        K_n(\cala_{\Phi^k}[\IZ])  \ar[d]^{K_n((i_k)_*)}
        \\
        \pi_n(\bfT_{\bfK(\Phi^{-1})}) \oplus \overline{K}_{n-1}(\Nil(\cala,\Phi))
        \oplus \overline{K}_{n-1}(\Nil(\cala,\Phi)) \ar[r]_-{a_n\oplus c_n^+ \oplus c_n^-
        }^-{\cong}
        &
        K_n(\cala_{\Phi}[\IZ]) }
    \]
    where the upper horizontal isomorphisms is the one defined
    in~\eqref{twisted_BHS_on_homotopy_groups} for $\Phi^k$, the lower horizontal
    isomorphisms is the one defined in~\eqref{twisted_BHS_on_homotopy_groups} for $\Phi$,
    and the vertical arrows have been defined
    in~\eqref{pi_n(bbf)},~\eqref{pi_n(bfp_k)},~\eqref{widetilde(K)_n(V_k)},
    and~\eqref{K_n((i_k)_ast)};
  \item\label{F_k_and_V_k_versus_I_k_ast_and_i_k_upper_ast:restriction} The following
    diagram commutes for $n \in \IZ$
    \[
      \xymatrix@!C=16em{\pi_n(\bfT_{\bfK(\Phi^{-1})})\oplus
        \overline{K}_{n-1}(\Nil(\cala,\Phi)) \oplus \overline{K}_{n-1}(\Nil(\cala,\Phi))
        \ar[r]^-{a_n\oplus c_n^+ \oplus c_n^-}_-{\cong} \ar[d]_{(\pi_n(\bff) \circ
          \trf_n(\bfp_k)) \oplus \overline{K}_{n-1}(F_k) \oplus \overline{K}_{n-1}(F_k)}
        & K_n(\cala_{\Phi}[\IZ]) \ar[d]^{K_n(i_k^*)}
        \\
        \pi_n(\bfT_{\bfK(\Phi^{-k})}) \oplus \overline{K}_{n-1}(\Nil(\cala,\Phi^k))
        \oplus \overline{K}_{n-1}(\Nil(\cala,\Phi^k)) \ar[r]_-{a_n\oplus c_n^+ \oplus
          c_n^- }^-{\cong} & K_n(\cala_{\Phi^k}[\IZ]) }
    \]
    where the upper isomorphisms is the one defined
    in~\eqref{twisted_BHS_on_homotopy_groups} for $\Phi$, the lower isomorphisms is the
    one defined in~\eqref{twisted_BHS_on_homotopy_groups} for $\Phi^k$, and the horizontal
    arrows have been defined
    in~\eqref{pi_n(bbf)},~\eqref{trf_n(bfp_k)},~\eqref{widetilde(K)_n(F_k)},
    and~\eqref{K_n(i_k_upper_ast)}.
  \end{enumerate}
\end{theorem}
\begin{proof}~\ref{F_k_and_V_k_versus_I_k_ast_and_i_k_upper_ast:induction} We give only an
  outline of the proof and leave some details to the reader.  In the sequel we use the
  notation of~\cite{Lueck-Steimle(2016BHS)}.
  
  We obtain from~\cite[Theorem~0.1~(i)]{Lueck-Steimle(2016BHS)} for $n \in \IZ$ an
  isomorphism
  \begin{equation}
    a_n \oplus b_n^+ \oplus b_n^- \colon
    \pi_n(\bfT_{\bfK(\Phi^{-1})}) \oplus N\!K_n(\cala_{\Phi}[t])  \oplus  N\!K_n(\cala_{\Phi}[t^{-1}])
    \xrightarrow{\cong} K_n(\cala_{\Phi}[\IZ]),
    \label{twisted_BHS_on_homotopy_groups_for_NK}
  \end{equation}
  where $N\!K_n(\cala_{\Phi}[t^{\pm}])$ is the kernel of the map
  $K_n(\cala_{\Phi}[t^{\pm}]) \to K_n(\cala)$ coming from the functor
  $\cala_{\Phi}[t^{\pm}] \to \cala$ given taking the coefficient of $t^0$. One easily
  checks that the functors $(i_k)_*$ and $i_k^*$ of~\eqref{(i_k)_ast}
  and~\eqref{(i_k)_upper_ast} induce homomorphisms
  \begin{eqnarray}
    N\!K_n((i_k)_*)_{\pm} \colon N\!K_n(\cala_{\Phi^k}[t^{\pm 1}]) & \to & N\!K_n(\cala_{\Phi}[t^{\pm 1}]);
                                                                           \label{NK_n((i_k)_ast}
    \\
    N\!K_n(i_k^*)_{\pm} \colon N\!K_n(\cala_{\Phi}[t^{\pm 1}]) & \to & N\!K_n(\cala_{\Phi^k}[t^{\pm 1}]).
                                                                       \label{NK_n(i_k)_upper_ast}
  \end{eqnarray}
  Then following diagram commutes for $n \in \IZ$
  \[
    \xymatrix@!C=15em{\pi_n(\bfT_{\bfK(\Phi^{-k})})\oplus N\!K_n(\cala_{\Phi^k}[t])
      \oplus N\!K_n(\cala_{\Phi^k}[t^{-1}]) \ar[r]^-{a_n\oplus b_n^+ \oplus
        b_n^-}_-{\cong} \ar[d]_{(\pi_n(\bfp_k) \circ \pi_n(\bff)^{-1}) \oplus
        N\!K_n((i_k)_*)_{+} \oplus N\!K_n((i_k)_*)_{-}} & K_n(\cala_{\Phi^k}[\IZ])
      \ar[d]^{K_n((i_k)_*)}
      \\
      \pi_n(\bfT_{\bfK(\Phi^{-1})}) \oplus N\!K_n(\cala_{\Phi}[t]) \oplus
      N\!K_n(\cala_{\Phi}[t^{-1}]) \ar[r]_-{a_n\oplus b_n^+ \oplus b_n^- }^-{\cong} &
      K_n(\cala_{\Phi}[\IZ]) }
  \]
  where the upper horizontal arrow is the
  isomorphism~\eqref{twisted_BHS_on_homotopy_groups_for_NK} for $\Phi^k$, the lower
  horizontal arrow is the isomorphism~\eqref{twisted_BHS_on_homotopy_groups_for_NK} for
  $\Phi$ and the vertical homomorphism have been defined
  in~\eqref{pi_n(bbf)},~\eqref{pi_n(bfp_k)},~\eqref{K_n((i_k)_ast)}
  and~\eqref{NK_n((i_k)_ast}. The proof of commutativity for the terms
  $N\!K_n(\cala_{\Phi^k}[t])$ and $N\!K_n(\cala_{\Phi^k}[t^{-1}])$ is obvious, the one for
  the term $\pi_n(\bfT_{\bfK(\Phi^{-k})})$ is left to the reader.

  We obtain from~\cite[Theorem~0.1~(ii)]{Lueck-Steimle(2016BHS)} for $n \in \IZ$ an
  isomorphism
  \begin{eqnarray}
    \alpha(\Phi,\pm)_n \colon  K_{n-1}(\cala) \oplus N\!K_n(\cala_{\Phi}[t^{\pm}])
    & \xrightarrow{\cong} &
     K_{n-1}(\Nil(\cala,\Phi)).
     \label{iso_alpha(Phi)_pm}
  \end{eqnarray}
  and hence an isomorphism
  \begin{eqnarray}
    \widetilde{\alpha}(\Phi,\pm)_n \colon  N\!K_n(\cala_{\Phi}[t^{\pm}])
    & \xrightarrow{\cong} &
   \overline{K}_{n-1}(\Nil(\cala,\Phi)).
     \label{iso_widetilde(alpha)(Phi)_pm}
  \end{eqnarray}
  Therefore it remains to show that the diagram
  \begin{equation}
    \xymatrix@!C=10em{N\!K_n(\cala_{\Phi^k}[t^{\pm}]) \ar[r]_-{\cong}^-{\widetilde{\alpha}(\Phi^k,\pm)_n}
      \ar[d]_{N\!K_n((i_k)_*)_{\pm}}
      &
      \overline{K}_{n-1}(\Nil(\cala,\Phi^k)) \ar[d]^{\overline{K}_{n-1}(V_k)}
      \\
      N\!K_n(\cala_{\Phi^k}[t^{\pm}]) \ar[r]^-{\cong}_-{\widetilde{\alpha}(\Phi,\pm)_n}
      &
      \overline{K}_{n-1}(\Nil(\cala,\Phi))
    }
    \label{diagram_comparing_Nil_and_NK_for_Induction}
  \end{equation}
  commutes for $n \in \IZ$.  Fix a natural number $m$. Then the
  diagram~\eqref{diagram_comparing_Nil_and_NK_for_Induction}
  is a retract of the corresponding diagram
  \[
    \xymatrix@!C=17em{N\!K_{n+m}(\cala[\IZ^m]_{\Phi[\IZ^m]^k}[t^{\pm}])
      \ar[r]_-{\cong}^-{\widetilde{\alpha}(\Phi[\IZ^m]^k,\pm)_{n+m}}
      \ar[d]_{N\!K_{n+m}((i_k)_*)_{\pm}}
      &
      \overline{K}_{n+m-1}(\Nil(\cala[\IZ^m],\Phi[\IZ^m]^k)) \ar[d]^{\overline{K}_{n+ m -1}(V_k)}
      \\
      N\!K_{n+m}(\cala[\IZ^m]_{\Phi[\IZ^m]^k}[t^{\pm}])
      \ar[r]^-{\cong}_-{\widetilde{\alpha}(\Phi[\IZ^m],\pm)_{n+m}}
      &
      \overline{K}_{n+m-1}(\Nil(\cala[\IZ^m],\Phi)[\IZ^m])
    }
    \]
    where we have replaced $\cala$ by $\cala[\IZ^m]$ and $\Phi$ by $\Phi[\IZ^m]$,
    see~\cite[Lemma~2.2, Theorem~6.2, Remark~6.3]{Lueck-Steimle(2014delooping)}.  Hence we
    can assume without loss of generality that $n \ge 1 $ when proving the commutativity
    of~\eqref{diagram_comparing_Nil_and_NK_for_Induction}, since our result shall be true
    for any additive category $\cala$.  So we can use the connective $K$-theory spectrum
    when dealing with the commutativity
    of~\eqref{diagram_comparing_Nil_and_NK_for_Induction}.  By
    inspecting~\cite{Lueck-Steimle(2016BHS)} one sees that this boils down to show that
    the following diagram commutes for $n \ge 1$,
  \begin{equation}
    \xymatrix@!C=10em{K_{n}(\Nil(A,\Phi^k)) \ar[r]^-{K_n(\chi_{\Phi^k})}  \ar[d]_{K_n(V_k)}
      &
      K_n(\Chcat(\cala_{\Phi^k}[t^{-1}])^w) \ar[d]^{\bfK(\Chcat((i_k)_*)^w)}
      \\
      K_{n}(\Nil(A,\Phi)) \ar[r]_-{K_n(\chi_{\Phi})}
      &
      K_n(\Chcat(\cala_{\Phi}[t^{-1}])^w) 
    }
    \label{diagram_comparing_Nil_and_Ch_for_induction}
  \end{equation}
  where the upper and the lower homomorphism are induced by the functors
  \begin{eqnarray}
    \chi_{\Phi^k} \colon \Nil(A,\Phi^k) & \to & \Chcat(\cala_{\Phi^k}[t^{-1}])^w;
                                                \label{chi_(Phi_upper_k)}
    \\
    \chi_{\Phi} \colon \Nil(A,\Phi) & \to & \Chcat(\cala_{\Phi}[t^{-1}])^w,
                                            \label{chi_(Phi))}
  \end{eqnarray}
  introduced in~\cite[Section~8]{Lueck-Steimle(2016BHS)}, the left horizontal arrow has
  been introduced in~\eqref{K_n(V_k)} and the right vertical arrow is induced by the
  functor $(i_k)_*$ of~\eqref{(i_k)_ast}.  Recall that $\Chcat(\cala_\Phi[t^{-1}])^w$ is
  the category of bounded chain complexes over
 $\cala_\Phi[t^{-1}]$, which are contractible as chain complexes over
 $\cala_\Phi[t,t^{-1}]$, and that the lower horizontal arrow sends an object
  $(A,\varphi)$ in $\Nil(\cala,\Phi)$ to the $\cala_{\Phi}[t^{-1}]$-chain complex
  concentrated in dimensions $0$ and $1$, whose first differential is
  $\id_A \cdot t^{-1} - \varphi \cdot t^0 \colon \Phi(A) \to A$. 

  Next we explain the key ingredients in the proof of the commutativity
  of~\eqref{diagram_comparing_Nil_and_Ch_for_induction} and leave it to the reader to
  figure out the routine to fill in the details based on standard fact about connective $K$-theory
  such as the Additivity Theorem for Waldhausen categories.

  Consider an object $\varphi \colon \Phi^k(A) \to A$ in $\Nil(\cala,\Phi^k)$ given by a
  nilpotent endomorphism $\varphi \colon \Phi^k(A) \to A$ in $\cala$.  We have defined its
  Verschiebung $V_k(\varphi)$ as an object in $\Nil(A,\Phi)$ given by a specific nilpotent
  endomorphism
  \[
    V_k(\varphi) \colon \Phi\bigl(\bigoplus_{j = 0}^{k-1} \Phi^j(A)\bigr) =\bigoplus_{i =
      1}^k \Phi^i(A) \to \bigoplus_{j = 0}^{k-1} \Phi^j(A)
  \]
  in $\cala$, see~\eqref{Verschiebung_V_k}. To it we can assign the morphism
  \[
    \id_{\bigoplus_{j = 0}^{k-1} \Phi^j(A)} \cdot t^{-1} - V_k(\varphi) \cdot t^0 \colon
    \Phi\bigl(\bigoplus_{j = 0}^{k-1} \Phi^j(A)\bigr) =\bigoplus_{i = 1}^k \Phi^i(A) \to
    \bigoplus_{j = 0}^{k-1} \Phi^j(A)
  \]
  in $\cala_{\Phi}[\IZ]$. It is given by the following $(k,k)$-matrix
  \[
    \begin{pmatrix}
      \id_{A} \cdot t^{-1} & 0 & 0 & \cdots & 0 & 0 & -\varphi \cdot t^0
      \\
      -\id_{\Phi(A)} \cdot t^0 & \id_{\Phi(A)} \cdot t^{-1} & 0 & \cdots & 0 & 0 & 0
      \\
      0 & -\id_{\Phi^2(A)} \cdot t^0 & \id_{\Phi^2(A)} \cdot t^{-1} & \cdots & 0 & 0 & 0
      \\
      \vdots & \vdots & \vdots & \vdots & \vdots & \vdots & \vdots
      \\
      0 & 0 & 0 & \cdots & -\id_{\Phi^{k-3}(A)} \cdot t^{-1} & 0 & 0
      \\
      0 & 0 & 0 & \cdots & -\id_{\Phi^{k-2}(A)} \cdot t^0 & -\id_{\Phi^{k-2}(A)} \cdot
      t^{-1} & 0
      \\
      0 & 0 & 0 & \cdots & 0 & -\id_{\Phi^{k-1}(A)} \cdot t^0 & -\id_{\Phi^{k-1}(A)} \cdot
      t^{-1}
    \end{pmatrix}
  \]
  
  of morphisms in $\cala_{\Phi}[\IZ]$.

  On the other hand we can assign to $\varphi \colon \Phi^k(A) \to A$ in $\cala$ the
  morphism in $\cala_{\Phi^k}[\IZ]$ given
  \[
    \id_{A} \cdot t^{-1} - \varphi \cdot t^0 \colon \Phi^k(A) \to A.
  \]
  Its image under the induction functor $(i_k)_*$ is the morphism in $\cala_{\Phi}[\IZ]$
  given by
  \[
    \id_A \cdot t^{-k} - \varphi \cdot t^0 \colon \Phi^k(A) \to A.
  \]
  Consider the morphism
  $w \colon \bigoplus_{i = 1}^k \Phi^i(A) \xrightarrow{\cong} \bigoplus_{i = 1}^k
  \Phi^i(A)$ in $\cala_{\Phi}[\IZ]$ given by the $(k,k)$-matrix
  \[
    \begin{pmatrix}
      \id_{\Phi(A)} \cdot t^0 & \id_{\Phi(A)} \cdot t^{-1} & \id_{\Phi(A)} \cdot t^{-2} &
      \cdots & \id_{\Phi(A)} \cdot t^{-k+3} & \id_{\Phi(A)} \cdot t^{-k+2} & \id_{\Phi(A)}
      \cdot t^{-k+1}
      \\
      0 & \id_{\Phi^2(A)} \cdot t^0 & \id_{\Phi^2(A)} \cdot t^{-1} & \cdots &
      \id_{\Phi^2(A)} \cdot t^{-k+4} & \id_{\Phi^2(A)} \cdot t^{-k+3} & \id_{\Phi^2(A)}
      \cdot t^{-k+2}
      \\
      0 & 0 & \id_{\Phi^3(A)} \cdot t^0 & \cdots & \id_{\Phi^3(A)} \cdot t^{-k+5} &
      \id_{\Phi^3(A)} \cdot t^{-k+4} & \id_{\Phi^3(A)} \cdot t^{-k+3}
      \\
      \vdots & \vdots & \vdots & \vdots & \vdots & \vdots & \vdots
      \\
      0 & 0 & 0 & \cdots & \id_{\Phi^{k-2}(A)} \cdot t^0 & \id_{\Phi^{k-2}(A)} \cdot
      t^{-1} & \id_{\Phi^{k-2}(A)} \cdot t^{-2}
      \\
      0 & 0 & 0 & \cdots & 0 & \id_{\Phi^{k-1}(A)} \cdot t^0 & \id_{\Phi^{k-1}(A)} \cdot
      t^{-1}
      \\
      0 & 0 & 0 & \cdots & 0 & 0 & \id_{\phi^{k}(A)} \cdot t^0
    \end{pmatrix}
  \]
  of morphisms in $\cala_{\Phi}[\IZ]$.  Note that $w$ is the same as the identity
  $(k,k)$-matrix from the $K$-theoretic point of view by its block structure. Then the
  composite
  \[\bigl(\id_{\bigoplus_{j = 0}^{k-1} \Phi^j(A)} \cdot t^{-1} - V_k(\varphi) \cdot
    t^0\bigr) \circ w \colon \Phi\bigl(\bigoplus_{j = 0}^{k-1} \Phi^j(A)\bigr)
    =\bigoplus_{i = 1}^k \Phi^i(A) \to \bigoplus_{j = 0}^{k-1} \Phi^j(A)
  \]
  is given by the $(k,k)$-matrix
  \[
    \begin{pmatrix}
      \id_A \cdot t^{-1} & \id_A \cdot t^{-2} & \id_A \cdot t^{-3} & \cdots & \id_A \cdot
      t^{-k +2} & \id_A \cdot t^{-k +1} & \id_A \cdot t^{-k} - \varphi \cdot t^0
      \\
      -\id_{\Phi(A)} \cdot t^0 & 0 & 0 & \cdots & 0 & 0 & 0
      \\
      0 & -\id_{\Phi^2(A)} \cdot t^0 & 0 & \cdots & 0 & 0 & 0
      \\
      \vdots & \vdots & \vdots & \vdots & \vdots & \vdots & \vdots
      \\
      0 & 0 & 0 & \cdots & 0 & 0 & 0
      \\
      0 & 0 & 0 & \cdots & -\id_{\Phi^{k-2}(A)} \cdot t^0 & 0 & 0
      \\
      0 & 0 & 0 & \cdots & 0 & -\id_{\Phi^{k-1}(A)} \cdot t^0 & 0
    \end{pmatrix}
  \]
  of morphisms in $\cala_{\Phi}[\IZ]$.  Note that this composite looks like
  $\id_A \cdot t^{-k} - \varphi \cdot t^0$ from the $K$-theoretic point of view by its
  block structure.  This finishes the proof of
  assertion~\ref{F_k_and_V_k_versus_I_k_ast_and_i_k_upper_ast:induction}.
  \\[1mm]~\ref{F_k_and_V_k_versus_I_k_ast_and_i_k_upper_ast:restriction} The following
  diagram commutes for $n \in \IZ$
  \[
    \xymatrix@!C=15em{\pi_n(\bfT_{\bfK(\Phi^{-1})})\oplus N\!K_n(\cala_{\Phi^1}[t])
      \oplus N\!K_n(\cala_{\Phi^1}[t^{-1}]) \ar[r]^-{a_n\oplus b_n^+ \oplus
        b_n^-}_-{\cong} \ar[d]_{(\pi_n(\bff) \circ \trf_n(\bfp_k)) \oplus
        N\!K_n(i_k^*)_{+} \oplus N\!K_n(i_k^*)_{-}}
      &
      K_n(\cala_{\Phi}[\IZ])
      \ar[d]^{K_n(i_k^*)}
      \\
      \pi_n(\bfT_{\bfK(\Phi^{-k})}) \oplus N\!K_n(\cala_{\Phi^k}[t]) \oplus
      N\!K_n(\cala_{\Phi^k}[t^{-1}]) \ar[r]_-{a_n\oplus b_n^+ \oplus b_n^- }^-{\cong}
      &
      K_n(\cala_{\Phi}^k[\IZ]) }
  \]
  where the upper horizontal arrow is the
  isomorphism~\eqref{twisted_BHS_on_homotopy_groups_for_NK} for $\Phi$, the lower
  horizontal arrow is the isomorphism~\eqref{twisted_BHS_on_homotopy_groups_for_NK} for
  $\Phi^k$ and the vertical homomorphism have been defined
  in~\eqref{pi_n(bbf)},~\eqref{trf_n(bfp_k)},~\eqref{K_n(i_k_upper_ast)}
  and~\eqref{NK_n(i_k)_upper_ast}. The proof of commutativity for the terms
  $N\!K_n(\cala_{\Phi^k}[t])$ and $N\!K_n(\cala_{\Phi^k}[t^{-1}])$ is obvious, the one for
  the term $\pi_n(\bfT_{\bfK(\Phi^{-k})})$ is left to the reader.

  Therefore it remains to show that the diagram
  \begin{equation}
    \xymatrix@!C=10em{N\!K_n(\cala_{\Phi}[t^{\pm}]) \ar[r]_-{\cong}^-{\widetilde{\alpha}(\Phi,\pm)_n}
      \ar[d]_{N\!K_n(i_k^*)_{\pm}}
      &
      \overline{K}_{n-1}(\Nil(\cala,\Phi)) \ar[d]^{\overline{K}_{n-1}(F_k)}
      \\
      N\!K_n(\cala_{\Phi^k}[t^{\pm}]) \ar[r]^-{\cong}_-{\widetilde{\alpha}(\Phi^k,\pm)_n}
      &
      \overline{K}_{n-1}(\Nil(\cala,\Phi^k))
    }
    \label{diagram_comparing_Nil_and_NK_for_restriction}
  \end{equation}
  commutes for $n \in \IZ$. By the same argument as it appears in the commutativity of the
  diagram~\eqref{diagram_comparing_Nil_and_NK_for_Induction}, one can show that it suffices
  to prove the commutativity of~\eqref{diagram_comparing_Nil_and_NK_for_restriction} for
  $n \ge 1$, or, in other words for connective $K$-theory. By
  inspecting~\cite{Lueck-Steimle(2016BHS)} one sees that this boils down to show that the
  following diagram
  \begin{equation}
    \xymatrix@!C=10em{K_{n}(\Nil(A,\Phi)) \ar[r]^-{K_n(\chi_{\Phi})}  \ar[d]_{K_n(F_k)}
      &
      K_n(\Chcat(\cala_{\Phi}[t^{-1}])^w) \ar[d]^{\bfK(\Chcat(i_k^*)^w)}
      \\
      K_{n}(\Nil(A,\Phi^k)) \ar[r]_-{K_n(\chi_{\Phi^k})}
      &
      K_n(\Chcat(\cala_{\Phi^k}[t^{-1}])^w) 
    }
    \label{diagram_comparing_Nil_and_Ch_for_restriction}
  \end{equation}
  commutes for $n \ge 1$.

  Next we explain the key ingredients in the proof of the commutativity
  of~\eqref{diagram_comparing_Nil_and_Ch_for_restriction} for connective $K$-theory
  and leave it to the reader to
  figure out the routine to fill in the details based on standard fact about $K$-theory
  such as the Additivity Theorem for Waldhausen categories.

  Let $\Phi \colon \cala \xrightarrow{\cong} \cala$ be an automorphism of an additive
  category.  Consider a nilpotent endomorphism $\varphi \colon \Phi(A) \to A$ representing
  an element in $\Nil(\cala,\Phi)$. Consider the morphism
  $\id_A \cdot t^{-1} - \varphi \cdot t^0 \colon \Phi(A) \to A$ in $\cala_{\Phi}[\IZ]$.
  The functor $i_k^* \colon \cala_{\Phi}[\IZ] \to \cala_{\Phi^k}[\IZ]$ sends it to the
  morphism $\bigoplus_{i = 1}^k \Phi^i(A) \to \bigoplus_{j = 0}^{k-1} \Phi^j(A)$ in
  $\cala_{\Phi^k}[\IZ]$ that is given by the $(k,k)$-matrix
  \[
    \begin{pmatrix}
      - \varphi \cdot t^0 & 0 & 0 & \cdots & 0& 0 & \id_{\Phi^k(A)} \cdot t^{-1}
      \\
      \id_{\Phi(A)} \cdot t^0 & - \Phi(\varphi) \cdot t^0 & 0 & \cdots & 0 & 0 & 0
      \\
      0 & \id_{\Phi^2(A)} \cdot t^0 & - \Phi^2(\varphi) \cdot t^0 & \cdots & 0 & 0 & 0
      \\
      \vdots & \vdots & \vdots & \vdots & \vdots & \vdots & \vdots
      \\
      0 & 0 & 0 & \cdots & \id_{\Phi^{k-2}(A)} \cdot t^0 & - \Phi^{k-2}(\varphi) \cdot t^0
      & 0
      \\
      0 & 0 & 0 & \cdots & 0 & \id_{\Phi^{k-1}(A)} \cdot t^0 & - \Phi^{k-1}(\varphi) \cdot
      t^0
    \end{pmatrix}
  \]
  of morphisms in $\cala_{\Phi}[\IZ]$.

  If we apply the Frobenius $F_k$ operator to the object $\varphi \colon \Phi(A) \to A$ of
  $\Nil(\cala,\Phi)$, we obtain the object
  $\varphi^{(k)} \colon \Phi^k(A) \xrightarrow{\Phi^{k-1}(\varphi)} \Phi^{k-1}(A)
  \xrightarrow{\Phi^{k-2}(\varphi)} \cdots \xrightarrow{\varphi} A$ of
  $\Nil(\cala,\Phi^k)$. To it we can assign the morphism
  $\id_A \cdot t^{-1} - \varphi^{(k)} \cdot t^0 \colon \Phi^k(A) \to A$ in
  $\cala_{\Phi^k}[\IZ]$.

  Consider the morphism
  \[
    u \colon \bigoplus_{i = 1}^k \Phi^i(A) \to \bigoplus_{j= 0}^{k-1} \Phi^j(A)
  \]
  in $\cala_{\Phi^k}[\IZ]$ that is given by the $(k,k)$-matrix
  \[
    \begin{pmatrix}
      - \varphi \cdot t^0 & - \varphi^{(2)} \cdot t^0 & - \varphi^{(3)} \cdot t^0 & \cdots
      & - \varphi^{(k-2)} \cdot t^0 & - \varphi^{(k-1)} \cdot t^0 & \id_A \cdot t^{-1} -
      \varphi^{(k)} \cdot t^0
      \\
      \id_{\Phi(A)} \cdot t^0 & 0 & 0 & \cdots & 0 & 0 & 0
      \\
      0 & \id_{\Phi^2(A)} \cdot t^0 & 0 & \cdots & 0 & 0 & 0
      \\
      0 & 0 & \id_{\Phi^3(A)} \cdot t^0 & \cdots & 0 & 0 & 0
      \\
      \vdots & \vdots & \vdots & \vdots & \vdots & \vdots & \vdots
      \\
      0 & 0 & 0 & \cdots & \id_{\Phi^{k-2}(A)} \cdot t^0 & 0 & 0
      \\
      0 & 0 & 0 & \cdots & 0 & \id_{\Phi^{k-1}(A)} \cdot t^0 & 0
    \end{pmatrix}
  \]
  of morphisms in $\cala_{\Phi^k}[\IZ]$. Note that from a $K$-theoretic point of view this
  morphism should give the same element in $K$-theory as the morphisms
  $\id_A \cdot t^{-1} - \varphi^{(k)} \cdot t^0 \colon \Phi^k(A) \to A$ in
  $\cala_{\Phi^k}[\IZ]$, just view its special block structure.

  We also have the automorphism
  \[
    v \colon \bigoplus_{i = 1}^{k} \Phi^j(A) \xrightarrow{\cong} \bigoplus_{j = 1}^{k}
    \Phi^j(A)
  \]
  in $\cala_{\Phi^k}[\IZ]$ given by the $(k,k)$-matrix
  \[
    \begin{pmatrix}
      \id_{\Phi(A)} \cdot t^0 & - \Phi(\varphi) \cdot t^0 & 0 & \cdots & 0 & 0 & 0
      \\
      0 & \id_{\Phi^2(A)} \cdot t^0 & - \Phi^2(\varphi) \cdot t^0 & \cdots & 0 & 0 & 0
      \\
      0 & 0 & - \id_{\Phi^3(A)} \cdot t^0 & \cdots & 0 & 0 & 0
      \\
      \vdots & \vdots & \vdots & \vdots & \vdots & \vdots & \vdots
      \\
      0 & 0 & 0 & \cdots & \id_{\Phi^{k-2}(A)} \cdot t^0 & - \Phi^{k-2}(\varphi) \cdot t^0
      & 0
      \\
      0 & 0 & 0 & \cdots & 0 & \id_{\Phi^{k-1}(A)} \cdot t^0 & \Phi^{k-1}(\varphi) \cdot
      t^0
      \\
      0 & 0 & 0 & \cdots & 0 & 0 & \id_{\Phi^{k}(A)} \cdot t^0
    \end{pmatrix}
  \]
  of morphisms in $\cala_{\Phi^k}[\IZ]$. Note that $v$ is the same as the identity
  $(k,k)$-matrix from the $K$-theoretic point of view by its block structure.

  Now one easily check the equality of morphisms
  $\bigoplus_{i = 1}^k \Phi^i(A) \to \bigoplus_{i = 0}^{k} \Phi^j(A)$ in
  $\cala_{\Phi^k}[\IZ]$
  \[
    u \circ v = i_k^*\bigl(\id_A \cdot t^{-1} - \varphi \cdot t^0\bigr).
  \]
  Now Theorem~\ref{F_k_and_V_k_versus_I_k_ast_and_i_k_upper_ast} follows.
\end{proof}

Let
\begin{eqnarray}
\overline{s}_n \colon \overline{K}_n(\Nil(\cala,\Phi)) & \to & K_{n-1}(\cala_{\Phi}[\IZ])
\label{injection_overline(s)_n}
\end{eqnarray}
be the homomorphism  coming from the Bass-Heller-Swan homomorphism,
see~\eqref{twisted_BHS_on_homotopy_groups},
and the inclusion  of the first Nil-term. Let
\begin{eqnarray}
\overline{r}_n \colon K_{n-1}(\cala_{\Phi}[\IZ]) & \to & \overline{K}_n(\Nil(\cala,\Phi)) 
\label{projection_overline(r)_n}
\end{eqnarray}
be the homomorphisms  coming from the Bass-Heller Swan homomorphism,
see~\eqref{twisted_BHS_on_homotopy_groups},
and the projection onto the first Nil-term. Obviously we have $\overline{r}_n \circ \overline{s}_n = \id$.

Theorem~\ref{F_k_and_V_k_versus_I_k_ast_and_i_k_upper_ast} implies

\begin{corollary}\label{cor:ind_res_V_and_F}\
  \begin{enumerate}
  \item\label{cor:ind_res_V_and_F:ind} We have for every $n \in \IZ$ the commutative diagram
    \[\xymatrix@!C=12em{ K_n(\cala_{\Phi^k}[\IZ]) \ar[r]^-{K_n((i_k)_*)} \ar[d]_{\overline{s}_n(\cala,\Phi^k)}
        & K_n(\cala_{\Phi}[\IZ]) \ar[d]^{\overline{s}_n(\cala,\Phi)}
        \\
        \overline{K}_{n-1}(\Nil(\cala,\Phi^k)) \ar[r]_-{\overline{K}_{n-1}(V_k)} &
        \overline{K}_{n-1}(\Nil(\cala,\Phi)) }
    \]
    where the vertical arrows have been defined n~\eqref{injection_overline(s)_n}
    and the horizontal arrows
    have been defined in~\eqref{widetilde(K)_n(V_k)} and~\eqref{K_n((i_k)_ast)};

  \item\label{cor:ind_res_V_and_F:res} We have for every $n \in \IZ$ the commutative diagram
    \[
      \xymatrix@!C=12em{K_n(\cala_{\Phi}[\IZ]) \ar[r]^-{K_n(i_k^*)} \ar[d]_{\overline{r}_n(\cala,\Phi)} &
        K_n(\cala_{\Phi^k}[\IZ]) \ar[d]^{\overline{r}_n(\cala,\Phi^k)}
        \\
        \overline{K}_{n-1}(\Nil(\cala,\Phi)) \ar[r]_-{\overline{K}_{n-1}(F_k)} &
        \overline{K}_{n-1}(\Nil(\cala,\Phi^k))}
    \]
    where the vertical arrows have been defined in~\eqref{projection_overline(r)_n}
    and the horizontal arrows have been defined in~\eqref{widetilde(K)_n(F_k)}
    and~\eqref{K_n(i_k_upper_ast)}.
  \end{enumerate}
\end{corollary}

Corollary~\ref{cor:ind_res_V_and_F}
has already been proved for commutative rings  as coefficients in the untwisted
case by Stienstra~\cite[Theorem~4.7]{Stienstra(1982)}.

For a natural number  $D$,  let $\Nil(\cala,\Phi)_D$ be the full subcategory of $\Nil(\cala,\Phi)$
  consisting of objects $(P,\varphi)$ satisfying $\varphi^{(D)} = 0$.  Obviously
  $\Nil(\cala,\Phi) = \bigcup_{d \ge 0 } \Nil(\cala,\Phi)_d$. Hence the canonical maps
  \begin{eqnarray}
    \colim_{D \to \infty} K_n(\Nil(\cala,\Phi)_D)
    & \xrightarrow{\cong} &
   K_n(\Nil(\cala,\Phi));
   \label{K_n(NIL)_as_colimit_over_K_n(NIL_D)}
    \\
    \colim_{D \to \infty} \overline{K}_n(\Nil(\cala,\Phi)_D)
    & \xrightarrow{\cong} &
     \overline{K}_n(\Nil(\cala,\Phi)),
       \label{overline(K)_n(NIL)_as_colimit_over_overline(K)_n(NIL_D)}
  \end{eqnarray}                                                                                 
  are bijective.  Given an element $z \in \overline{K}_n(\Nil(\cala,\Phi))$ and a natural
  number $D$, we say that $z$ is of \emph{nilpotence degree} $\le D$ if $z$ lies in the
  image of $\overline{K_n}(\Nil(\cala,\Phi)_D)\to \overline{K}_n(\Nil(\cala,\Phi))$.

  The next result is known  for rings  as coefficient  in the untwisted case,
  see Farrell~\cite[Lemma~3]{Farrell(1977)} for $n = 1$,
  Grunewald~\cite[Prop.~4.6]{Grunewald(2008Nil)}, Stienstra~\cite[p.~90]{Stienstra(1982)},
  and   Weibel~\cite[p.~479]{Weibel(1981)}.

 \begin{lemma}\label{lem:NK_n(F_k)(z)_vanishes_for_large_k} Fix an integer $n \in \IZ$
   Consider an element $z \in \overline{K}_n(\Nil(\cala,\Phi))$ of nilpotence degree
   $\le D$.
  
   Then the composite
   \[\overline{K}_{n-1}(\Nil(\cala,\Phi)) \xrightarrow{\overline{s}_n} K_n(\cala_{\Phi}[\IZ])
     \xrightarrow{K_n(i_k^*)} K_n(\cala_{\Phi^k}[\IZ])
   \]
   sends $z$ for every $k \ge D$ to zero, where $\overline{s}_n$ and $K_n(i_k^*)$ have been
   defined in~\eqref{injection_overline(s)_n} and~\eqref{K_n(i_k_upper_ast)}.
 \end{lemma}
 \begin{proof}
   Since the composite
   $\Nil(\cala,\Phi)_D\to \Nil(\cala,\Phi) \xrightarrow{F_k} \Nil(\cala,\Phi^k)$ sends
   an object $(P,\phi)$ to $(P,0)$ for $k \ge D$, the composite
   $\overline{K}_n(\Nil(\cala,\Phi)) \xrightarrow{\overline{K}_n(F_k)}
   \overline{K}_n(\Nil(\cala,\Phi^k))$ sends $z$ to zero for $k \ge D$.  Now
   Lemma~\ref{lem:NK_n(F_k)(z)_vanishes_for_large_k} follows from
   Corollary~\ref{cor:ind_res_V_and_F}~\ref{cor:ind_res_V_and_F:ind}.
 \end{proof}


\typeout{------------- Section 5: Mackey and Green functors for finite groups     ------------------------}

\section{Mackey and Green functors for finite  groups}%
\label{sec:Mackey_and_Green_functors_for_finite_groups}

For the reader's convenience we recall some basics about Mackey and Green functors and Dress induction
following~\cite[Section~2 and~3]{Bartels-Lueck(2007ind)}.  Throughout this section $F$ is a finite group.


\subsection{Mackey functors}\label{subsec:Mackey_functors}

 Let $\SETSf{F}$ be the category, whose objects are finite $F$-sets and 
whose morphisms are $F$-maps.

Let $\Lambda$ be an associative  commutative ring with unit.
Denote by $\MODcat{\Lambda}$  the abelian category of $\Lambda$-modules.
A \emph{bi-functor} $M = (M_*,M^*)$  from $\SETSf{F}$ to $\MODcat{\Lambda}$
consists of a covariant  functor
\[M_* \colon \SETSf{F} \to \MODcat{\Lambda}
\]
and a contravariant functor
\[
  M^* \colon \SETSf{F} \to \MODcat{\Lambda}
\]
that  agree on objects. We often write $M(X) = M_*(X) = M^*(X)$ for an object $X$ in $\SETSf{F}$
and $f_* = M_*(f)$ and $f^* = M^*(f)$ for a morphism $f \colon X \to Y$ in  $\SETSf{F}$

\begin{definition}[Mackey functor]\label{def:Mackey_functor}
A \emph{Mackey functor $M$ for $F$ with values in $\Lambda$-modules} is a bifunctor
from $\SETS{F}$ to $\MODcat{\Lambda}$ such that
\begin{itemize}
\item Double Coset formula\\[1mm]
  For any cartesian square in $\SETSf{F}$
  \[
  \xymatrix{X \ar[r]^{\overline{v}} \ar[d]_{\overline{u}}
    &
    X_1 \ar[d]^{u}
    \\
    X_2 \ar[r]_{v}
    &
    X_0
  }
\]
the following diagram of functors of abelian groups commutes
\[
  \xymatrix{M(X) \ar[r]^{\overline{v}_*} 
    &
    M(X_1) 
    \\
    M(X_2) \ar[r]_{v_*} \ar[u]^{\overline{u}^*}
    &
    M(X_0); \ar[u]_{u^*}
  }
\]

\item Additivity\\[1mm]
Consider two  objects  $X$ and $Y$ in $\SETSf{F}$. Let 
$i \colon X \to X \amalg Y$  and $j\colon Y \to X \amalg Y$ be the inclusions. Then the map
\[
i^* \times j^* \colon M(X \amalg Y) \to  M(X) \times M(Y)
\]
is bijective;

\end{itemize}
\end{definition}

\begin{remark}\label{rem:reformulation_of_Additivity}
One easily checks that the condition Additivity implies $M(\emptyset) = 0$
and is equivalent to the
requirement that $M_*(i) \oplus M_*(j) \colon M(X) \oplus M(Y) \to M(X \amalg Y)$
is bijective,  since the double coset formula implies
that $\id \times \id \colon M(\emptyset) \to M(\emptyset) \times  M(\emptyset)$ is an isomorphism
and  that $(M^*(i) \times  M^*(j )) \circ (M_*(i) \oplus  M_*(j))$ is
the identity.
\end{remark}

Let $M$, $N$ and $L$ be bi-functors from $\SETSf{F}$
values in $\Lambda$-modules. A pairing
\begin{equation}
  M \times N \to L
  \label{pairing_of_bifunctors}
\end{equation}
is a family of $\Lambda$-bilinear maps
\[
  \mu(X) \colon M(X) \times N(X) \to L(X), \quad (m,n) \mapsto \mu(m,n) = m \cdot n
\]
indexed by the objects $X$ of  $\SETSf{F}$ such that for every morphism $f \colon X \to Y$ in $\SETSf{F}$
we have
\begin{eqnarray}
&\begin{array}{lllll}
L^*(f)(x \cdot y) & = & M^*(f)(x) \cdot N^*(f)(y), & & x \in M(Y), y \in N(Y);
\\
x \cdot N_*(f)(y) & = & L_*(f)(M^*(f)(x) \cdot y),  & & x \in M(Y), y \in N(X);
\\
M_*(f)(x) \cdot y & = & L_*(f)(x \cdot N^*(f)(y)),  & & x \in M(X), y \in N(Y).
\end{array} &%
\label{conditions_for_a_pairing}
\end{eqnarray}


\subsection{Green functors}\label{subsec:Green_functors}

\begin{definition}[Green functor] A \emph{Green functor} for the finite group $F$ with
  values in $\Lambda$-modules is a Mackey functor $U$ together with a pairing
\[
  \mu \colon U \times U \to U
\]
and a choice of elements $1_X \in U(X)$ for each object $X$ in $\SETSf{F}$
such that for every  object $X$ in $\SETSf{F}$ the pairing 
$\mu(X) \colon U(X) \times U(X) \to U(X)$ and the element $1_X$ determine the structure of
an associative $\Lambda$-algebra with unit on $U(X)$.
Moreover, it is required that $U^*(f)(1_Y) = 1_X$ for every morphism $f \colon X \to Y$ in
$\SETSf{F}$.

A (left) $U$-module $M$ is a Mackey functor for the  group $F$ with values in
$\Lambda$-modules together with a pairing
\[
  \nu \colon U \times M \to M
\]
such that for every object $X$ in $\SETSf{F}$  the pairing $\nu(X) \colon U(X) \times M(X) \to
M(X)$ defines the structure of a $U(X)$-module on $M(X)$,
where $1_X$ acts as $\id_{M(X)}$.
\end{definition}


\subsection{Dress induction}\label{subsec:Dress_induction}

Let $F$ be a finite group. Let $\calf$ be a family of
subgroups of $F$ that is closed under taking subgroups and conjugation. An example for
$\calf$ is the family $\calh_p$ of $p$-hyperelementary subgroups of $F$ for any prime
$p$. A Green functor $\calu$ over $F$ is called \emph{$\calf$-computable} if the canonical
$\Lambda$-map
\[
\bigoplus_{H \in \calf}  \calu_*(\pr_H) \colon \bigoplus_{H \in \calf} \calu(F/H) \to \calu(F/F)
\]
is surjective, where $\pr_H \colon F/H \to F/F$ is the projection.
The next result is a mild generalization of the fundamental
work of Dress on induction, see~\cite{Dress(1973), Dress(1975)}.

  \begin{lemma}\label{lem:general_induction}
    Let $\calu$ be a Green functor over $F$ that is $\calf$-computable. Consider any $\calu$-module
    $\calm$. Then for every element $z \in \calm(F/F)$ we can find elements
    $u_H \in \calu(F/H)$ for $H \in \calf$ satisfying
    \[z = \sum_{H \in \calf}  \calm_*(\pr_H)\bigl( u_H\cdot \calm^*(\pr_H)(z)\bigr),
    \]
    where $\pr_H \colon F/H \to F/F$ is the projection.
  \end{lemma}
  \begin{proof}
     By assumption we can write the unit $1_{F/F} \in \calu(F/F)$ as the sum
    \[
    1_{F/F} = \sum_{H \in \calf} \calu_*(\pr _H)(u_H).
  \]
  for appropriate elements $u_H \in M(F/H)$.
  Now we compute using the various axioms for the structure of a Green functor and a
  module over a Green functor
  \begin{eqnarray*}
    z
    & = &
    1_{F/F} \cdot z
    \\
    & = &
    \biggl(\sum_{H \in \calf} \calu_*(\pr_H )(u_H)\biggr)  \cdot z
    \\
    & = &
   \sum_{H \in \calf} \calu_*(\pr_H)(u_H) \cdot z
    \\
    & = &
   \sum_{H \in \calf}  \calm_*(\pr_H)\bigl( u_H\cdot \calm^*(\pr_H)(z)\bigr).
  \end{eqnarray*}
\end{proof}


\subsection{The Green functor $\Sw_F$}\label{subsec:Sw_F}

Let $G$ be a group. Next we define the \emph{Swan ring} $\Sw(G)$ associated to $G$.
As an abelian group $\Sw(G)$ is defined in terms of generators and
relations as follows.  Generators $[M]$ are given by $\IZ G$-isomorphism classes of
$\IZ G$-modules $M$, which are finitely generated free as $\IZ$-modules.  If we have an
exact sequence $0 \to M_0 \to M_1 \to M_2 \to 0$ of such modules, then we require the
relation $[M_1] = [M_0] + [M_2]$. The multiplication in the Swan ring is given by the
tensor product over $\IZ$ with the diagonal $G$-action. Note that $\Sw(G)$ is a
commutative ring, whose unit is given by $\IZ$ equipped with the trivial $G$-action.

In~\cite[Section~3]{Bartels-Lueck(2007ind)} the definition of $\Sw(G)$ is extended
from groups to small groupoids.   For a small groupoid
$\calg$ an element in $\Sw(\calg)$ is given by a covariant functor from $\calg$ to
the category of finitely generated free abelian groups.  Actually, $\Sw(\calg)$ is
the Grothendieck group of such functors under the obvious objectwise notion of a short
exact sequence of such functors. The structure of a commutative ring comes from the tensor
product of abelian groups applied objectwise. The unit is represented by the constant
functor, which sends every  object in $\calg$ to $\IZ$ and every morphism in $\calg$ to
$\id_{\IZ}$.

Given a $G$-set $X$, denote by $\calt^G(X)$ its \emph{transport
  groupoid}. Its set of objects is $X$, its set of morphism from
$x_0 \to x_1$ is $G_{x_0,x_1} = \{g \in G \mid x_1 = gx_0\}$, and
composition comes from the multiplication in $G$.  Now we can consider
$\Sw(\calt^G(X))$ for a $G$-set $X$.  For every subgroup $H$ of $G$
there is an obvious isomorphism
\begin{equation}
  \Sw(H) \xrightarrow{\cong} \Sw(\calt^G(G/H))
  \label{iso_Sw_upper_f(H;Z)_cong_Sw_upper_f_Ff(calt_upper_F(F/H);Z)}
\end{equation}
coming from the facts that $\calt^G(G/H)$ is a connected groupoid and
the automorphism group of $eH$ is $H$.

The elementary proof that we get for a finite group $F$ a Green functor
$\Sw^F$, which is given for a finite $F$-set $X$ by
$\Sw^F(X) = \Sw(\calt^F(X))$, can be found
in~\cite[Section~3]{Bartels-Lueck(2007ind)}.

\begin{lemma}\label{lem:Swan_is_computable}
Let $F$ be a finite group. Then the  Green functor with values in
$\IZ_{(p)}$-modules $\IZ_{(p)} \otimes_{\IZ} \Sw_F$ is
computable for $\calh_p$. 
\end{lemma}
\begin{proof}
  See for instance~\cite[Proof of Lemma~4.1~(c)]{Bartels-Lueck(2007ind)} which is based
  on~\cite[6.3.3]{Dieck(1979)},~\cite[Lemma 4.1]{Swan(1960a)} and~\cite[Section 12]{Swan(1963)}.
\end{proof}

Let $p$ be a prime
and denote  by $\calh_p$ be the family of $p$-hyperelementary subgroups of the finite group $F$.
We get from Lemma~\ref{lem:general_induction} and Lemma~\ref{lem:Swan_is_computable}

 \begin{lemma}\label{lem:general_induction_for_SW} 
    Let $F$ be a finite group.  Consider any $\Sw_F$-module $\calm$. Then for every
   element $z \in \calm(F/F)$ we can find elements $a_H \in \IZ_{(p)}$ and
   $u_H \in \Sw(H)$ for every $H \in \calf$ such that
   \[
     z = \sum_{H \in \calf} a_H \cdot \calm_*(\pr_H)\bigl(u_H \cdot
     \calm^*(\pr_H)(z)\bigr)
   \]
   holds in $\calm(F/F)_{(p)} = \IZ_{(p)} \otimes_{\IZ} \calm(F/F)$, where
   $\pr_H \colon F/H \to F/F$ is the projection.
 \end{lemma}

 \begin{remark}\label{rem:two_version_of_Sw_agree}
   There is also a version of $\Sw(H)$ for groups $H$  and $\Sw(\calg)$ for groupoids $\calg$, where one
   replaces finitely generated free abelian group by finitely generated abelian group.
   These two versions are isomorphic, see~\cite[page 890]{Pedersen-Taylor(1978)}
   and~\cite[Section~3]{Bartels-Lueck(2007ind)} for groupoids.

   It is more convenient to work with the version for  finitely generated free abelian
   groups, since for finitely generated free abelian groups $M$ the functor
   $M \otimes_{\IZ} -$ is exact.
 \end{remark}


\typeout{------------- Section 6: $K$-theory functors associated to $G$-$\IZ$-categories
  ----------------------}

\section{$K$-theoretic functors associated to $G$-$\IZ$-categories}%
\label{sec:K-theoretic_functors_associated_to_G-IZ-categories}

Let $G$ be a group and let $\Lambda$ be a commutative ring.  


\subsection{The $K$-theoretic covariant $\OrG{G}$-spectrum associated to a $G$-$\IZ$-cate\-gory}%
\label{subsec:The_K-theoretic_covariant_Or(G)-spectrum_associated_to_an_G-Z-category}

Let $\cala$ be a $G$-$\Lambda$-category.
For a $G$-set $X$ we define a $\Lambda$-category $\cala(X)$ as
follows. Objects are pairs $(A,x)$ with $A \in \ob(\cala)$ and $x \in X$.  A morphism
$\varphi \colon (A,x) \to (A',x')$ is a formal finite sum
$\varphi = \sum_{g \in G_{x,x'}} \varphi_g \cdot g$, where
$G_{x,x'} = \{g \in G \mid x' = gx\}$ and $\varphi_g$ is a morphism in $\cala$ from $gA$
to $A'$.  For a morphism $\varphi' \colon (A',x) \to (A'',x'')$ given by the formal finite
sum $\varphi' = \sum_{g' \in G_{x',x''}} \varphi'_g \cdot g'$, we define the composite
$\varphi' \circ \varphi \colon (A,x) \to (A'',x'')$ by the finite formal sum
$\varphi' \circ \varphi = \sum_{g'' \in G_{x,x''}} (\varphi' \circ \varphi)_{g''} \cdot
g''$, where for $g'' \in G_{x,x''}$ we put
\[
  (\varphi' \circ \varphi)_{g''}
  = \sum_{\substack{g \in G_{x,x'}, g' \in G_{x',x''} \\ g'' = g'g}} \varphi'_{g'} \circ g'\varphi_g.
\]
Given a $G$-map $f \colon X \to Y$, define a functor of $\Lambda$-categories
$\cala(f) \colon \cala(X) \to \cala(Y)$ by sending an object $(A,x)$ to the object
$(A,f(x))$ and a morphism $\varphi \colon (A,x) \to (A',x')$ given by the formal
sum $\varphi = \sum_{g \in G_{x,x'}} \varphi_g \cdot g$ to the morphism given the same
finite formal sum. This makes sense because of $G_{x,x'} \subseteq G_{f(x),f(x')}$.
Hence we get a covariant functor
\begin{equation}
 Z_{\cala} \colon \SETS{G} \to \addcat, \quad X \mapsto \cala(X)_{\oplus}.
 \label{Z_cala_on_G-sets}
\end{equation}
It induces the covariant functor 
\begin{equation}
 \bfK_{\cala} \colon \SETS{G} \to \Spectra, \quad X \mapsto \bfK(\cala(X)_{\oplus}).
 \label{bfK_cala_on_G-sets}
\end{equation}
In particular we get a covariant functor called
\emph{the $K$-theoretic covariant $\OrG{G}$-spectrum associated to $\cala$}
\begin{equation}
  \bfK_{\cala} \colon \OrG{G} \to \Spectra, \quad G/H \mapsto \bfK(\cala(G/H)_{\oplus}).
  \label{bfK_cala_orbit-category}
\end{equation}

\begin{remark}\label{rem:FJ-assembly}
  Let $H_*^G(-;\bfK_{\cala})$ be the $G$-homology theory associated to $\bfK_{\cala}$.  It
  has the property that $H_n^G(G/H;\bfK_{\cala}) = \pi_n(\bfK_{\cala}(G/H))$ holds for any
  subgroup $H$ of $G$.  Denote by $\EGF{G}{\FIN}$ and $\EGF{G}{\VCYC}$ respectively the
  classifying space for the family $\FIN$ of finite subgroups and of the family $\VCYC$
  of virtually cyclic subgroups of $G$, see for instance~\cite{Lueck(2005s)}.
  The projection $\EGF{G}{\VCYC} \to G/G$ induces
the assembly map for $n \in \IZ$
  \[
    H_n^G(\EGF{G}{\VCYC};\bfK_{\cala}) \to H_n^G(G/G;\bfK_{\cala}) =
    \pi_n(\bfK_{\cala}(G/G)).
  \]
  The Farrell-Jones Conjecture predicts that it is bijective for all $n \in \IZ$.
  If $R$  is a ring and we take for $\cala$ the additive category $\underline{R}_{\oplus}$, then
  this assembly map can be identified with the assembly map~\ref{FJS-assembly_for_rings}.
  
  The canoncial map  $\EGF{G}{\FIN} \to  \EGF{G}{\VCYC}$  induces the relative assembly map
  \[
    H_n^G(\EGF{G}{\FIN};\bfK_{\cala}) \to H_n^G(\EGF{G}{\VCYC};\bfK_{\cala}).
  \]
  If $R$  is a ring and we take for $\cala$ the additive category $\underline{R}_{\oplus}$, then
  this assembly map can be identified with the relative assembly map appearing in
  Theorem~\ref{the:passage_from_Fin_to_Vcyc}.

  For more information about assembly maps and the Farrell-Jones Conjecture we refer for instance
  to~\cite{Lueck(2019handbook),Lueck(2022book),Lueck-Reich(2005)}.
\end{remark}


\subsection{Restriction}%
\label{subsec:Restriction}

For a $G$-map $f \colon X \to Y$, we think of
$K_n(\cala(f)_{\oplus}) \colon K_n(\cala(X)_{\oplus}) \to K_n(\cala(Y)_{\oplus})$
as induction. We can define a kind of restriction,
if we assume  that $f^{-1}(y)$ is finite for every $y\in Y$, as follows.
Next we define a functor of $\Lambda$-categories
\[
\res(f) \colon \cala(Y) \to \cala(X)_{\oplus}.
\]
It sends an object $(A,y)$ to $\bigoplus_{x \in f^{-1}(y)} (A,x)$.
Consider a morphism $\varphi \colon (A,y) \to (A',y')$ given by a formal sum
$\sum_{g \in G_{y,y'}} \varphi_g \cdot g$. Then $\res(f)(\varphi) \colon
\bigoplus_{x \in f^{-1}(y)} (A,x) \to \bigoplus_{x' \in f^{-1}(y')} (A',x')$
is given by the collection  of morphisms $\res(f)(\varphi)_{x,x'} \colon  (A,x)  \to  (A',x')$
for $x \in f^{-1}(y)$ and $x' \in f^{-1}(y')$ if we put
$\res(f)(\varphi)_{x,x'} = \sum_{g \in G_{x,x'}} \varphi_g \cdot g$.
This makes sense because of $G_{x,x'} \subseteq G_{y,y'}$.
The functor $\res(f)$ induces a functor of additive $\Lambda$-categories
$\res(f)_{\oplus}  \colon \cala(Y)_{\oplus}  \to (\cala(X)_{\oplus})_{\oplus}$.
  Composing it with the obvious equivalence of additive $\Lambda$-categories
  $(\cala(X)_{\oplus})_{\oplus} \xrightarrow{\simeq} \cala(X)_{\oplus}$ yields a functor of
  additive $\Lambda$-categories, denoted in the same way
  \begin{equation}
\res(f)_{\oplus} \colon \cala(Y)_{\oplus} \to \cala(X)_{\oplus}.
\label{res(f)_oplus}
\end{equation}


\subsection{The pairing with the Swan group}%
\label{subsec:The_pairing_with_the_Swan_group}

In this section we construct a  bilinear pairing for a $G$-set $X$ and $n \in \IZ$
\begin{equation}
P^G_X \colon \Sw(\calt^G(X)) \times K_n(\cala(X)_{\oplus}) \to K_n(\cala(X)_{\oplus}).
\label{pairing_SW_with_K_n}
\end{equation}

We have to construct for any covariant functor $M$ from $\calt^G(X)$
to the category of finitely generated free abelian groups a functor of
$\IZ$-categories $F(M) \colon \cala(X) \to \cala(X)_{\oplus}$, since
then we can define for the element $[M] \in \Sw(\calt^G(X))$
represented by $M$ the homomorphism
$P([M],-) \colon K_n(\cala(X)_{\oplus}) \to K_n(\cala(X)_{\oplus})$ to
be $K_n(F(M)_{\oplus})$.  For $x \in X$ let $r(x)$ be the rank of the
finitely generated free abelian group $M(x)$ and choose an ordered basis
$\{b_1, \ldots, b_{r(x)}\}$ of $M(x)$.  Consider $x,x' \in X$ with
$G_{x,x'} \not= 0$.  Then $r(x) = r(x')$ and for $g \in G_{x,x'}$ we
get an invertible $(r(x),r(x))$-matrix with entry $\rho(g)_{i,i'} \in \IZ$ for $i \in \{1, \ldots, r(x)\}$ and
$i' \in \{1, \ldots, r(x')\}$, which describes the homomorphism of
finitely generated free abelian groups $M(g) \colon M(x) \to M(x')$
for the morphism $g \colon x \to x'$ in $\tau^G(X)$ with respect to
the chosen ordered bases. Now $F(M)$ sends an object $(A,x)$ to
$\bigoplus_{i =1 }^{r(x)} A(x)$.  A morphism
$\sum_{g \in G} f_g \cdot g \colon (A,x) \to (A',x')$ is sent to the
morphism
$\bigoplus_{i =1 }^{r(x)} A(x) \to \bigoplus_{i' =1 }^{r(x')} A'(x')$,
whose component for $i \in \{1, \ldots, r(x)\}$ and
$i' \in \{1, \ldots, r(x')\}$ is given by
$\sum_{g \in G} \rho(g)_{i,i'} \cdot f_g \cdot g \colon (A,x) \to
(A',x')$.  Note that the choice of basis does not matter, since it does
not change the equivalence class of $F(M)$. We leave it to the reader
to check that we get a well-defined bilinear
pairing~\eqref{pairing_SW_with_K_n}.


\subsection{The special case of a ring as coefficients}%
\label{subsec:The_special_case_of_a_ring_as_coefficients}

The following example is illuminating and the most important one.

Let $\Lambda$ be a commutative ring and let $R$ be a $\Lambda$-algebra,
where we tacitly always require that $\lambda r = r \lambda$ holds for every
$\lambda \in \Lambda$ and $r \in R$. Consider a group homomorphism
$\rho \colon G \to \aut_{\Lambda}(R)$ to the group of automorphisms of
the $\Lambda$-algebra $R$. The \emph{twisted group ring with
  $R$-coefficients} $R_{\rho}[G]$ is defined as follows. Elements are
finite formal sums $\sum_{g \in G} r_g \cdot g$ for $r_g \in R$. The
multiplication in $R_{\rho}[G]$ is given by
\[
  \left(\sum_{g \in G} r_g \cdot g\right) \cdot \left(\sum_{g' \in G}
    r'_{g'} \cdot g'\right) = \sum_{g'' \in G}
  \left(\sum_{\substack{g,g' \in G\\g'' = g'g}} r_g \cdot
    \rho(g)(r'_{g'})\right) \cdot g''.
\]
Note that $R_{\rho}[G]$ inherits the structure of  $\Lambda$-algebra from $R$.

Define $\underline{R}$ to be the $\Lambda$-category, which has precisely one object $\ast_R$
and whose $\Lambda$-module of endomorphisms is $R$. Composition is given by the
multiplication in $R$ and the $\Lambda$-structure on the endomorphisms comes from the
structure of a  $\Lambda$-algebra.  From $\rho$ we obtain the structure of
$G$-$\Lambda$-category on $\underline{R}$ and we can consider the $\Lambda$-category
$\underline{R}(X)$ for any $G$-set $X$.

Consider a subgroup $H \subseteq G$.
There is an equivalence of $\Lambda$-categories
\begin{equation}
T(H) \colon \underline{R_{\rho|_H}[H]} \xrightarrow{\simeq} \underline{R}(G/H),
\label{underline(R_rho|_H[H]_simeq_underline(R)(G/H)}
\end{equation}
that sends the object $\ast_{R_{\rho|_H}[H]}$ to the object
$(\ast_R, eH)$ and a morphism in $\underline{R_{\rho|_H}[H]}$, which
is given by the element $\sum_{h \in H} r_h \cdot h$ in
$R_{\rho|_H}[H]$, to the morphism $(\ast_R, eH) \to (\ast_R, eH)$ in
$\underline{R}(G/H)$ given by $\sum_{h \in H} r_h \cdot h$
again. Obviously $T(H)$ is fully faithful.
Every object in $\underline{R}(G/H)$ is isomorphic to an object
in the image of $T(H)$, namely, for any object $(\ast_R, gH)$ we get
an isomorphism $(\ast_R, eH) \xrightarrow{\cong} (\ast_R, gH)$ by
$1_R \cdot g$. Hence $T_H$ is an equivalence of $\Lambda$-categories.
From~\eqref{underline(R_rho|_H[H]_simeq_underline(R)(G/H)} we obtain
an equivalence of additive $\Lambda$-categories
\begin{equation}
T(H)_{\oplus} \colon \underline{R_{\rho|_H}[H]}_{\oplus} 
  \xrightarrow{\simeq} \underline{R}(G/H)_{\oplus}.
\label{underline(R_rho|_H[H]_oplus_simeq_underline(R)(G/H)_oplus}
\end{equation}
and hence for every $n \in \IZ$ an isomorphism
\begin{equation}
  K_n(T(H)_{\oplus}) \colon K_n(R_{\rho|_H}[H]) \xrightarrow{\cong} K_n(\underline{R}(G/H)_{\oplus}).
 \label{K_n(R_(rho|_H)[H])_cong_K_n(underline(R)(G/H)_oplus)}
\end{equation}
Consider subgroups $H \subseteq K \subseteq G$ with $[K:H] < \infty$. Let
$\pr \colon G/H \to G/K$ be the projection and $i \colon H \to K$ be the inclusion. Then
$K_n(\underline{R}(\pr)) \colon K_n(\underline{R}(G/H)) \to K_n(\underline{R}(G/K))$
corresponds under the
isomorphism~\eqref{K_n(R_(rho|_H)[H])_cong_K_n(underline(R)(G/H)_oplus)} to the induction
homomorphism $i_* \colon K_n(R_{\rho|_H}[H]) \to K_n(R_{\rho|_K}[K])$, whereas
$K_n(\res(\pr)_{\oplus}) \colon K_n(\underline{R}(G/K)_{\oplus}) \to
K_n(\underline{R}(G/H)_{\oplus})$ corresponds under the
isomorphism~\eqref{K_n(R_(rho|_H)[H])_cong_K_n(underline(R)(G/H)_oplus)} to the
restriction homomorphism $i^* \colon K_n(R_{\rho|_K}[K]) \to K_n(R_{\rho|_H}[H])$.  The
pairing $P$ of~\eqref{pairing_SW_with_K_n} corresponds under the
isomorphisms~\eqref{iso_Sw_upper_f(H;Z)_cong_Sw_upper_f_Ff(calt_upper_F(F/H);Z)}
and~\eqref{K_n(R_(rho|_H)[H])_cong_K_n(underline(R)(G/H)_oplus)} to the pairing
$ \Sw(H) \times K_n(R_{\rho|_H}[H]) \to K_n(R_{\rho|_H}[H])$  coming from the
fact that, for a $\IZ H$-module $M$, whose underlying abelian group is finitely generated
free, and a finitely generated projective $R_{\rho|_H}[H]$-module $P$, we can equip
$M \otimes_{\IZ} P$ with the $R_{\rho|_H}[H]$-module structure given by
\[
  \left(\sum_{g \in G} r_h \cdot h\right) \cdot (m \otimes p) = \sum_{h \in H} \bigl(hm \otimes (r_h \cdot h \cdot p)\bigr)
\]
for $\sum_{h \in H} r_h \cdot h \in R_{\rho|_H}[H]$, $m \in M$, and $p \in P$
and thus obtain a finitely generated projective $R_{\rho|_H}[H]$-module.


\subsection{The Green and Mackey structure of a finite quotient group}%
\label{subsec:The_Green_and_Mackey_structure_of_a_finite_quotient_group}

Fix a (not necessarily finite) group $G$ and a surjective group homomorphism
$\nu \colon G \to F$ onto a finite group $F$.  Let $\cala$ be a $G$-$\IZ$-category.  Fix
$n \in \IZ$. Then we can define a module over the Green functor $\Sw_F$ for $F$ as
follows.  The underlying Mackey functor $M(Z,n) = (M(Z,n)_*,M(Z,n)^*)$ is on an object $X$
which is a finite $F$-set $X$, given by
\[
  M(Z,n)_*(X) = M(Z,n)^*(X) = K_n(\cala(\nu^*X)_{\oplus}),
\]
where $\nu^*X$ is the $G$-set obtained from the $F$-set $X$ by restriction with $\nu$.
The covariant functor $M(Z,n)_*$ sends an $F$-map $f \colon X \to Y$ to  the homomorphism
$K_n(\cala(\nu^*f)_{\oplus}) \colon K_n(\cala(\nu^*X)_{\oplus}) \to
K_n(\cala[\nu^*Y]_{\oplus})$ defined in~\eqref{bfK_cala_on_G-sets}. The contravariant
functor $M(Z,n)_*$ sends an $F$-map $f \colon X \to Y$ to 
$K_n(\res(\nu^*f)_{\oplus}) \colon K_n(\cala[\nu^*Y]_{\oplus}) \to
K_n(\cala(\nu^*X)_{\oplus})$ for the functor of additive categories
$\res(\nu^*f)_{\oplus}$ defined in~\eqref{res(f)_oplus}. For a finite $F$-set $X$ we get a
map of abelian groups

\begin{equation}
\nu^* \colon \Sw_F(X) \to \Sw_G(\nu^*X)
\label{nu_upper_ast_on_Sw}
\end{equation}
by precomposing a covariant functor $M(Z,n)$ from $\tau^F(X)$ to the category of finitely
generated free abelian groups with the obvious functor $\calt^G(\nu^*X) \to \calt^F(X)$
induced by $\nu$.  Now we obtain from~\eqref{pairing_SW_with_K_n}
and~\eqref{nu_upper_ast_on_Sw} for all finite $F$-sets $X$ a pairing
\[
\Sw_F(X) \otimes K_n(\cala(\nu^*X)_{\oplus}) \to  K_n(\cala(\nu^*X)_{\oplus}).
\]
We leave the lengthy but straightforward proof to the reader that these data define a
Mackey functor $\calm$ for $F$ and the structure of a Green-module over $\Sw_F$ on it. This fact
is not surprising in view of
Subsection~\ref{subsec:The_special_case_of_a_ring_as_coefficients}, since the proof in this
special case is well-known.


\typeout{---------------- Section 7: Setting up quotients of $V$ -----------------------}

\section{Infinite covirtually cyclic groups and their finite quotients}%
\label{sec:Infinite_covirtually:_cyclic_groups_and_their_finite_quotients}


\subsection{Basics about infinite covirtually cyclic subgroups}%
\label{subsec:Basics_about_non-compact_virtually_cyclic_subgroups}

A group $V$ is called \emph{covirtually cyclic},
if it contains a normal finite  subgroup $K \subseteq V$ such that $V/K$ is  cyclic.
Equivalently, $V$ is either a finite group or there is an extension
$1 \to K \to V \xrightarrow{\pi_V}  C\to 1$ of groups such that $K$ is finite and $C$ is infinite cyclic.
An infinite covirtually cyclic group is the same as a infinite virtually cyclic group of type I.

Let $V$ be a covirtually cyclic subgroup of $G$ which is infinite. Then it contains a
finite subgroup $K = K_V$ with the property such that any finite subgroup of $G$ is
 contained in $K$.  Note that $K$ is uniquely determined by this property, is a
 characteristic and in particular normal subgroup of $V$, and the quotient
 $C_V := V/K_V$ is infinite cyclic.  

 \begin{definition}[Polarization]\label{def:polarization}
  A \emph{polarization of $V$} is a choice of an element $t \in V$ whose image under the
  projection $\pi = \pi_V \colon V \to C_V$ is a generator.
\end{definition}


 \subsection{Basic definitions and notation}\label{subsec:Basic_definitions_and_notation}

         Fix an integer $M \ge 1$. Let $m$ be the order of the automorphism  $c_t \colon K \xrightarrow{\cong} K$
         given by conjugation with $t$.  Then $t^mkt^{-m} = k$ holds for all
         $k \in K_V$.  Hence the infinite cyclic subgroups $\langle t^{Mm}\rangle$ and $\langle t^{m}\rangle$ of $V$ 
         are  normal and their  intersection with
         $K_V$ is trivial. Define finite groups
         \begin{eqnarray*}
           F
           & = &
           V/\langle t^{Mm} \rangle;
           \\ 
           \widehat{F}
           & = &
           V/\langle t^{m} \rangle.
         \end{eqnarray*}     
         Denote by
         $\nu \colon V \to F$  and
         $\beta \colon F \to \widehat{F}$ the canonical
         projections. We have the exact sequence
         $1 \to \ker(\beta) \xrightarrow{j} F\xrightarrow{\beta} \widehat{F} \to 1$.
        
         Note that the following square  commutes
         \begin{equation}
           \xymatrix@!C= 10em{K_V \rtimes_{c_t} \IZ \ar[r]^-{\cong} \ar[r]^-{\cong} \ar[d]_{\id_F \rtimes p_{Mm}}
             &
             V \ar[d]^{\nu}
             \\
             K_V \rtimes_{c_t} \IZ/Mm\ar[r]^-{\cong}  \ar[d]_{\id_F \rtimes \overline{p}_M}
             &
             F \ar[d]^{\beta}
             \\
             K_V \rtimes_{c_t} \IZ/m \ar[r]^-{\cong} 
             &
             \widehat{F}
           }
           \label{diagram_identifying_semi-direct_products}
         \end{equation}
         where the horizontal maps are the obvious isomorphism given by the elements
         $t \in V$ and $\nu(t) \in F$,
         and $\beta \circ \nu(t) \in \widehat{F}$, and
         $p_{Mm} \colon \IZ \to \IZ/Mm$ and $\overline{p}_M \colon \IZ/Mm \to \IZ/m$ are
         the obvious projections.


 \subsection{Some properties of subgroups}\label{subsec:Some_properties_of_subgroups}

         Consider a subgroup $H$ of $F$ and a prime $q$.  We compute
         \begin{multline*}
           [F : H]
           = \frac{|F|}{|H|}
           = \frac{M\cdot m\cdot |K_V|}{|H|}
           = \frac{M\cdot m\cdot |K_V|}{|\beta(H)| \cdot |j^{-1}(H)|}
           = \frac{M\cdot m\cdot |K_V|\cdot [\widehat{F} : \beta(H)]}{|\widehat{F}| \cdot |j^{-1}(H)|}
           \\
           = \frac{M\cdot m\cdot |K_V|\cdot [\widehat{F} : \beta(H)]}{m\cdot |K_V|\cdot |j^{-1}(H)|}
           =  \frac{M\cdot  [\widehat{F} : \beta(H)]}{|j^{-1}(H)|}.
         \end{multline*}
         This implies
         \begin{eqnarray}
           \label{general_estimate_for_log_q([IZ/Mm:pr(H)])}
          \log_q([F : H])
           & = &
           \log_q\left(\frac{M\cdot  [\widehat{F} : \beta(H)]}{|j^{-1}(H)|}\right)
           \\
           & = &
           \log_q(M) +  \log_q([\widehat{F} : \beta(H)])  - \log_q(|j^{-1}(H)|)
           \nonumber
           \\
           & \ge &
                   \log_q(M)  - \log_q(|j^{-1}(H)|).
                   \nonumber
         \end{eqnarray}

         Given a prime $p$, a finite group $H$ is called \emph{$p$-hyperelementary},
         if it can be written as an extension $1 \to C \to H \to P \to 1$
         for a  cyclic group $C$ and a $p$-group $P$, and 
         and is  called \emph{$p$-elementary}
         if it is isomorphic to $C \times P$ 
         or a  cyclic group $C$ and a $p$-group $P$.
         One can always arrange  that the order of $C$ is prime to $p$.

         \begin{lemma}\label{lem:special_estimate_for_log_q([IZ/Mm:pr(H)])}
           Let $p$ be a prime with $p \not= q$. Let $i \colon K \to F$ be the injective
           group homomorphism given by restricting $\nu \colon V \to F$ to $K$.  Suppose
           that $q$ divides $|i^{-1}(H)|$ and that $H$ is a $p$-hyperelementary group.  Then
           \[
             \log_q([F: H]) \ge \log_q(M).
           \]
           \
         \end{lemma}
\begin{proof}
  Note that for $k \in K_V$ and $y \in \ker(\beta)$ we have $\nu(k)\cdot j(y) = j(y) \cdot \nu(k)$.
  This implies that we get a well-defined group homomorphism
  \[
    i^{-1}(H)  \times j^{-1}(H) \to H, \quad (k,y) \mapsto i(k)  \cdot j(y).
  \]
  It is injective by the following calculation for $(k,y) \in(\nu^{-1}(H) \cap K)  \times j^{-1}(H)$
  using the fact that $\beta \circ i$ is injective
  \begin{multline*}
    e_H = i(k) \cdot j(y)
    \\
    \implies e_{\widehat{F}} = \beta (e_H) =  \beta(i(k) \cdot j(y))
    =  \beta(i(k)) \cdot \beta(j(y))
    = \beta(i(k))  \cdot e_{\widehat{F}} =  \beta(i(k)) 
    \\
    \implies \beta \circ i(k) = e_{\widehat{F}} \implies k = e_K
    \implies k = e_K  \;\text{and} \; j(y) = e_H
    \\
    \implies k = e_K \;\text{and} \;
    y = e_{\ker(\beta)}.
  \end{multline*}
  Since $H$ is $p$-hyperelementary and $p \not= q$, the $q$-Sylow subgroup of $H$ is
  cyclic.  Hence the $q$-Sylow subgroup of $i^{-1}(H) \times j^{-1}(H)$ is cyclic.  Since
  $q$ divides $|i^{-1}(H)|$ by assumption, $q$ does not divide $|j^{-1}(H)|$ and we get
  $\log_q(|j^{-1}(H)|) = 0$. Now apply~\eqref{general_estimate_for_log_q([IZ/Mm:pr(H)])}.
\end{proof}

\begin{lemma}\label{p-hyper_elementary_implies_p_elementary}
  Suppose that $H$ is $p$-hyperelementary and that $i^{-1}(H)$ is a $p$-group.  Then $H$
  is $p$-elementary.
\end{lemma}
\begin{proof}
  By~\cite[Lemma~3.1]{Hambleton-Lueck(2012)} it suffices to show that for every prime $q$
  with $q\mid |H|$ and $q \not = p$ there exists an epimorphism from $H$ onto a
  non-trivial finite cyclic group of $q$-power order. Let $\alpha \colon F \to C_V/(Mm\cdot C_V)$ be
  the canonical projection, whose kernel is $K$. Since $\alpha$ induces an epimorphism
  $H \to \alpha(H)$ onto a finite cyclic group $\alpha(H)$, it suffices to show that the
  $q$-Sylow subgroup of $\alpha(H)$ is non-trivial for any prime $q$ with $q\mid |H|$ and
  $q \not = p$. Choose an element $h \in H$ and an integer $a \ge 1$ with $h \not = e_H$ and
  $h^{q^a} = e_H$.  It suffices to show that $\alpha(h)$ is not the the unit element.
  Suppose the  contrary. Then we can find $x \in i^{-1}(H)$ with $i(x) = h$.  Since $i^{-1}(H)$ is a
  finite $p$-group, we can choose an integer $b \ge 1$ satisfying $x^{p^b} = e_K$. As
  $p \not= q$, we can find integers $\lambda, \mu$ with $\lambda q^a + \mu \cdot p^b =  1$.
  We compute
  \[i(x) = i(x^{\lambda q^a + \mu p^b}) = (i(x)^{q^a})^{\lambda} \cdot i((x^{p^b})^{\mu})
    = (h^{q^a})^{\lambda} \cdot i(e^{\mu}) = e_F \cdot e_F = e_F.
  \]
  Hence $x = e_K$ which implies $h = e_H$, a contradiction.
\end{proof}


 \typeout{---- Section 8: On Nil-terms for non-compact virtually cyclic subgroups -----------}

\section{On Nil-terms for infinite covirtually cyclic subgroups}%
\label{sec:On_Nil-terms_for_infinite_covirtually_cyclic_subgroups}

For the remainder of this section we fix an infinite covirtually cyclic group $V$ with a polarization $t \in V$
and a $V$-$\IZ$-category $\cala$. Let  the covariant functor
\begin{equation}
  Z = Z_{\cala} \colon \SETS{V} \to \addcat
  \label{Z_cala_for_V}
\end{equation}
be the functor coming from the functor  in~\eqref{Z_cala_on_G-sets}
for $G = V$ and $\nu \colon V \to F$ taken from
Subsection~\ref{subsec:Basic_definitions_and_notation}.


\subsection{The basic diagram~\ref{com_diagram_Gamma_Gamma_to_V}}\label{subsec:The_basic_diagram}

In the sequel we denote for two subgroups $H_0$ and $H_1$ of $V$ with $H_0 \subseteq H_1$
by $\pr \colon V/H_0 \to V/H_1$ the $V$-map given by the canonical projection. Given
subgroups $H_0$ and $H_1$ of $V$ and $v \in V$ in with $v^{-1}H_0v \subseteq H_1$, we
denote by $R_{v} \colon V/H_0 \to V/H_1$ the $V$-map sending $v'H_0$ to $v'vH_1$. Note
that for $v_0,v_1 \in V$ with $v_i^{-1}H_0v_i \subseteq H_1$ for $i = 0,1$ we have
$R_{v_0} = R_{v_1}$ if and only if $v_1^{-1}v_0 \in H_1$ holds.

Given $d \in \IZ^{\ge 1} = \{ n \in \IZ \mid n \ge 1\}$, let
$C[d] = d \cdot C \subseteq C$ be the subgroup of index $d$ in $C$ and $V[d]$ be the
preimage of $C[d]$ under the projection $\pi \colon V\to C$. In particular we get
$V[1] = V$.

Let $W\subseteq V$ be a subgroup of $V$ of finite index. Then $W$ itself is an infinite 
covirtually cyclic group. Its  maximal finite  subgroup is
$K_W = K_V \cap W$. Let $d = d(W) \in \IZ^{\ge 1}$ be the natural
number given by the index $d$ of $\pi_V(W) $ in the infinite cyclic group $C$. Then
$W \subseteq V[d]$ and we have $W = V[d]$ if and only if $K_V \subseteq W$, or,
equivalently, $K_W = K_V$ holds.

Fix a polarization $t_W \in W$. Since  the following diagram of $V$-spaces commutes
\[
  \xymatrix{V/K_W \ar[rr]^{R_{t_W}} \ar[dr]_{\pr} & & V/K_W \ar[ld]^{\pr}
    \\
    & V/W & }
\]
we get from $Z(\pr) \colon Z(V/K_W) \to Z(V/W)$ a functor of additive categories
\begin{equation}
  U(W,t_W) \colon Z(V/K_W)_{Z(R_{t_W})}[\IZ] \to Z(V/W).
  \label{U(W,t_W)}
\end{equation}

\begin{lemma}\label{U(W,t_w)_is_an_equivalence}
  The functor $U(W,t_W)$ is for any choice of polarization $t_W \in W$ an equivalence of
  additive categories and induces in particular for every $n \in \IZ$ an isomorphism 
  \[
    K_n(U(W,t_W))\colon K_n(Z(V/K_W)_{Z(R_{t_W})}[\IZ]) \to K_n(Z(V/W)).
  \]
\end{lemma}
\begin{proof} This follows from the
  isomorphism~\eqref{K_n(R_(rho|_H)[H])_cong_K_n(underline(R)(G/H)_oplus)} and the obvious
  fact that the canonical map $K_W \rtimes_{c_{t_W}}\IZ \to W$ is an isomorphism.
\end{proof}

We get for $n \in \IZ$ from the
isomorphism appearing in Lemma~\ref{U(W,t_w)_is_an_equivalence},
the twisted Bass-Heller Swan decomposition applied to
$\cala = Z(V/K_W)$ and $\Phi = Z(R_{t_W})$, see
Theorem~\ref{F_k_and_V_k_versus_I_k_ast_and_i_k_upper_ast}, and the projection onto and
the inclusion of the first Nil-term maps
\begin{eqnarray*}
    s(W,t_W)_n \colon \overline{K}_{n-1}(\Nil(Z(V/K_W),Z(R_{t_W})))
    & \to &
    K_n(Z(V/W));
    \\
    r(W,t_W)_n \colon K_n(Z(V/W))
    & \to &
    \overline{K}_{n-1}(\Nil(Z(V/K_W),Z(R_{t_W}))),
\end{eqnarray*}
satisfying $r(W,t_W)_n \circ s(W,t_W)_n = \id_{\overline{K}_{n-1}(\Nil(Z(V/K_W),Z(R_{t_W})))}$.

Fix a polarization $t \in V$ of $V$. The following diagram commutes by
Theorem~\ref{F_k_and_V_k_versus_I_k_ast_and_i_k_upper_ast} 
\begin{equation}
  \xymatrix@!C=15em{\overline{K}_{n-1}(\Nil(Z(V/K),Z(R_{t^d})))
    \ar[r]^-{\overline{K}_{n-1}(V_d)} \ar[d]^{s(V[d],t^d)_n}
    &
    \overline{K}_{n-1}(\Nil(Z(V/K),Z(R_t)))
    \ar[d]^{s(V,t)_n}
    \\
    K_n(Z(V/V[d])) \ar[r]^-{K_n(Z(\pr))} \ar[d]^{r(V[d],t^d)_n}
    &
    K_n(Z(V/V))  \ar[d]^{r(V,t)_n}
    \\
    \overline{K}_{n-1}(\Nil(Z(V/K),Z(R_{t^d})))
    \ar[r]^-{\overline{K}_{n-1}(V_d)}
    &
    \overline{K}_{n-1}(\Nil(Z(V/K),Z(R_t)))
  }
  \label{com_diagram_V_to_V[d]}
\end{equation}
where $V_d$ is the Verschiebungs operator, see~\eqref{Verschiebung_V_k}.

Let $W\subseteq V$ be a subgroup of $V$ of finite index.
Fix a polarization $t$ of $V$. Then there is $y \in K = K_V$ with $yt^{d} \in W$.
           Fix $y = y(W)\in K$ with $yt^{d} \in W$. Then  the element $yt^{d}$ is a polarization of $W$.
           As the composites $V/K_W\xrightarrow{\pr} V/K \xrightarrow{R_{t^{d}}} V/K$
           and $V/K_W \xrightarrow{R_{yt^{d}}} V/K_W \xrightarrow{\pr} V/K$ agree,
           $Z(\pr) \colon Z(V/K_W) \to Z(V/K)$ induces a map
           \[
           \overline{K}_{n-1}(Z(\pr)) \colon \overline{K}_{n-1}(\Nil(Z(V/K_W),Z(R_{yt^{d}})))
           \to \overline{K}_{n-1}(\Nil(Z(V/K),Z(R_{t^{d}})))
         \]
         for every $n \in \IZ$. One
           easily easily checks that the following diagram commutes
           \begin{equation}
  \xymatrix@!C=18em{\overline{K}_{n-1}(\Nil(Z(V/K_W),Z(R_{yt^{d}})))
    \ar[r]^-{\overline{K}_{n-1}(Z(\pr))} \ar[d]^{s(W,yt^{d})_n}
    &
    \overline{K}_{n-1}(\Nil(Z(V/K),Z(R_{t^{d}})))
    \ar[d]^{s(t,d)_n}
    \\
    K_n(Z(V/W)) \ar[r]^-{K_n(Z(\pr))} \ar[d]^{r(W,yt^{d})_n}
    &
    K_n(Z(V/V[d]))  \ar[d]^{r(t,d)_n}
    \\
    \overline{K}_{n-1}(\Nil(Z(V/K_W),Z(R_{yt^{d}})))
    \ar[r]^-{\overline{K}_{n-1}(Z(\pr))} 
    &
    \overline{K}_{n-1}(\Nil(Z(V/K),Z(R_{t^{d}}))).
  }
 \label{com_diagram_W_to_V[d]}
\end{equation}

Define
\begin{multline}
  \Xi_{n-1}(W,y(W)t^{d(W)})   \colon \overline{K}_{n-1}(\Nil(Z(V/K_W),Z(R_{y(W)t^{d(W)}})))
  \\
  \to
  \overline{K}_{n-1}(\Nil(Z(V/K),Z(R_t)))
 \label{XI_(n-1)(Gamma,yt_upper_d)}
\end{multline}
to be the composite
\begin{multline*}
 \overline{K}_{n-1}(\Nil(Z(V/K_W),Z(R_{y(W)t^{d(W)}})))
 \xrightarrow{{\overline{K}_n(Z(\pr))}}
  \overline{K}_{n-1}(\Nil(Z/K),Z(R_{t^{d(W)}}))
  \\
  \xrightarrow{\overline{K}_{n-1}(V_d)}
  \overline{K}_{n-1}(\Nil(Z(V/K),Z(R_t))).
  \end{multline*}
  Then we obtain by concatenating~\eqref{com_diagram_V_to_V[d]}
  and~\eqref{com_diagram_W_to_V[d]} the commutative diagram
  \begin{equation}
  \xymatrix@!C=20em{\overline{K}_{n-1}(\Nil(Z(V/K_W)),Z(R_{y(W)t^{d(W)}}))
    \ar[r]^-{\Xi_{n-1}(W,y(W)t^{d(W)}) } \ar[d]^{s(W,y(W)t^{d(W)})_n}
    &
    \overline{K}_{n-1}(\Nil(Z(V/K),Z(R_t)))
    \ar[d]^{s(V,t)_n}
    \\
    K_n(Z(V/W)) \ar[r]^-{K_n(Z(\pr))} \ar[d]^{r(W,y(W)t^{d(W)})_n}
    &
    K_n(Z(V/V))  \ar[d]^{r(V,t)_n}
    \\
    \overline{K}_{n-1}(\Nil(Z(V/K_W),Z(R_{y(W)t^{d(W)}})))
    \ar[r]^-{\Xi_{n-1}(W,y(W)t^{d(W)}) } 
    &
    \overline{K}_{n-1}(\Nil(Z(V/K),Z(R_t))).
  }
  \label{com_diagram_Gamma_Gamma_to_V}
\end{equation}


\subsection{Improving the induction results}%
\label{subsec:Improving_the_induction_results}

For the remainder of this section consider the setup of
Subsection~\ref{subsec:Basic_definitions_and_notation}. In particular we have fixed a
natural number $M$ and a surjective group homomorphism $\nu \colon V \to F$ onto a finite
group $F$. Moreover, we fix $n \in \IZ$ and let $M(Z,n)$ be the module over the Green functor
$\Sw_F$ associated to $Z$ in
Subsection~\ref{subsec:The_Green_and_Mackey_structure_of_a_finite_quotient_group} with
respect to the epimorphism $\nu \colon V \to F$.

       \begin{lemma}\label{the:Computabilty_in_terms_of_p-hyperlelementary_subgroups}         
         Consider any element $z$ in $K_n(Z(V/V)) = \calm(Z,n)(F/F)$.
            Let $p$ be a prime number. Let $\calh_{p} $ be the
           family of $p$-hyperelementary subgroups of $F$.

           Then there are elements $a_H \in \IZ_{(p)}$ and $u_H \in \Sw(H)$  for each $H \in \calh_{p}$
           such that
           \[
             z = \sum_{H \in \calh_{p}} a_H \cdot \calm(Z,n)_*(\pr_H)\bigl(u_H \cdot \calm(Z,n)^*(\pr_H)(z)\bigr)
           \]
           holds in $K_n(Z(V/V))_{(p)}$, where $\pr_H \colon F/H \to F/F$ is the projection.
         \end{lemma}
     \begin{proof}
       This follows from Lemma~\ref{lem:general_induction_for_SW}.
     \end{proof}

\begin{lemma}\label{lem:res_is_zero_for_large_k}
  Consider an element $z \in \overline{K}_{n-1}(\Nil(Z(V/K_V),Z(R_t)))$
  of nilpotence degree $\le D$.

  Then for every $H \subseteq F$ satisfying 
  $[F:H] \ge D$ the composite
  \[
   \overline{K}_{n-1}(\Nil(Z(V/K_V),Z(R_t)))
    \xrightarrow{s_n(V,t)} K_n(Z(V/V))  \xrightarrow{\calm(Z,n)^*(\pr)}
   K_n(Z(V/\nu^{-1}(H))) 
\]
sends $z$ to zero.
\end{lemma}
\begin{proof} Put $d = [F : H] = [V : \nu^{-1}(H)]$. Then we have
  $\nu^{-1}(H) \subseteq V[d]$ and the homomorphism
  $\calm(Z,n)^*(\pr) \colon K_n(Z(V/V)) \to K_n(Z(V/\nu^{-1}(H)))$ is the composite
  \[
  K_n(Z(V/V)) \xrightarrow{\calm(Z,n)^*(\pr)} K_n(Z(V/V[d]))
  \xrightarrow{\calm(Z,n)^*(\pr)} K_n(Z(V/\nu^{-1}(H))).
  \]
  The following diagram commutes
  \[
    \xymatrix@!C=14em{K_n(Z(V/K_V)_{Z(R_t)}[\IZ]) \ar[r]^-{K_n(U(V,t))}_-{\cong}  \ar[d]_{K_n(i_d^*)}
    &
      K_n(Z(V/V)) \ar[d]^{\calm(Z,n)^*(\pr)}
      \\
      K_n(Z(V/K_V)_{Z(R_t)^d}[\IZ])
      \ar[r]_-{K_n(U(V,t^d))}^-{\cong}
      &
       K_n(Z(V/V[d]))
    }
  \]
  where the bijective horizontal arrows have been defined in~\eqref{U(W,t_W)}.
  Now the claim follows from Lemma~\ref{lem:NK_n(F_k)(z)_vanishes_for_large_k},
  since the map $s_n(V,t)$ is by definition the composite
  \[\overline{K}_{n-1}(\Nil(Z(V/K),Z(R_t))) \xrightarrow{\overline{s}_n}
    K_n(Z(V/K)_{Z(R_t)}[\IZ])  \xrightarrow{K_n(U(V,t))} K_n(Z(V/V)).
  \]
\end{proof}

Lemma~\ref{the:Computabilty_in_terms_of_p-hyperlelementary_subgroups} and
Lemma~\ref{lem:res_is_zero_for_large_k} imply

\begin{lemma}\label{the:Computabilty_in_terms_of_p-hyperlelementary_subgroups_improved}
  Consider an element $z$ in
  $z \in \overline{K}_{n-1}(\Nil(Z(V/K_V),Z(R_t)))$ of nilpotence
  degree $ \le D$.  Let $p$ be a prime number. Let $\calh_{p} $ be the
  family of $p$-hyperelementary subgroups of $F$.

  Then there are elements $a_H \in \IZ_{(p)}$ and $u_H$ for each
  $H \in \calh_{p}$ such that
  \[
    s_n(V,t)(z) = \sum_{\substack{H \in \calh_{p}\\{[F: H]}< D}} a_H
    \cdot \calm(Z,n)_*(\pr_H)\bigl(u_H \cdot \calm(Z,n)^*(\pr_H) \circ
    s_n(V,t)(z)\bigr)
  \]
  holds in $K_n(Z(V/V))_{(p)}$.
\end{lemma}

Define $\overline{\calh_{p}}$ to be the family
$p$-hyperelementary subgroups $H$ of $F$, for which
$i^{-1}(H) = H \cap F \subseteq F$ is a $p$-group.

\begin{lemma}\label{lem:Computabilty_in_terms_of_p-hyperlelementary_subgroups_final}
  Consider an element $z$ in $\overline{K}_{n-1}(\Nil(Z(V/K_V),Z(R_t)))$ of nilpotence
  degree $\le D$.  Let $p$ be a prime number.  Assume that $\log_q(l) < \log_q(M)$ holds
  for every natural number $l$ with $l < D$ and every prime $q$ that satisfies $q \le \max\{D,|\widehat{F}|\}$
  and is different from $p$.

  Then there are   elements $a_H \in \IZ_{(p)}$ and $u_H$  for each $H \in \overline{\calh_{p}}$
  such that
           \[
            s_n(V,t)(z) = \sum_{H \in \overline{\calh_{p}}}
               a_H \cdot \calm(Z,n)_*(\pr_H)\bigl(u_H \cdot \calm(Z,n)^*(\pr_H)\circ s_n(V,t)(z)\bigr)
           \]
           holds in $K_n(Z(V/V))_{(p)}$.
\end{lemma}
\begin{proof}
  In view of
  Lemma~\ref{the:Computabilty_in_terms_of_p-hyperlelementary_subgroups_improved} it
  suffices to show under the assumptions above that $i^{-1}(H) = H \cap F$ is a $p$-group
  for $H \in \calh_{p}$, provided that $[F:H] < D$ holds. Because of
  Lemma~\ref{lem:special_estimate_for_log_q([IZ/Mm:pr(H)])} it remains
  to show $\log_q([F:H]) < \log_q(M)$ for every every $H \in \calh_{p}$ and every prime $q$ that divides
  $|i^{-1}(H)| = |H \cap F|$ and is different from $p$, provided that $[F:H] < D$ holds.
  If $q$  satisfies $q \le \max\{D,|\widehat{F}|\}$, then $\log_q([F:H]) < \log_q(M)$
  follows from the assumptions.  Suppose $q$ that satisfies $q >
  \max\{D,|\widehat{F}|\}$. Then $q$ does divide neither $[F:H]$ nor $\widehat{F}$. Since
  $q$ divides $ |H \cap F|$ and hence $|F|$ and $|F| = M \cdot |\widehat{F}|$ holds, $q$
  divides $|M|$ and hence $\log_q(M) \ge 1$. As $q$ does not divide  $[F:H]$, we have
  $\log_q([F:H])= 0$ and hence $\log_q([F:H]) < \log_q(M)$.
\end{proof}

 \begin{notation}\label{not:calv_p(V)}
Let $\calv_p(V)$ be the set of infinite covirtually cyclic subgroups $W \subseteq V$ such that
$K_W = W \cap K_V$ is a $p$-group.
\end{notation}
Note that $W \cap K$ is indeed the maximal finite  subgroup $K_{W}$ of $W$. For
$W \in \calv_p(V)$ define the natural number $d(W)$ to be the index
$[C: \pi_V(W)]$, where $\pi_V \colon V \to V/K_V = C_V$ is the projection onto the
infinite cyclic subgroup $C_V$. For every $W \in \calv_p(V)$ choose an element
$y(W) \in K_V$ such that $y(W)t^{d(W)}$ is a polarization of $W$. We have defined
in~\eqref{XI_(n-1)(Gamma,yt_upper_d)}  the map
\begin{multline*}
\Xi_{n-1}(W,y(W)t^{d(W)})   \colon \overline{K}_{n-1}\bigl(\Nil(Z(V/K_{W}),Z(R_{y(W)t^{d(W)}}))\bigr)
\\
\to   \overline{K}_{n-1}\bigl(\Nil(Z(V/K_V),Z(R_t))\bigr).
\end{multline*}
The isomorphism type of the abelian group 
$\overline{K}_{n-1}\bigl(\Nil(Z(V/K_{W}),Z(R_{y(W)t^{d(W)}}))\bigr)$
and the image of $\Xi_{n-1}(W,y(W)t^{d(W)})$ are  independent of the
choice of $y(W)$, since for any other choice $y(W)'$ we have the 
commutative diagram 
\[
  \xymatrix@!C=18em{\overline{K}_{n-1}\bigl(\Nil(Z(V/K_{W}),Z(R_{y(W)t^{d(W)}}))\bigr)
    \ar[dd]^{\overline{K}_{n-1}(\Nil(R_{y(W)'y(W)^{-1}}))}_{\cong} \ar[rd]^-{\quad \Xi_{n-1}(W,y(W)t^{d(W)})}
    &
    \\
    &
    \overline{K}_{n-1}\bigl(\Nil(Z(V/K_V),Z(R_t))\bigr)
    \\
    \overline{K}_{n-1}\bigl(\Nil(Z(V/K_{W}),Z(R_{y(W)'t^{d(W)}}))\bigr)
    \ar[ru]_-{\quad \Xi_{n-1}(W,y(W)'t^{d(W)})}
     &
   }
   \]
   with an isomorphisms as vertical  arrow.

\begin{theorem}\label{the:basic_result_about_induction_for_nil_terms_for_Z}
 The  map
\begin{multline*}
\bigoplus_{W \in \calv_p(V)} \Xi_{n-1}(W,y(W)t^{d(W)})_{(p)} \colon \bigoplus_{W \in \calv_p(V)}
\overline{K}_{n-1}\bigl(\Nil(Z(V/K_{W}),Z(R_{y_{W}t^{d(W)}}))\bigr)_{(p)} 
\\
\to
  \overline{K}_{n-1}\bigl(\Nil(Z(V/K),Z(R_t))\bigr)_{(p)} 
\end{multline*}
is surjective.
\end{theorem}
\begin{proof} Consider $z \in \overline{K}_{n-1}(\Nil(Z(V/K),Z(R_t)))$. Choose $D$ such
  that the nilpotence degree of $z$ is $\le D$.  Now choose a natural number $M$ such that
  $\log_q(l) < \log_q(M)$ holds for every natural number $l$ with $l < D$ and every prime
  $q$ that satisfies $q \le \max\{D,|\widehat{F}|\}$ and is different from $p$.  From
  Lemma~\ref{lem:Computabilty_in_terms_of_p-hyperlelementary_subgroups_final} we get
  elements $a_H \in \IZ_{(p)}$ and $u_H$ for each $H \in \overline{\calh_{p}}$ such that
  \begin{equation}
  s(V,t)_n(z) = \sum_{H \in \overline{\calh_{p}}}
  a_H \cdot \calm(Z,n)_*(\pr_H)\bigl(u_H \cdot \calm(Z,n)^*(\pr_H)\circ s_n(V,t)(z)\bigr)
  \label{value_of_s(V,t)_n(z)}
\end{equation}              
holds in $K_n(Z(V/V))_{(p)}$.
Define for $H \in \overline{\calh_p}$ the subgroup  $W_H$ of $V$ by $W_H := \nu^{-1}(H)$.
For $H \in \overline{\calh_p}$ the following diagram
\begin{equation}
  \xymatrix@!C=19em{K_n(Z(V/W_H)) \ar[r]^-{r(W_H,y(W_H)t^{d(W_H)})_n}    
      \ar[d]_{\calm(Z,n)_*(\pr_H)}
      &
      \overline{K}_{n-1}(\Nil(Z(V/K_{W_H}),Z(R_{y(W_H)t^{d(W_H)}})))
      \ar[d]^{\Xi_{n-1}(W_H,y_{W_H}t^{d(W_H)})} 
      \\
      K_n(Z(V/V)) \ar[r]_-{r(V,t)_n}
      &
      \overline{K}_{n-1}(\Nil(Z(V/K),Z(R_t))).}
    \label{certain_com_diagram}
  \end{equation}
  commutes because of the commutative diagram~\eqref{com_diagram_Gamma_Gamma_to_V}.
  Define the element in
  $z_H \in \overline{K}_{n-1}\bigl(\Nil(Z(V/K_{W}),Z(R_{y_{W}t^{d(W)}}))\bigr)_{(p)}$ by
  \begin{equation}
    z_H =  r(W_H,y_{W_H}t^{d(W_H)})_n\circ \bigl(u_H \cdot \calm(Z,n)^*(\pr_H)\circ s(V,t)_n(z)\bigr).
    \label{z_H}
  \end{equation}
  Recall that we have $r(V,t)_n \circ s(V,T)_n = \id$. Now we compute 
  \begin{eqnarray*}
    z
    & = &
    r(V,t)_n \circ s(V,T)_n(z)
    \\
    & \stackrel{\eqref{value_of_s(V,t)_n(z)}}{=} &
    r(V,T)_n \circ \biggl(\sum_{H \in \overline{\calh_{p}}}
    a_H \cdot \calm(Z,n)_*(\pr_H)\bigl(u_H \cdot \calm(Z,n)^*(\pr_H)\circ s(V,t)_n(z)\bigr)\biggr)
     \\
    & = &
   \sum_{H \in \overline{\calh_{p}}}
   a_H \cdot r(V,T)_n \circ \calm(Z,n)_*(\pr_H)\bigl(u_H \cdot  \calm(Z,n)^*(\pr_H) \circ s(V,t)_n(z)\bigr)      
   \\
    & \stackrel{\eqref{certain_com_diagram}}{=}&
   \sum_{H \in \overline{\calh_{p}}}
          a_H \cdot \Xi_{n-1}(W_H,y_{W_H}t^{d(W_H)}) \circ r(W_H,y(W_H)t^{d(W_H)})_n
    \\
    & & \hspace{50mm} \circ \bigl(u_H \cdot \calm(Z,n)^*(\pr_H)\circ s(V,t)_n(z)\bigr)
    \\
    & \stackrel{\eqref{z_H}}{=} &
   \sum_{H \in \overline{\calh_{p}}}
   a_H \cdot \Xi_{n-1}(W_H,y(W_H)t^{d(W_H)})(z_H).
  \end{eqnarray*}
  Note that for $H \in \overline{\calh_p}$ the subgroup
  $W_H := \nu^{-1}(H)$ of $V$ belongs to $\calv_p(V)$ introduced in
  Notation~\ref{not:calv_p(V)}.  This finishes the proof of Theorem~\ref{the:basic_result_about_induction_for_nil_terms_for_Z}.
\end{proof}

\begin{corollary}\label{cor:vanishing_of_overline(K)_(n-1)(Nil(Z(V/K),Z(R_t)))_(p)}
  If $\overline{K}_{n-1}\bigl(\Nil(Z(V/K_{W}),Z(R_{y_{W}t^{d(W)}}))\bigr)_{(p)} = 0$
  holds for every $W \in \calv_p(V)$,   then we get
  \[
  \overline{K}_{n-1}\bigl(\Nil(Z(V/K),Z(R_t))\bigr)_{(p)} = 0.
  \]
\end{corollary}


\typeout{---- Section 9: Proof  of Theorem~\ref{the:passage_from_Fin_to_Vcyc} and
Theorem~\ref{the:FJ_and_regular} for additive categories -----------}

\section{Proof  of Theorem~\ref{the:passage_from_Fin_to_Vcyc} and
Theorem~\ref{the:FJ_and_regular} for additive categories}%
\label{sec:Proof_of_Theorem_ref(the:passage_from_Fin_to_Vcyc)_and_Corollary}
 
Let $G$ be a group and $\Lambda$ be a commutative ring. Let $\cala$ be an additive
$G$-$\Lambda$-category.  We have defined $\calp(G,\Lambda)$ in Notation~\ref{calp(G,R)}.

\begin{theorem}\label{the:passage_from_Fin_to_Vcyc_additive_categories}
  Suppose that the additive category $\cala$ is regular in the sense
  of~\cite[Definition~6.2]{Bartels-Lueck(2020additive)}. Then the canonical map
  \[
    H_n^G(\EGF{G}{\FIN};\bfK_{\cala}) \to H_n^G(\EGF{G}{\VCYC};\bfK_{\cala})
  \]
  is a $\calp(G,\Lambda)$-isomorphism for all $n \in \IZ$.
\end{theorem}
\begin{proof}
  The relative assembly map
  \[
    H_n^G(\EGF{G}{\CVCYC};\bfK_{\cala}) \to H_n^G(\EGF{G}{\VCYC};\bfK_{\cala})
  \]
  is an isomorphism for all $n \in \IZ$ by~\cite[Remark~1.6]{Davis-Quinn-Reich(2011)}, see
  also~\cite[Theorem~13.44]{Lueck(2022book)}. Hence it suffices to show that
  the relative assembly map
  \[
    H_n^G(\EGF{G}{\FIN};\bfK_{\cala}) \to H_n^G(\EGF{G}{\CVCYC};\bfK_{\cala})
  \]
  is a $\calp(G,\Lambda)$-isomorphism for all $n \in \IZ$.  By the Transitivity Principle
  appearing in~\cite[Theorem~3.3]{Bartels-Echterhoff-Lueck(2008colim)}, see
  also~\cite[Theorem~15.12]{Lueck(2022book)}, we can assume without loss of
  generality that $G$ itself is an infinite covirtually cyclic group $V$.  Hence we have
  to show for any covirtually cyclic group $V$ that the assembly map
  \begin{equation}
    H_n^V(\EGF{V}{\FIN};\bfK_{\cala}) \to H_n^V(\EGF{V}{\CVCYC};\bfK_{\cala})
    = H_n^V(V/V;\bfK_{\cala}) = \pi_n(\bfK_{\cala}(V))
    \label{assembly_map_for_V}
  \end{equation}
  is a $\calp(G,\Lambda)$-isomorphism, where we view $\cala$ as a $V$-$\Lambda$-category
  by restriction from $G$ to $V$.  Write $V = K \rtimes_{c_t} \IZ$. The twisted
  Bass-Heller-Swan isomorphism of~\eqref{twisted_BHS_on_homotopy_groups} applied to
  $V = K \rtimes_{c_t} \IZ$ yields an isomorphism
  \begin{multline*}
  \pi_n(T_{\bfK(c_t)}) \times \overline{K}_{n-1}\bigl(\Nil(Z_{\cala}(V/K),Z(R_t))\bigr)
  \times \overline{K}_{n-1}\bigl(\Nil(Z_{\cala}(V/K),Z(R_t))\bigr)
  \\
\xrightarrow{\cong} \pi_n(\bfK_{\cala}(V)).
\end{multline*}
There is an identification
  \[\pi_n(T_{\bfK(c_t)}) \xrightarrow{\cong}  H_n^V(\EGF{V}{\FIN};\bfK_{\cala})
  \]
  coming from the fact that a model for $\EGF{V}{\FIN}$ is $\IR$ with respect to the
  $V$-action given by $v \cdot r = \pr(v) + r$ for $r \in \IR$ and the canonical
  projection $\pr \colon V \to \IZ$. Under these identifications the
  map~\eqref{assembly_map_for_V} becomes the obvious inclusion.  Hence
  map~\eqref{assembly_map_for_V} is a $\calp(G,\Lambda)$-isomorphism if we can show
 \[
 \IZ[\calp(G,\Lambda)^{-1}] \otimes_{\IZ} \overline{K}_{n-1}\bigl(\Nil(Z_{\cala}(V/K),Z(R_t))\bigr) = 0
\]
for every $n \in \IZ$. Because of
Corollary~\ref{cor:vanishing_of_overline(K)_(n-1)(Nil(Z(V/K),Z(R_t)))_(p)} it suffices to
show that for any $p \not\in \calp(G,\Lambda)$ and any $W \in \calv_p(V)$ we have
$\overline{K}_{n-1}\bigl(\Nil(Z_{\cala}(G/K_{W}),Z(R_{y_{W}t^{d(W)}}))\bigr)= 0$.

Because of~\cite[Lemma~7.7 and Theorem~8.1]{Bartels-Lueck(2020additive)} it suffices to
show that the additive category $Z_{\cala}(G/K_{W})$ is regular. Since it is equivalent to
$\cala|_{K_W}[K_W]$ and $K_W$ is a $p$-group, it suffices to show for any $p$-subgroup
$P$ of $V$ that $\cala|_P[P]$ is regular.  If the $p$-subgroup $P$ is trivial, then
$\cala|_P[P] = \cala$ and hence by assumption regular. If $P$ is non-trivial, then by
definition of $\calp(V,\Lambda)$ the prime $p$ is invertible in $\Lambda$. We leave it to
the reader to check that then $\cala|_P[P]$ is regular, as $\cala$ is
regular. It is not hard to extend the proof
for rings in~\cite[Lemma~7.4~(2)]{Bartels-Lueck(2023forward)}  to additive categories.
 This finishes the proof of
Theorem~\ref{the:passage_from_Fin_to_Vcyc_additive_categories}.
\end{proof}

\begin{corollary}\label{cor:FJ_and_regular_additive_category}
  Suppose $G$ satisfies the Full Farrell-Jones Conjecture. Assume that the additive
  $\Lambda$-category $\cala$ is regular in the sense
  of~\cite[Definition~6.2]{Bartels-Lueck(2020additive)} and that the order of any finite
  subgroup of $G$ is invertible in $\Lambda$.
  
  Then the canonical map
  \[
    \colimunder_{H \in \SubGF{G}{\FIN}} K_0(\cala(G/H)) \to K_0(\cala[G/G])
  \]
  is an isomorphism and
  \[
    K_n(\cala[G/G]) = 0 \quad \text{for} \;n \le -1.
  \]
  where $\cala[G/G]$ has been defined in
  Subsection~\ref{subsec:The_K-theoretic_covariant_Or(G)-spectrum_associated_to_an_G-Z-category}.
\end{corollary}
\begin{proof}
  This follows from Theorem~\ref{the:passage_from_Fin_to_Vcyc_additive_categories} using
  the equivariant Atiyah-Hirzebruch spectra sequence and the fact that $\cala(G/H)$ is
  equivalent to the regular additive category $\cala|_H[H]$ for finite $H \subseteq G$,
  which implies $K_i(\cala(G/H)) = 0$ for $i \le -1$ and $|H| < \infty$, see for
  instance~\cite[Proof of Proposition~13.48~(iv)]{Lueck(2022book)},~\cite[Section~4]{Bartels-Lueck(2023recipes)}.
\end{proof}


\typeout{---- Section 10: Nil-terms as modules over the ring of big Witt vectors -----------}

\section{Nil-terms as modules over the ring of big Witt vectors}%
\label{sec:Nil-terms_as_modules_over_the_ring_of_big_Witt_vectors}

We extend the result of Weibel~\cite[Corollary~3.2]{Weibel(1981)}, which is based on
work by Stienstra~\cite{Stienstra(1982)}, that for a ring $R$ of
characteristic $N$ for some natural number $N$ we have
$N\!K_n(RG)[1/N] = \{0\}$ for every group $G$ and every
$n \in \IZ$, to additive categories allowing twisting by an automorphism.


\subsection{Review of the ring of big Witt vectors}\label{subsec:Review_of_the_ring_of_big_Witt_vectors}

Let $\Lambda$ be a commutative ring. Let $W(\Lambda)$ be the commutative ring of big Witt
vectors.  The underlying abelian group is the multiplicative group $1 + t\Lambda[[t]]$ of
formal power series $1 + \lambda_1t + \lambda_2t^2 + \ldots$ in $t$ with coefficients in
$\Lambda$ and leading term $1$.  We do not give the details of the multiplicative
structure $\ast$ but at least mention that it is the unique continuous functional with the
property that $(1 -\lambda t) \ast (1-\mu t) = (1- \lambda \mu t)$ holds for
$\lambda, \mu \in \Lambda$.  We mention that it satisfies for $m,n \in \IZ^{\ge1}$ and
$\lambda, \mu \in \Lambda$
\begin{equation}
  (1- \lambda t^m) \ast (1- \mu t^n) = (1 - \lambda^{n/d}\mu^{m/d}t^{mn/d})^d,
  \label{formula_for_multiplication_in_W(Lambda)}
\end{equation} 
where $d$ is greatest common divisor of $m$ and $n$. The unit for the addition is $1$,
whereas the unit for the multiplication is $(1-t)$.

For $N \in \IZ^{\ge 1}$ the subgroup $I_N = 1 + t^N\Lambda[[t]]$ is actually an ideal in
$W(\Lambda)$. We have $W(\Lambda) = I_1 \supset I_2 \supset I_3 \supset \cdots$ and
$\bigcap_{N \ge 1} I_N = \{1\}$. Thus we obtain the so called $t$-adic topology on
$W(\Lambda )$.  Then $W(\Lambda)$ is separated and complete in this topology. The quotient
rings $W_N(\Lambda) = W(\Lambda )/I_{N+1}$ are called the rings of $N$-truncated Witt
vectors. For more information we refer to~\cite[Section~I.1]{Bloch(1977)}.


\subsection{Endomorphisms rings}\label{subsec:Endomorphisms_rings}
Let $\End(\Lambda)$ be the exact category of endomorphism of finitely generated projective
$\Lambda$-modules.  Objects are pairs $(P,f)$ consisting of a finitely generated projective
$\Lambda$-module $P$ together with a $\Lambda$-endomorphism $f \colon P \to P$. A morphisms
$g \colon (P_0,f_0) \to (P_1,f_1)$ is a $\Lambda$-homomorphism $g \colon P_0 \to P_1$ satisfying
$f_1 \circ g = g \circ f_0$. We have the inclusion $i \colon \MODcat{\Lambda}_{\fgp}\to \End(\Lambda)$
sending a finitely generated projective $\Lambda$-module $P$, to $(P,0)$.  It has a
retraction $r \colon \End(\Lambda) \to \MODcat{\Lambda}_{\fgp}$ sending $(P,f)$ to $P$.  Let
$K_0(\End(\Lambda))$ and $K_0(\Lambda) = K_0(\MODcat{\Lambda}_{\fgp})$ be the projective class groups
associated to the exact categories $\End(\Lambda)$ and $\MODcat{\Lambda}_{\fgp}$. Define the
reduced projective class group
\begin{equation}
  \overline{K}_0(\End(\Lambda)):= \cok\bigl(K_0(i) \colon K_0(\MODcat{\Lambda}_{\fgp})) \to K_0(\End(\Lambda))\bigr).
  \label{widetilde(K)_0(End(Lambda))}
\end{equation}
Note that then we get from $i$ and $r$ a natural isomorphism
\begin{equation}
  K_0(\End(\Lambda)) \xrightarrow{\cong} K_0(\Lambda) \oplus \overline{K}_0(\End(\Lambda)).
  \label{splitting_of_K_0(End(Lambda))}
\end{equation}
Note that the tensor product over $\Lambda$ induces the structure of a commutative ring on
$K_0(\End(\Lambda))$.  Since the image of $K_0(i)$ is an ideal, the quotient
$\overline{K}_0(\End(\Lambda))$ inherits the structure of a commutative ring.

There is a well-defined map $\eta \colon \overline{K}_0(\End(\Lambda)) \to W(\Lambda)$,
which sends the class of $(P,f)$ for an endomorphisms $f \colon P \to P$ of a finitely
generated projective $\Lambda$-module $P$ to its characteristic polynomial
$\det_{\Lambda}(1 -tf)$.  The next result is due to Almkvist~\cite{Almkvist(1974)}.

\begin{theorem}\label{the:Almkvist}
  We obtain a well-defined injective ring homomorphism
  \[\eta \colon \overline{K}_0(\End(\Lambda)) \to W(\Lambda)
  \]
  whose image consists of all rational functions, i.e., quotients $x/x'$ of polynomials $x,x'$
  in $1 + t \Lambda[t] \subseteq 1 + t \Lambda[[t]]$.
\end{theorem}

The $t$-adic topology on $\overline{K}_0(\End(\Lambda))$ is given by the filtration
$\overline{K}_0(\End(A)) = \eta^{-1}(I_1) \supset \eta^{-1}(I_2) \supset \eta^{-1}(I_3)
\supset \cdots$.  Let $\overline{K}_0(\End(A))^{\widehat{\ }}$ be the completion of
$\overline{K}_0(\End(A))$ with respect to the $t$-adic topology.  Since $W(\Lambda)$ is
separated and complete in the $t$-adic topology and the image of $\eta$ is dense in
$W(\Lambda)$, we conclude from Theorem~\ref{the:Almkvist} that $\eta$ induces an
isomorphism of topological rings
\begin{equation}
\widehat{\eta} \colon \overline{K}_0(\End(A))^{\widehat{\ }}\xrightarrow{\cong} W(\Lambda).
\label{iso_widehat(chi)}
\end{equation}


\subsection{The action on Nil-groups}\label{subsec:The_action_on_Nil-groups}
Let $\cala$ be an additive $\Lambda$-category, i.e., a small category enriched over the
category of $\Lambda$-modules coming with a direct sum $ \oplus$ compatible with the
$\Lambda$-module structures. Let $\Phi \colon \cala \xrightarrow{\cong} \cala$ be an
automorphism of additive $\Lambda$-categories of $\cala$. Denote by $\Nil(\cala;\Phi)$ the
associated Nil-category, which inherits the structure of an exact category. Recall that
then the $K$-groups $K_s(\Nil(\cala;\Phi))$ and $\overline{K}_s(\Nil(\cala;\Phi))$ are
defined for $s \in \IZ$.

Let $\underline{\Lambda}_{\oplus}$ be the additive $\Lambda$-category
obtained from the obvious $\Lambda$-category having precisely one
object by adding finite sums. More precisely, for every  
$m \in \IZ^{\ge}$ we have the object $[m]$ in $\underline{\Lambda}$.
A morphism $C \colon [m] \to [n]$ for $m,n \ge 1$ is a $(m,n)$-matrix
$C = (c_{i,j})$ over $\Lambda$. The object $[0]$ is declared to be the zero object.
Composition is given by matrix
multiplication, whereas the direct sum is given on objects by assigning
to two object $[m]$ and $[n]$ the object $[m+n]$ and on morphisms by the
block sum of matrices.

We define a bilinear functor of additive $\Lambda$-categories
\begin{equation}
F \colon \underline{\Lambda}_{\oplus} \times \cala \to \cala
\label{functor_F}
\end{equation}
as follows. On objects it is defined by sending $([m],A)$ to $\bigoplus_{i = 1}^m A$.
Consider morphisms $C \colon [m] \to [n]$ and $f \colon A \to B$ in $\underline{\Lambda}$
and $\cala$.  Define the morphism
\[
  (C \otimes f) = (C \oplus f)_{i,j} \colon C([m],A) = \bigoplus_{i = 1}^m A
  \to C([n],B) =\bigoplus_{j= 1}^n B
\]
by $(C \oplus f)_{i,j} = c_{i,j} \cdot f \colon A \to B$ for
$(i,j) \in \{1, \ldots, m \} \times \{1, \ldots, n \}$. If we fix
an object $([n],C)$ in $\End(\underline{\Lambda})$, we get a functor of exact categories
$F_{([n],C)} \colon \Nil(\cala;\Phi) \to \Nil(\cala;\Phi)$ by sending an object
$(A,\varphi)$ that is given by the nilpotent endomorphism $\varphi \colon \Phi(A) \to A$ to the
object that is given by the nilpotent endomorphism defined by the composite
\[
  \Phi\left(\bigoplus_{i=1}^n A\right) \xrightarrow{\sigma(n,A)^{-1}} C([n],\Phi(A)) = \bigoplus_{i=1}^n \Phi(A)
\xrightarrow{F(C,\varphi)} C([n],A) = \bigoplus_{i=1}^n A
\]
for the canonical isomorphism $\sigma(n,A) \colon \bigoplus_{i=1}^n \Phi(A)  \xrightarrow{\cong}
  \Phi\left(\bigoplus_{i=1}^n A\right) $.
It induces a homomorphism 
$\overline{K}_n(F_{([n],C)}) \colon \overline{K}_n(\Nil(\cala;\Phi)) \to
\overline{K}_n(\Nil(\cala;\Phi))$ for every $n \in \IZ$. 
Now one easily checks that the
collection of these homomorphisms defines a bilinear pairing of abelian groups for every
$s \in \IZ$
\begin{equation}
  T_s \colon \overline{K}_0(\End(\Lambda))  \times \overline{K}_s(\Nil(\cala;\Phi))  \to \overline{K}_s(\Nil(\cala;\Phi)).
  \label{bilinear_pairing_for_widetilde(K)_0(End(Lambda))_and_overline(K)_n(Nil(cala;Phi))}
\end{equation}

\begin{lemma}\label{pairing_on_completion}
  The pairing $T_s$ of~\eqref{bilinear_pairing_for_widetilde(K)_0(End(Lambda))_and_overline(K)_n(Nil(cala;Phi))}
  extends uniquely to continuous pairing
  \[
    \widehat{T}_s \colon W(\Lambda)   \times \overline{K}_s(\Nil(\cala;\Phi))
    \to \overline{K}_s(\Nil(\cala;\Phi))
\]
for $s \in \IZ$, where we equip $\overline{K}_s(\Nil(\cala;\Phi))$ with the discrete
topology and $W(\Lambda)$ with the $t$-adic topology.

Moreover, this pairing  turns $\overline{K}_s(\Nil(\cala;\Phi))$ into a $W(\Lambda)$-module.
\end{lemma}
\begin{proof} Because of the isomorphism of topological rings~\eqref{iso_widehat(chi)},
  the existence of the extension $\widehat{T}_s$ follows if we can show for every element
  $y \in \overline{K}_s(\Nil(\cala;\Phi))$ that there exists an integer $N$ such that for
  every $x \in \eta^{-1}(I_N)$ we have $T_s(x,y) = y$. This is done as follows.

  In the first step we reduce the claim to the special case $s \ge 1$.
 
  From~\cite[Definition~2.1, Theorem~3.4 and Lemma~6.5]{Lueck-Steimle(2014delooping)}
  applied to the functor sending $(\cala, \Phi)$ to the connective $K$-theory spectrum
  $\bfK(\Nil(\cala,\Phi))$, we get a natural (untwisted) Bass-Heller-Swan isomorphisms
  for $s \in \IZ$
  \begin{multline*}
    K_{s-1}(\Nil(\cala,\Phi)) \oplus K_{s}(\Nil(\cala,\Phi)) \oplus
    N\!K_{s}(\Nil(\cala,\Phi)) \oplus N\!K_{s}(\Nil(\cala,\Phi))
    \\
    \xrightarrow{\cong} K_{s}(\Nil(\cala[\IZ],\Phi[\IZ])).
  \end{multline*}
  Restricting this isomorphism to $K_{s-1}(\Nil(\cala,\Phi))$ yields a (split) injective
  homomorphisms
  $\sigma_s \colon K_{s-1}(\Nil(\cala,\Phi)) \to K_{s}(\Nil(\cala[\IZ],\Phi[\IZ]))$,
  natural in $(\cala, \Phi)$.  It induces a (split) injective homomorphism
  $\overline{\sigma}_s \colon \overline{K}_{s-1}(\Nil(\cala,\Phi)) \to
  \overline{K}_{s}(\Nil(\cala[\IZ],\Phi[\IZ]))$, natural in $(\cala, \Phi)$,
  since we also have the (untwisted) Bass-Heller-Swan isomorphism
  \[
    K_{s-1}(\cala) \oplus K_{s}(\cala) \oplus N\!K_{s}(\cala) \oplus N\!K_{s}(\cala)
    \xrightarrow{\cong} K_{s}(\cala[\IZ]).
 \]
  Now we   obtain for every $s \in \IZ$ a commutative diagram
  \[\xymatrix@!C=18em{\overline{K}_0(\End(\Lambda))  \times \overline{K}_{s-1}(\Nil(\cala;\Phi))
      \ar[r]^-{T_{s-1}(\cala, \Phi)}
      \ar[d]_{\id_{\overline{K}_0(\End(\Lambda))} \otimes \sigma_s}
      &
      \overline{K}_s(\Nil(\cala;\Phi))
      \ar[d]_{\sigma_s}
      \\
      \overline{K}_0(\End(\Lambda))  \times \overline{K}_{s}(\Nil(\cala[\IZ];\Phi[\IZ]))
      \ar[r]^-{T_s(\cala[\IZ], \Phi[\IZ])}
      &
      \overline{K}_s(\Nil(\cala[\IZ];\Phi[\IZ])).
    }
  \]
  Since for every integer $s$ we can find a natural number $k$ with $s + k \ge 1$
  and we can iterate the construction of the diagram above $k$-times, it suffices to show for every element
  $y \in \overline{K}_s(\Nil(\cala;\Phi))$ for $s \ge 1$ that there exists an integer $N$ such that for
  every   $x \in \eta^{-1}(I_N)$ we have $T_s(x,y) = y$.  Moreover, we can work for the rest of the proof
  with the connective $K$-theory spectrum of  Waldhausen categories.

  Let $\Nil(\cala;\Phi)_N$ be the full exact subcategory of $\Nil(\cala;\Phi)$ consisting
  of those objects $(A,\varphi)$, whose nilpotence degree is less or equal to $N$, i.e.,
  $\varphi^{(N)} \colon \Phi^N(A) \to A$ is trivial.  For every natural number $N$ the
  construction of the pairing $T_s$
  of~\eqref{bilinear_pairing_for_widetilde(K)_0(End(Lambda))_and_overline(K)_n(Nil(cala;Phi))}
  yields also a pairing 
    
  \begin{equation}
    T[N]_s \colon \overline{K}_0(\End(\Lambda))  \times \overline{K}_s(\Nil(\cala;\Phi)_N)
    \to \overline{K}_s(\Nil(\cala;\Phi)).
    \label{bilinear_pairing_for_widetilde(K)_0(End(Lambda))_and_overline(K)_n(Nil(cala;Phi)_N)}
  \end{equation}
  These pairings $T[N]_s$ and $T_s$ are compatible with the map 
  $\overline{K}_s(\Nil(\cala;\Phi)_N) \to \overline{K}_s(\Nil(\cala;\Phi))$ coming from
  the obvious inclusions of full subcategories. In view of the
  isomorphism~\eqref{K_n(NIL)_as_colimit_over_K_n(NIL_D)},
  it suffices to show for every
  natural number $N$ and $s \ge 1$ that $T[N]_s(x,y) = 0$ holds for $x \in \eta^{-1}(I_N)$ and
  $y \in \overline{K}_n(\Nil(\cala;\Phi)_N)$.

  For a natural number $n \ge 1$ and 
  $\lambda_1, \lambda_2, \dots, \lambda_n$ in $\Lambda$, define the automorphism
  $C_n(\lambda_1, \lambda_2, \ldots, \lambda_r;\varphi) \colon [n] \to [n]$ in
  $\End(\Lambda)$ by the $(n,n)$-matrix over $\Lambda$
\[
  C_n(\lambda_1, \lambda_2, \ldots, \lambda_n)
  =
  \begin{pmatrix}
      0 & 0 & 0 & \cdots &  0 & 0 & \lambda_n
    \\
    1 & 0 & 0  & \cdots & 0 & 0 & \lambda_{n-1}
    \\
    0 & 1 & 0  & \cdots & 0 & 0 & \lambda_{n-2}
    \\
    0 & 0 & 1 & \cdots & 0 & 0 & \lambda_{n-2}
    \\
    0 & 0 &  0 & \cdots & 0 & 0 & \lambda_{n-3}
    \\
    \vdots & \vdots & \vdots &  \ddots& \vdots & \vdots & \vdots
    \\
    0 & 0 &  0 & \cdots & 0 & 0 & \lambda_3
    \\
    0 & 0 &  0 & \cdots & 1 & 0 & \lambda_2
    \\
    0 & 0 &  0 & \cdots & 0 & 1 & \lambda_1
  \end{pmatrix}.
\]
Then
$\eta\bigl([C_n(-\lambda_1, -\lambda_2, \ldots -\lambda_n)]\bigr) = 1 +\lambda_1 t +
\lambda\cdot t^2 + \cdots + \lambda_n \cdot t^n$.  Because of Theorem~\ref{the:Almkvist}
it suffices to prove $T[N]_n([C_n(\lambda_1, \lambda_2, \ldots \lambda_n)],y) = 0$
for every $s,N  \ge 1$, elements $\lambda_1, \lambda_2, \ldots, \lambda_n$ in
$\Lambda$  and $y \in \overline{K}_s(\Nil(\cala;\Phi)_N)$, provided that 
$\lambda_i = 0$ holds for $i = 1,2 \ldots, N-1$.

Fix $s,N  \ge 1$ and $\lambda_1, \lambda_2, \ldots, \lambda_n$ in $\Lambda$ 
such that  $\lambda_i = 0$ holds for $i = 1,2 \ldots, N-1$.
In the sequel we use the notation of~\cite[Section~8]{Lueck-Steimle(2016BHS)}.
We have defined the functor of Walhausen categories
\[
\chi_{\Phi} \colon\Nil(\cala;\Phi) \to \Chcat(\cala_\Phi[t^{-1}])^w,
\]
in~\eqref{chi_(Phi))}.  It induces an isomorphism
 \begin{equation}
 K_s(\chi_{\Phi}) \colon K_s(\Nil(\cala,\Phi)) \xrightarrow{\cong} K_s(\Chcat(\cala_\Phi[t^{-1}])^w)
 \label{Iso_K_s(chi)}
\end{equation}
for $s \ge 1$, see~\cite[Theorem~8.1]{Lueck-Steimle(2016BHS)}.  Note that
in~\cite[Theorem~8.1]{Lueck-Steimle(2016BHS)} the category $\cala$ is assumed to be
idempotent complete but this does not matter, since the maps
$K_s(\Nil(\cala,\Phi)) \to K_s(\Nil(\Idem(\cala),\Idem(\Phi)))$ and
$K_s(\Chcat(\cala_\Phi[t^{-1}])^w) \to K_s(\Chcat(\Idem(\cala)_{\Idem(\Phi)[t^{-1}]})^w)$
  induced by the inclusions are isomorphisms for $s \ge 1$ by the usual cofinality
  argument, see for instance~\cite[page~225]{Lueck-Steimle(2014delooping)},
  where $\Idem$ denotes idempotent completion.

Given a natural number  $n \ge 1$, elements $\lambda_1, \lambda_1, \ldots, \lambda_n \in \Lambda$, and
      a nilpotent endomorphism $\varphi \colon \Phi(A) \to A$ in $\cala$, define
      the  morphism in $\cala_{\Phi}[t^{-1}]$
      \[
        U_{\lambda_1,\lambda_2, \ldots, \lambda_n}(\varphi)  \colon \Phi(A)^n \to A^n 
      \]
      by  the $(n,n)$-matrix
      \[
    \begin{pmatrix} \id_{A} \cdot t^{-1} & 0 & 0 & \cdots & 0 & 0 & 0 & -(\lambda_n \cdot  \varphi) \cdot t^0
           \\
           -\varphi \cdot t^0 & \id_{A} \cdot t^{-1} & 0 &  \cdots & 0 & 0 & 0 &   -(\lambda_{n-1} \cdot \varphi) \cdot t^0
           \\
           0  & -\varphi \cdot t^0 & \id_{A} \cdot t^{-1} & \cdots & 0 & 0 & 0 & -(\lambda_{n-2} \cdot \varphi) \cdot t^0
           \\
           \vdots & \vdots & \vdots &  \ddots &\vdots & \vdots & \vdots & \vdots
           \\
           0  & 0 & 0  &\cdots & -\varphi \cdot t^0 & \id_{A} \cdot t^{-1} & 0 & -(\lambda_{3} \cdot\varphi) \cdot t^0
           \\
           0  & 0 & 0 &\cdots & 0 & -\varphi \cdot t^0 & \id_{A} \cdot t^{-1} & - (\lambda_{2} \cdot\varphi) \cdot t^0
           \\
             0  & 0 & 0  &\cdots & 0 & 0 & - \varphi \cdot t^0 & \id_{A} \cdot t^{-1} -(\lambda_{1} \cdot \varphi) \cdot t^0
           \end{pmatrix}.
         \]
         Thus we obtain a functor of Waldhausen categories
         \[
        U_{\lambda_1,\lambda_2, \ldots, \lambda_n} \colon \Nil(\cala,\Phi)\to \Chcat(\cala_\Phi[t^{-1}])^w,
      \]
      which induces a homomorphism
      \[
        K_s(U_{\lambda_1,\lambda_2, \ldots, \lambda_n}) \colon K_s(\Nil(\cala,\Phi))\to K_s(\Chcat(\cala_\Phi[t^{-1}])^w).
      \]
      This map can easily be identified with the map
      sending $y \in K_s(\Nil(\cala,\Phi))$ to the image of
      $T_s([C_n(\lambda_1, \lambda_2, \ldots \lambda_n)],y) = 0$ under the injective homomorphism~\eqref{Iso_K_s(chi)}.
      Hence it suffices to show that the image of the composite
          \[
            K_s(\Nil(\cala,\Phi)_N ) \xrightarrow{K_s(J)} K_s(\Nil(\cala,\Phi))
            \xrightarrow{K_s(U_{\lambda_1,\lambda_2, \ldots, \lambda_n})}
            K_s(\Chcat(\cala_\Phi[t^{-1}])^w)
          \]
          is contained in the image of 
          \[
             K_s(\cala) \xrightarrow{K_s(I)} K_s(\Nil(\cala,\Phi)) \xrightarrow{K_s(\chi_{\Phi})}
            K_s(\Chcat(\cala_\Phi[t^{-1}])^w),
          \]
          where $J$ is the obvious inclusion, and $I$ sends an object $A$ to $(A,0)$.

          Given a natural number  $n \ge 1$  and
      a nilpotent endomorphism $\varphi \colon \Phi(A) \to A$ in $\cala$, define
      the automorphism in $\cala_{\Phi}[t,t^{-1}]$
      \[
        E_n(\varphi)  \colon A^n \xrightarrow{\cong} A^n
      \]
      by  the $(n,n)$-matrix
         \[
    \begin{pmatrix} \id_{\Phi(A)} \cdot t^0& 0 &  0 & \cdots & 0 & 0 & 0 & 0
           \\
           \varphi \cdot t  &  \id_{\Phi(A)} \cdot t^0& 0  & \cdots & 0 & 0 & 0 &   0
           \\
           (\varphi \cdot t)^2  & \varphi \cdot t &  \id_{\Phi(A)} \cdot t^0 &\cdots & 0 & 0 & 0 & 0
           \\
           \vdots & \vdots & \vdots &  \ddots &\vdots & \vdots & \vdots & \vdots
           \\
           (\varphi \cdot t)^{n-3}  & (\varphi \cdot t)^{n-4} & (\varphi \cdot t)^{n-5}   &\cdots & \varphi \cdot t &  \id_{\Phi(A)} \cdot t^0& 0 & 0
           \\
           (\varphi \cdot t)^{n-2}  & (\varphi \cdot t)^{n-3} & (\varphi \cdot t)^{n-4}   &\cdots &  (\varphi \cdot t)^2 & \varphi \cdot t &  \id_{\Phi(A)} \cdot t^0 & 0
           \\
           (\varphi \cdot t)^{n-1}  & (\varphi \cdot t)^{n-2}  & (\varphi \cdot t)^{n-3}   &\cdots & (\varphi \cdot t)^3  & (\varphi \cdot t)^2
           & \varphi \cdot t & \id_{\Phi(A)} \cdot t^0
           \end{pmatrix}.
         \]
         Then we get in  $\cala_{\Phi}[t,t^{-1}]$  that the composite
         $F_n(\varphi) := E_n(\varphi)  \circ U_{\lambda_1,\lambda_2, \ldots, \lambda_n}(\varphi) 
          \colon \Phi(A)^n \to A^n$
         is given by the matrix
     \[
          \begin{pmatrix} \id_{A} \cdot t^{-1} & 0 & 0 &\cdots & 0 & 0 & 0 & v_1(\varphi)
           \\
          0  & \id_{A} \cdot t^{-1} & 0 &  \cdots & 0 & 0 & 0 &   v_2(\varphi)
           \\
           0  &0& \id_{A} \cdot t^{-1} & \cdots & 0 & 0 & 0 & v_3(\varphi)
           \\
           \vdots & \vdots & \vdots &   \ddots &\vdots & \vdots & \vdots & \vdots
           \\
           0  & 0 & 0 & \cdots & 0 & \id_{A} \cdot t^{-1} & 0 & v_{n-2}(\varphi)
           \\
           0  & 0 & 0 & \cdots & 0 & 0 & \id_{A} \cdot t^{-1} & v_{n-1}(\varphi)
           \\
             0  & 0 & 0 & \cdots & 0 & 0 & 0 & \id_{A} \cdot t^{-1} -v_n(\varphi)
           \end{pmatrix}
         \]
         where we define for $k = 1,2, , \ldots n$
         \[
           v_n(\varphi) = \sum_{i = 0}^{k-1} (\varphi \cdot t)^{i} \circ \bigl((\lambda_{n+1-k +i} \cdot \varphi) \cdot t^0\bigr).
          \]
          Note that $E_n(\varphi)$ does not necessarily live in $\cala_{\Phi}[t^{-1}]$. However the composite
          \[
            \widehat{E}_n(\varphi) := (\id_{\Phi^{1-n}(A)} \cdot t^{1-n})^n \circ E_n(\varphi) \colon A^n \to \Phi^{1-n}(A)
          \]
          does, where for any morphisms $\psi \colon A  \to B$
          in $\cala_{\Phi}[t^{-1}]$ we denote by $\psi^n$ the morphism $\bigoplus_{j=1}^n \psi \colon
          A^n = \bigoplus_{j=1}^n A\to B^n = \bigoplus_{j=1}^n B$ and $\id_{\Phi^{1-n}(A)} \cdot t^{1-n}$ is a morphism
        from $A \to \Phi^{1-n}(A)$ in $\cala_{\Phi}[t^{-1}]$. So the composite
        \[F_n(\varphi) := \widehat{E}_n(\varphi)  \circ U_{\lambda_1,\lambda_2, \ldots, \lambda_n}(\varphi) \colon 
        \Phi(A)^n \to \Phi^{1-n}(A)^n
      \]
      of morphisms in $\cala_{\Phi}[t^{-1}]$ is  given by the matrix
        \[
          \begin{pmatrix} \id_{\Phi^{1-n}(A)} \cdot t^{-n} & 0 & \cdots   & 0 & w_1(\varphi)
           \\
          0  & \id_{\Phi^{1-n}(A)} \cdot t^{-n} &   \cdots  & 0 &   w_2(\varphi)
           \\
           \vdots & \vdots &   \ddots   & \vdots & \vdots
           \\
           0  & 0 &\cdots   & \id_{\Phi^{1-n}(A)} \cdot t^{-n} & w_{n-1}(\varphi)
           \\
             0  & 0 & \cdots   & 0 & \id_{\Phi^{1-n}(A)} \cdot t^{-n} -w_n(\varphi)
           \end{pmatrix}
         \]
       where we define for $k = 1,2, , \ldots n$
         \[
           w_k(\varphi) = \id_{\Phi^{1-n}(A)} \cdot t^{1-n} \circ v_n =
         \sum_{i = 0}^{k-1} \lambda_{n+1-k +i}  \cdot (\id_{\Phi^{1-n}(A)}\cdot t^{i+1-n}) \circ (\Phi^{-i}(\varphi^{(i)}) \cdot t^0)
         \]
         for the morphism $\varphi^{(i)} = \varphi \circ \Phi(\varphi) \circ \cdots \circ \Phi^i(\varphi) \colon \Phi^{i+1}(A) \to A$
         in $\cala$.
       
     Note  that we get functors of Waldhausen categories
         \[
        \widehat{E}_n \colon \Nil(\cala,\Phi)\to \Chcat(\cala_\Phi[t^{-1}])^w
      \]
         sending an object $(A,\varphi)$ to $\widehat{E}_n(\varphi)$ and
         \[
        \widehat{F}_n \colon \Nil(\cala,\Phi)\to \Chcat(\cala_\Phi[t^{-1}])^w
      \]
      sending an object $(A,\varphi)$ to $\widehat{F}_n(\varphi)$. One easily checks using
      Additivity and suitable exact sequences of $\cala_{\Phi}[t,t^{-1}]$- chain complexes
      that we get the equality of morphisms
      $K_s(\Nil(\cala,\Phi))\to K_s(\Chcat(\cala_\Phi[t^{-1}])^w)$ 
      \[
        K_s(\widehat{F}_n) = K_s(U_{\lambda_1,\lambda_2, \ldots, \lambda_n}) +
        K_s(\widehat{E}_n).
      \]
      Note that $w_n(\varphi) = 0$ if $\varphi^{(N)} = 0$, since
      $\lambda_1 = \lambda_{N-1} = 0$ is assumed. Hence $\widehat{F}_n \circ J$ is given
      by a lower triangular matrix, all whose entries on the diagonal are given by
      $\id_{\Phi^{1-n}(A)} \cdot t^{-n}$. The matrix $\widehat{E}_n(\varphi)$ is given by
      a lower triangular matrix, all whose entries on the diagonal are given by
      $\id_{\Phi^{-n}(A)} \cdot t^{-n}$. This implies that the image of both
      $K_s(\widehat{E}_n) \circ K_s(J)$ and $K_s(\widehat{F}_n) \circ K_s(J) $ lie in the image of
      \[
        K_s(\cala) \xrightarrow{K_s(I)} K_s(\Nil(\cala,\Phi)) \xrightarrow{K_s(\chi_{\Phi})}
        K_s(\Chcat(\cala_\Phi[t^{-1}])^w).
      \]
      Hence the same statement is true for
      $K_s(U_{\lambda_1,\lambda_2, \ldots, \lambda_n}) \circ K_s(J)$. This finishes the
      proof that the pairing $T_s$
      of~\eqref{bilinear_pairing_for_widetilde(K)_0(End(Lambda))_and_overline(K)_n(Nil(cala;Phi))}
      extends to a continuous pairing
  \[
  \widehat{T}_s \colon W(\Lambda)   \times \overline{K}_s(\Nil(\cala;\Phi))  \to \overline{K}_s(\Nil(\cala;\Phi))
\]
for $s \in \IZ$. 

Since the image of $\eta$ is dense in $W(R)$, this continuous extension is unique.

  Obviously the pairing $T_n$ turns $\overline{K}_n(\Nil(\cala;\Phi))$ into a module over the commutative ring
  $\overline{K}_0(\End(\Lambda))$.
  We conclude from  the uniqueness of the extension $\widehat{T}_n$ that
  it turns $\overline{K}_n(\Nil(\cala;\Phi))$ into a module over the commutative ring $W(\Lambda)$.

  This finishes  the proof of
Lemma~\ref{pairing_on_completion}.
     \end{proof}


\subsection{Consequences of the $W(\Lambda)$-module structure on the Nil-terms}%
\label{subsec:Consequences_of_the_W(Lambda)-module_structure_on_the_Nil-terms}

\begin{theorem}\label{the:Consequences_of_the_W(Lambda)-module_structure_on_the_Nil-terms}
  Let $N$ be a natural number. Let $\Lambda$ be a commutative ring with unit $1_{\Lambda}$. Consider 
  an additive $\Lambda$-category $\cala$ together with an automorphism $\Phi \colon \cala \xrightarrow{\cong} \cala$
  of additive $\Lambda$-categories. Let $\calp$ be a set of primes. Then we get for every $n \in \IZ$:

  \begin{enumerate}
  \item\label{the:Consequences_of_the_W(Lambda)-module_structure_on_the_Nil-terms:S_upper_1Z}
    If $\Lambda$ is an $\IZ[\calp^{-1}]$-algebra, then $\overline{K}_n(\Nil(\cala;\Phi))$ is an $\IZ[\calp^{-1}]$-module;
  \item\label{the:Consequences_of_the_W(Lambda)-module_structure_on_the_Nil-terms_Z_p}
    If $\Lambda$ is a $\IZ_p$-algebra, then $\overline{K}_n(\Nil(\cala;\Phi))$ is a $\IZ_p$-module;
  \item\label{the:Consequences_of_the_W(Lambda)-module_structure_on_the_Nil-terms:characteristic}
  If $N \cdot 1_{\Lambda}$ is zero in $\Lambda$, then $\overline{K}_n(\Nil(\cala;\Phi))[1/N]$ vanishes.
  \end{enumerate}
\end{theorem}
\begin{proof} Using Lemma~\ref{pairing_on_completion} the proof in Weibel~\cite[Corollary~3.3]{Weibel(1981)}
  extends to our setting,
  see also~\cite[Lemma~3.10]{Bartels-Lueck-Reich(2008appl)}.
\end{proof}


 \typeout{---- Section 11: The Farrell-Jones Conjecture at  the prime $p$ ----}

\section{The Farrell-Jones Conjecture for totally disconnected groups at  the prime $p$}%
\label{sec:The_Farrell-Jones_Conjecture_for_totally_disconnected_groups_at_the_prime_p}

The goal of this section is to prove
Theorem~\ref{the:passage_from_COP_to_CVcyc_in_characteristic_N}.  It will follow from
Theorem~\ref{the:generalized_FJC}, which confirms a version of the $\COP$-Farrell-Jones
Conjecture that is more general than the one already treated
in~\cite{Bartels-Lueck(2023K-theory_red_p-adic_groups)}.

Throughout this section $G$ is a \emph{td-group}, i.e., a locally compact second countable
totally disconnected topological Hausdorff group.


 \subsection{Various sets of primes}\label{subsec:Various_stes_of_prime}

 We need the following sets of primes.

 \begin{notation}[$\calp(G)$]\label{calp(G)} Let $G$ be a td-group.  Define $\calp(G)$ to
  be the set of primes $q$, for which there exist compact open subgroups $U'$ and $U$ of
  $G$ such that $U' \subseteq U$ holds and $q$ divides the index $[U:U']$.
\end{notation}

If $G$ is a compact td-group, then $\calp(G)$ is the set of all primes $q$ for which there
exists a compact open subgroup $U$ of $G$ such that $q$ divides  $[G:U]$.
We have $\calp(\widetilde{U}) = \{p\}$ for the compact open subgroup
$\widetilde{U} \subseteq Q$ appearing in
Assumption~\ref{ass:existence_of_widetilde(U)_G-Z-categories}, provided that
$\widetilde{U}$ is non-trivial.

\begin{lemma}\label{lem:cal_and_subgroups}
  Let $G'$ be a (not necessarily open or compact) closed subgroup of the td-group $G$.  Then
  $\calp(G') \subseteq \calp(G)$.
\end{lemma}
\begin{proof}
  Consider compact open subgroups $L'$ and $L$ of $G'$ satisfying $L' \subseteq L$.  Since
  $L$ is a compact subgroup of  $G$, we can find a compact open subgroup $L_0 \subseteq G$ with
  $L \subseteq L_0$, see~\cite[Lemma~2.3]{Bartels-Lueck(2023K-theory_red_p-adic_groups)}.
  Since $L'$ is open in $L$, we can find an open subset
  $V \subseteq L_0$ with $L' = V \cap L$. Since $L'$ is a compact subgroup of  $L_0 $ and contained in $V$,
  we can find a compact  open  subgroup $L'_0$ of $L_0$ such that $L' \subseteq L_0' \subseteq V $
  holds, see~\cite[Lemma~2.3]{Bartels-Lueck(2023K-theory_red_p-adic_groups)}.
  Hence we get $L'_0\cap L = L'$. This implies that the obvious map $L/L' \to L_0/L_0'$ is
  injective and hence $[L:L']$ divides $[L_0 : L_0']$.  Since $L_0$ and $L_0'$ are compact
  open subgroups of $G$, any prime $q$ that divides $[L:L']$ divides $[L_0 : L_0']$ and
  hence belongs to $\calp(G)$. This implies $\calp(G') \subseteq \calp(G)$.
\end{proof}

\begin{remark}\label{rem:finiteness_of_calp(G)}
  Suppose  that $G$ has up to conjugacy only finitely many maximal compact open subgroups $K_1$, $K_2$,
  \ldots,  $K_n$. (This assumption is satisfied for every reductive $p$-adic group.)
  Suppose  that Assumption~\ref{ass:existence_of_widetilde(U)_G-Z-categories} is satisfied and 
  let $\widetilde{U}$ be the compact open subgroup appearing in
  Assumption~\ref{ass:existence_of_widetilde(U)_G-Z-categories}. For $i = 1,2, \ldots n$
  the subgroup $K_i \cap \widetilde{U}$ has finite index in $K_i$
  and we define $\calp_i$ to be the finite set consisting of those
  primes that divide the index $[K_i : (K_i \cap \widetilde{U})]$. One easily checks that
  $\calp(K_i) \subseteq \calp(K_i \cap \widetilde{U}) \cup \calp_i$ holds.
  Since $\calp(\widetilde{U}) \subseteq \{p\}$, we conclude from Lemma~\ref{lem:cal_and_subgroups}
  \[
  \calp(G) \subseteq \bigcup_{i = 1}^n \calp(K_i) \subseteq \{p\} \cup \bigcup_{i = 1}^n \calp_i.
\]
In particular $\calp(G)$ is a finite set.
\end{remark}

The notion of a \emph{Hecke category with $G$-support} is introduced
in~\cite[Definition~5.1]{Bartels-Lueck(2023foundations)}.
Given a Hecke category with $G$-support and a open subgroup $U \subseteq G$,
one obtains by restriction to $U$ the Hecke category with $U$-support $\calb|_U$,
see~\cite [Notation~5.4]{Bartels-Lueck(2023foundations)}. For the notions of a uniform regular or $l$-uniform regular
additive category we refer to~\cite[Definition~6.2]{Bartels-Lueck(2020additive)}.

\begin{notation}[$\calp(\calb)$]\label{not:calp(calb)}
  Let $\calb$ be Hecke category with $G$-support. Let $\calp(\calb)$ be the largest set of
  primes with the property that for any $d \in \IZ$ with $d \ge 0$ there is $l(d) \in \IZ$
  with $l(d) \ge 0$ such that for every compact open subgroup $U$ of $G$ with
  $\calp(U) \subseteq \calp(\calb)$ the category $(\calb_U)_\oplus[\IZ^d]$ is
  $l(d)$-uniformly regular coherent.
\end{notation}

\begin{notation}[$\calp(R)$]\label{not:calp(R)}
   For a ring $R$ let $\calp(R)$ be the set of primes which are invertible in $R$. 
\end{notation}

\begin{remark}\label{rem:calp(underline(R))}
  If the ring $R$ satisfies $\IQ \subseteq R$, or, equivalently, that $\calp(R)$ consists
  of all primes, then the Hecke algebra $\calh(G,R,\rho,\omega)$ and the Hecke category
  with $Q$-support $\calb(Q,R,\rho,\omega)$ are introduced
  in~\cite[Section~2.2]{Bartels-Lueck(2023forward)} and~\cite[Sections~6.B and
  6.C]{Bartels-Lueck(2023foundations)} for any $G$, $N$, $\pr \colon G \to Q$, $\rho$, and
  $\omega$ as appearing in~\cite[Section~2.1]{Bartels-Lueck(2023forward)} and
  in~\cite[Section~6.A]{Bartels-Lueck(2023foundations)}. We will assume that
  $N \subseteq G$ is \emph{locally central}, i.e., its centralizer in $G$ is an open subgroup of $G$.
  If we replace  the condition
  $\IQ \subseteq R$  by 
  Assumption~\ref{ass:existence_of_widetilde(U)_G-Z-categories} applied in the case where $G$ is replaced by $Q$,
  then   $\calh(G,R,\rho,\omega)$ and $\calb(G,R,\rho,\omega)$ are still defined.
  Recall that Assumption~\ref{ass:existence_of_widetilde(U)_G-Z-categories}
  is satisfied if there exists a prime $p$ such that $p$ is invertible in $R$ and
  $Q$ is a subgroup of a reductive $p$-adic group,
  see~\cite[Lemma~1.1]{Meyer-Solleveld(2010)} and Lemma~\ref{lem:cal_and_subgroups}.

  Now suppose that $R$ is uniformly regular  and that
  Assumption~\ref{ass:existence_of_widetilde(U)_G-Z-categories}, is satisfied.  Then
  $\calp(R)$ is contained in $\calp(\calb(G,R,\rho,\omega))$.  The proof of this fact is
  the same as the one of~\cite[Theorem~7.2]{Bartels-Lueck(2023forward)}, one just has
  to observe that~\cite[Lemma~7.4~(3)]{Bartels-Lueck(2023forward)} still holds if one
  replaces the condition $\IQ \subseteq R$ by the condition that the order of the finite
  group $|D|$ is invertible in $R$, and that $\Idem(\calb(G,R,\rho,\omega)_{\oplus})$ is
  equivalent to $\Idem(\underline{\calh(G,R,\rho,\omega)}_{\oplus})$,
  see~\cite[Lemma~6.6]{Bartels-Lueck(2023foundations)}.
\end{remark}

Given a set $\calp$ of primes, a map of abelian groups is called a
\emph{$\calp$-isomorphism} if it becomes an isomorphism after inverting every element in
$\calp$.  A map of spectra $\bff \colon \bfE \to \bfF$ is called a \emph{weak
  $\calp$-homotopy equivalence} if $\pi_n(F) \colon \pi_n(\bfE) \to \pi_n(\bfF)$ is a
$\calp$-isomorphism for all $n \in \IZ$. If $\calp$ is empty, this is of course the same
as a weak homotopy equivalence.


\subsection{The Farrell-Jones Conjecture in prime characteristic}%
\label{subsec:The_Farrell-Jones_Conjecture_at_the_prime_p}
The constructions of the
following two assembly maps  can be found
in~\cite[Definitions~3.8 and~5.10]{Bartels-Lueck(2023K-theory_red_p-adic_groups)}.
Since the following proofs are rather
formal, the reader does not need to know the definitions and constructions of $\calb$,
$\calb[G/U]$, $\calb[G/U]$, $\calb|_U$,  $\calb[G/U]_{\oplus}$, $\calb[G/U]_{\oplus}[\IZ^d]$,
$\contcover^\allG_G(P)$,  and $\contc^{\allG,0,\nowedge}_G( M )$
and of the following two assembly maps  to understand the
assertions and proofs of this subsection.

\begin{definition}[$\COP$-assembly map]\label{def:assembly-cop-calb} Let $G$ be a
  td-group and let $\calb$ be a Hecke  category with $G$-support.
  The projections $G/U \to G/G$
  induce a map
  \begin{equation}\label{eq:assembly-cop-calb}
    \hocolimunder_{G/U \in \Or_{\COP}(G)} \bfK \big( \calb[G/U] \big)
    \to \bfK \big( \calb[G/G] \big) \simeq \bfK(\calb).
  \end{equation}
  We call this the \emph{$\COP$-assembly map for $\calb$}.
\end{definition}

\begin{definition}[$\CVCYC$-assembly map]\label{def:assembly-CvCYC-calb} Let $G$
  be a td-group and let $\calb$ be a Hecke category with $G$-support.  The maps $P \to \ast$ for
  $P \in \EP\CVCYC(G)$ induce a map
  \begin{equation}\label{eq:FJ-assembly-map}
    \hocolimunder_{P \in \EP\CVCYC(G)} \bfK \big( \contc_G(P;\calb) \big) \;
    \to \; \bfK \big( \contc_G(\ast;\calb) \big).
  \end{equation}
  This is the \emph{$\CVCYC$-assembly map for $\calb$}.
\end{definition}

We want to prove

\begin{theorem}\label{the:generalized_FJC}
  Let $p$ be a prime. Assume that $G $ is modulo a compact subgroup isomorphic to a
  closed subgroup of a reductive $p$-adic group. Let $\calb$ be a category  with $G$-support.

  \begin{enumerate}
  \item\label{the:generalized_FJC:Z/N}
    Let $N$ be a natural number such that for the category  $\calb$ with $G$-support the underlying
    $\IZ$-category $\calb$ is obtained by restriction with the projection $\IZ \to \IZ/N$
    from a $\IZ/N$-category $\calb'$. Then the $\COP$-assembly map of~\eqref{eq:assembly-cop-calb}
    is a $\calp_N$-equivalence  for every $n \in \IZ$, where $\calp_N$ is the set of primes  dividing $N$;

  \item\label{the:generalized_FJC:(Reg)}
    The $\COP$-assembly map~\eqref{eq:assembly-cop-calb}  is a  weak homotopy equivalence,
    if $\calp(G) \subseteq \calp(\calb)$ holds.
 \end{enumerate}
\end{theorem}

  If $\calp(\calb)$ consists all primes, then
  Theorem~\ref{the:generalized_FJC}~\ref{the:generalized_FJC:(Reg)} has already been proved
  in~\cite[Theorem~1.11 and Theorem~1.15]{Bartels-Lueck(2023K-theory_red_p-adic_groups)}.

  In Subsection~\ref{subsec:Reduction_from_COP_to_CVCYC} the proof of Theorem~\ref{the:generalized_FJC}
  will be obtained by inspecting the
  proof in~\cite[Theorem~1.11 and Theorem~1.15]{Bartels-Lueck(2023K-theory_red_p-adic_groups)}
  taking Theorem~\ref{the:Consequences_of_the_W(Lambda)-module_structure_on_the_Nil-terms}~%
\ref{the:Consequences_of_the_W(Lambda)-module_structure_on_the_Nil-terms:characteristic} into account.


\subsection{Reduction from $\COP$ to $\CVCYC$}\label{subsec:Reduction_from_COP_to_CVCYC}

In this subsection we give the proof of Theorem~\ref{the:generalized_FJC}.

We conclude from by~\cite[Theorem~5.15]{Bartels-Lueck(2023K-theory_red_p-adic_groups)},

\begin{theorem}\label{the:VCYC_ass}
 The $\CVCYC$-assembly map~\eqref{def:assembly-CvCYC-calb} is a weak homotopy equivalence
for any  Hecke category with $G$-support $\calb$ if $G $ is  a reductive $p$-adic group.
\end{theorem}

So in order to
prove Theorem~\ref{the:generalized_FJC} in the special case where $G$ is a reductive $p$-adic group,
we only need to analyse the reduction from $\CVCYC$ to $\COP$ which has been
carried out in~\cite[Theorem~14.1]{Bartels-Lueck(2023K-theory_red_p-adic_groups)} in the
case that every prime number lies in $\calp(\calb)$.   In the sequel we often omit $\calb$ from the notation.

\begin{theorem}\label{thm:contC_G(P)-for-CVCYC-via-COMP_prime}
  Let $\calp$ be a (possibly empty) set of primes.
  Suppose that the $\CVCYC$-assembly 
  map~\eqref{eq:FJ-assembly-map}
  \[
    \hocolimunder_{P \in \EP\CVCYC(G)} \bfK \big( \contc_G(P) \big) \;
    \to \; \bfK \big( \contc_G(\ast) \big)
  \]
  is a $\calp$-equivalence and  that  for every 
  $P \in \EP\CVCYC(G)$ the canonical map  
  \[
    \hocolimunder_{(Q,f) \in \EP\Or_\COM(G) \downarrow P} \bfK \big( \contcover^\allG_G(Q) \big) 
    \to\bfK \big( \contcover^\allG_G(P) \big)
  \]
  is a $\calp$-equivalence.

  Then the $\COP$-assembly map~\eqref{eq:assembly-cop-calb} 
  \[
  \hocolimunder_{G/U \in \Or_{\COP}(G)} \bfK \big(
    \calb[G/U] \big) \; \to \; \bfK \big(\calb[G/G] \big)
    \simeq \bfK(\calb)
  \]
  is also $\calp$-equivalence.
\end{theorem}
\begin{proof}
The arguments on the proof that~\cite[Theorem~14.7]{Bartels-Lueck(2023K-theory_red_p-adic_groups)}
implies~\cite[Theorem~14.1]{Bartels-Lueck(2023K-theory_red_p-adic_groups)} presented
in~\cite[Section~14.B]{Bartels-Lueck(2023K-theory_red_p-adic_groups)} carry directly
over to our setting and lead directly to  a proof of Theorem~\ref{thm:contC_G(P)-for-CVCYC-via-COMP_prime}.
Namely, in the diagram appearing~\cite[Section~14.B]{Bartels-Lueck(2023K-theory_red_p-adic_groups)}
the maps $\alpha_1$ and $\widehat{\alpha_1}$
are $\calp$-equivalences and the other four  maps labeled by $\simeq$
are also weak homotopy equivalences in our setting.
\end{proof}

Fix $P = (G/V_1,\dots,G/V_n)$ with $V_i \in \CVCYC$.  Let $K_i \subseteq V_i$ be the
maximal compact open subgroup of $V_i$.  Set $M := (G/K_1,\dots,G/K_n)$.  The quotients
$\Gamma_i := V_i / K_i$ are either infinite cyclic or trivial.  Let
$\Gamma := \Gamma_1 \times \cdots \times \Gamma_n$.  Then $\Gamma$ is a finitely generated
free abelian group of rank at most $n$.  There are canonical maps
$h_i \colon \Gamma_i \to \homendo_{\Or(G)}(G/K_i)$, sending $\gamma \in \Gamma_i$ to
$G/K_i \to G/K_i, gK_i \mapsto g\widehat{\gamma} K_i$ for any choice $\widehat{\gamma } \in V_i$
which is mapped to $\gamma$ under the projection $V_i \to \Gamma_i$.
These combine to an action of $\Gamma$ on $M$ by
morphisms in $\EP\Or(G)(M)$.  We write $\underline{\Gamma}$ for the category with exactly
one object $\ast_\Gamma$ whose endomorphisms are given by $\Gamma$.  The action of
$\Gamma$ on $M$ determines a functor
$h \colon \underline{\Gamma} \to \EP\Or_\COM(G)\downarrow P$ that sends $\ast_\Gamma$ to
$\pi \colon M \to P$.  Now $h$ induces a map
$\bfK \big( \contc^{\allG,0,\nowedge}_G( M ) \big)
    \to \bfK \big( \contc^{\allG,0,\nowedge}_G( M )[\Gamma]\bigr)$,
where $\contc^{\allG,0,\nowedge}_G( M )[\Gamma]$ is the
additive category of twisted Laurent series with respect to the $\Gamma$-action on
$\contc^{\allG,0,\nowedge}_G( M )$.

\begin{lemma}\label{lem:reduction_to_(10)}
  Let $\calp$ be a set of primes (which may be empty). Then the following assertions are equivalent
  \begin{enumerate}
  \item\label{lem:reduction_to_(10)_in_EP_CVCYC(G)} 
  For all
  $P \in \EP\CVCYC(G)$ the canonical map  appearing in Theorem~\ref {thm:contC_G(P)-for-CVCYC-via-COMP_prime}
  \[
    \hocolimunder_{(Q,f) \in \EP\Or_\COM(G) \downarrow P} \bfK \big( \contcover^\allG_G(Q) \big) 
    \xrightarrow{\sim} \bfK \big( \contcover^\allG_G(P) \big)
  \]
  is a $\calp$-equivalence;

\item\label{lem:reduction_to_(10):underline(Gamma)} The canonical map
  \[\hocolimunder_{\underline{\Gamma}} \bfK \big( \contc^{\allG,0,\nowedge}_G( M ) \big)
    \to \bfK \big( \contc^{\allG,0,\nowedge}_G( M )[\Gamma]\bigr)
    \]
 is a $\calp$-equivalence.
\end{enumerate}
\end{lemma}
\begin{proof}
  This follows from~\cite[Proposition~14.9]{Bartels-Lueck(2023K-theory_red_p-adic_groups)}
  and the commutative diagram~\cite[14.10]{Bartels-Lueck(2023K-theory_red_p-adic_groups)}
  since the proof of~\cite[Proposition~14.9]{Bartels-Lueck(2023K-theory_red_p-adic_groups)},
  the construction of the commutative diagram~\cite[14.10]{Bartels-Lueck(2023K-theory_red_p-adic_groups)}
  and the fact  that the map (1) to (9) appearing
  there are weak homotopy equivalences presented
  in~\cite[Subsections~14.C,~14.E, and~14.F]{Bartels-Lueck(2023K-theory_red_p-adic_groups)}
  carry directly over to our setting,
  since no assumptions about $\calb$ are used.
\end{proof}

\begin{lemma}\label{lem:desired_calp(cala,G)-equivalence_in_characteristic_N}
  Let $N$ be a natural number such that for the category $\calb$ with $G$-support the
  underlying $\IZ$-category $\calb$ is obtained by restriction with the projection
  $\IZ \to \IZ/N$ from a $\IZ/N$-category $\calb'$. Then the
  \[
    \hocolimunder_{\underline{\Gamma}} \bfK \big( \contc^{\allG,0,\nowedge}_G( M ) \big)
    \to \bfK \big( \contc^{\allG,0,\nowedge}_G( M )[\Gamma]\bigr)
  \]
  is a $\calp$-homotopy equivalence for $\calp$ the set of primes dividing $N$.
\end{lemma}
\begin{proof}
  By inspecting the definitions the additive category $\contc^{\allG,0,\nowedge}_G( M )$
  inherits the structure of an additive $\IZ/N$-category. Now one proceeds by induction
  over the rank for $\Gamma$ using the isomorphism~\eqref{twisted_BHS_on_homotopy_groups}
  and Theorem~\ref{the:Consequences_of_the_W(Lambda)-module_structure_on_the_Nil-terms}~%
\ref{the:Consequences_of_the_W(Lambda)-module_structure_on_the_Nil-terms:characteristic}.
\end{proof}

\begin{lemma}\label{lem:desired_calp(cala,G)-equivalence_calp(calb;G)_is_empty}
  If $\calp(G) \subseteq \calp(\calb)$ holds, then
then canonical map
  \[\hocolimunder_{\underline{\Gamma}} \bfK \big( \contc^{\allG,0,\nowedge}_G( M ) \big)
    \to \bfK \big( \contc^{\allG,0,\nowedge}_G( M )[\Gamma]\bigr)
    \]
  is a weak homotopy equivalence.
\end{lemma}
 \begin{proof}
The proof that the map appearing in Lemma~\ref{lem:desired_calp(cala,G)-equivalence_calp(calb;G)_is_empty}
is a weak homotopy equivalence   is a variation of the one
appearing  in~\cite[Section~14.H]{Bartels-Lueck(2023K-theory_red_p-adic_groups)}.
  By the arguments in~\cite[Section~14.H]{Bartels-Lueck(2023K-theory_red_p-adic_groups)}
  it suffices to check that the map~\cite[(14.14)]{Bartels-Lueck(2023K-theory_red_p-adic_groups)}
  is a weak homotopy equivalence, since the map appearing in
  Lemma~\ref{lem:desired_calp(cala,G)-equivalence_calp(calb;G)_is_empty}
  is the map (7) appearing at the very bottom
  of the diagram~\cite[(14.10)]{Bartels-Lueck(2023K-theory_red_p-adic_groups)}.
  Hence it suffices to check that the argument appearing~\cite[Lemma~14.16]{Bartels-Lueck(2023K-theory_red_p-adic_groups)}
  carries over if we assume $\calp(G) \subseteq \calp(B)$ and do not use 
  not the stronger~\cite[Assumption~3.11]{Bartels-Lueck(2023K-theory_red_p-adic_groups)},
  which is equivalent to requiring that we can choose $\calp(\calb)$ to be the set of all primes,
  since then we can 
  apply~\cite[Theorem.~14.1]{Bartels-Lueck(2020additive)},
  exactly as we did  in~\cite[Section~14.H]{Bartels-Lueck(2023K-theory_red_p-adic_groups)} after the proof
  of~\cite[Lemma~14.16]{Bartels-Lueck(2023K-theory_red_p-adic_groups)}.

  Let $(U_{r,i})_{r = 1,\ldots,n, i \in \IN_{\ge 1}}$ be the system of compact open subgroups of $G$
  as they appear in~\cite[Section~14.H]{Bartels-Lueck(2023K-theory_red_p-adic_groups)}.
  In view of the proof of~\cite[Lemma~14.16]{Bartels-Lueck(2023K-theory_red_p-adic_groups)}, it remains to explain
  why for every $i \in \{1,2, \ldots, n \}$ and every  element $\lambda \in |Q_i|$ the category
  $\bigl((\calb|_{G_{\lambda}})_\oplus\bigr)[\IZ^d]$
  is $l$-uniformly regular coherent,  where $|Q_i| = G/U_{1,i} \times \cdots \times G/U_{n,i}$.
  In the proof of~\cite[Lemma~14.16]{Bartels-Lueck(2023K-theory_red_p-adic_groups)}
  we had  used~\cite[Assumption~3.11]{Bartels-Lueck(2023K-theory_red_p-adic_groups)}
  precisely at this place, but nowhere else
  in the proof of~\cite[Theorem~14.7]{Bartels-Lueck(2023K-theory_red_p-adic_groups)}, and we 
  do not want to use this assumption  here. In the situation here it suffices to show
  $\calp(G_{\lambda})  \subseteq \calp(\calb)$. This follows from the assumption $\calp(G) \subseteq \calp(\calb)$ and
  Lemma~\ref{lem:cal_and_subgroups}.
This finishes the proof of Lemma~\ref{lem:desired_calp(cala,G)-equivalence_calp(calb;G)_is_empty}.
\end{proof}

\begin{proof}[Proof of Theorem~\ref{the:generalized_FJC}]
Theorem~\ref{the:generalized_FJC} follows from Theorem~\ref{the:VCYC_ass},
Theorem~\ref{thm:contC_G(P)-for-CVCYC-via-COMP_prime}, Lemma~\ref{lem:reduction_to_(10)},
Lemma~\ref{lem:desired_calp(cala,G)-equivalence_in_characteristic_N}, and
Lemma~\ref{lem:desired_calp(cala,G)-equivalence_calp(calb;G)_is_empty}, provided
that $G$ is a reductive $p$-adic group. Now the general case, where $G $ is modulo a compact subgroup isomorphic to a
closed subgroup of a reductive $p$-adic group, follows from the proof of~\cite[Theorem~1.5]{Bartels-Lueck(2023foundations)},
which directly carries over to our setting.
\end{proof}

We state the following version of the $\COP$-Farrell-Jones Conjecture.

\begin{notation}[$\calp(\calb,G)$]\label{not:calp(calb,G)}
  Let $\calb$ be Hecke category with $G$-support. Define $\calp(\calb,G)$ to be the set of primes
  $q$ that  belong to $\calp(G)$ but not to $\calp(\calb)$.
\end{notation}

\begin{conjecture}[(Generalized) $\COP$-Farrell-Jones Conjecture]%
\label{con:(Generalized)_COP-FJCT_at_the_prime_p}
The td-group $G$ satisfies the \emph{(Generalized) $\COP$-Farrell-Jones Conjecture},
if for every Hecke category with $G$-support the $\COP$-assembly map~\eqref{eq:assembly-cop-calb} 
  \begin{equation*} \hocolimunder_{G/U \in \Or_{\COP}(G)} \bfK \big(
    \calb[G/U] \big) \; \to \; \bfK \big( \calb[G/G] \big)
    \simeq \bfK (\calb)
  \end{equation*}
  is a $\calp(\calb,G)$-equivalence.
\end{conjecture}

If $\calp(G) \subseteq \calp(\calb)$, then Theorem~\ref{the:generalized_FJC}~\ref{the:generalized_FJC:(Reg)}
confirms Conjecture~\ref{con:(Generalized)_COP-FJCT_at_the_prime_p} if
$G $ is modulo a compact subgroup isomorphic to a closed subgroup of a reductive $p$-adic group.


\subsection{Proof of  Theorem~\ref{the:passage_from_COP_to_CVcyc_in_characteristic_N}}%
\label{subsec:Proof_of_Theorem_ref(the:passage_from_COP_to_CVcyc_in_characteristic_N)}

\begin{proof}[Proof of Theorem~\ref{the:passage_from_COP_to_CVcyc_in_characteristic_N}]%
\ref{the:passage_from_COP_to_CVcyc_in_characteristic_N:general_char_p}
  We conclude from Lemma~\ref{lem:desired_calp(cala,G)-equivalence_in_characteristic_N}
  that  the canonical map
  \[\hocolimunder_{\underline{\Gamma}} \bfK \big( \contc^{\allG,0,\nowedge}_G( M ) \big)
    \to \bfK \big( \contc^{\allG,0,\nowedge}_G( M )[\Gamma]\bigr)
    \]
    is a $\calp_N$-homotopy equivalence for $\calp_N$ the set of primes dividing $N$
    provided that $N$ is  a natural number such that for the category  $\calb$ with $G$-support the underlying
    $\IZ$-category $\calb$ is obtained by restriction with the projection $\IZ \to \IZ/N$
    from a $\IZ/N$-category $\calb'$.

    Assertion~\ref{the:passage_from_COP_to_CVcyc_in_characteristic_N:general_char_p}
    of Theorem~\ref{the:passage_from_COP_to_CVcyc_in_characteristic_N}
    follows from Theorem~\ref{the:generalized_FJC}~\ref{the:generalized_FJC:(Reg)},
    as for the specific choice of $\calb = \calb(G,R,\rho,\omega)$ appearing in
    Remark~\ref{rem:calp(underline(R))} the map
    assembly~\eqref{eq:COP-assembly-homolgy-theory-R} is obtained from
    the $\COP$-assembly map~\eqref{eq:assembly-cop-calb}
    by applying $\pi_n$, see~\cite[Section~6.D]{Bartels-Lueck(2023foundations)} and
    $ \calb(G,R,\rho,\omega)$ has the required property above because of the assumption $N \cdot 1_R = 0$.
    \\[1mm]~\ref{the:passage_from_COP_to_CVcyc_in_characteristic_N:Artinian} In view of
    the arguments appearing in the proof of~\cite[Theorem~2.16]{Bartels-Lueck(2023recipes)}
    based on the equivariant Atyiah-Hirzebruch spectra sequence
    of~\cite[Theorem~2.1]{Bartels-Lueck(2023recipes)}, it suffices to show that
    $K_n(\calh(p^{-1}(U),R,\rho|_{p^{-1}(U)},\omega)) \cong
    H_n^G(G/U;\bfK_{\calb(G,R,\rho;\omega)})$ vanishes for every $n \le -1$ and every
    compact open subgroup $U \subseteq Q$.  The arguments in the proof
    of~\cite[Lemma~8.1]{Bartels-Lueck(2023forward)} apply also to our setting and
    imply that it suffices to show for any compact normal subgroup $K \subseteq p^{-1}(U)$
    with $K \in P$ that $K_n(\calh(p^{-1}(U)//K,R,\rho|_{p^{-1}(U)},\omega)) = 0$ holds
    for $n \le -1$ The proof of~\cite[Lemma~7.5]{Bartels-Lueck(2023forward)} carries over to our setting
    and implies  that  $\calh(p^{-1}(U)//K,R,\rho|_{p^{-1}(U)},\omega)$ can be identified
    as a crossed product ring $R \ast D$ for some group $D$ invertible in $R$. As $R$ is
    Artinian, $R \ast D$ and hence are $\calh(p^{-1}(U)//K,R,\rho|_{p^{-1}(U)},\omega)$
    are Artinian. Now we conclude from~\cite[Theorem~4.15~(ii)]{Lueck(2022book)}
    applied in the case of a trivial group and $k = 0$ that
    $K_n(\calh(p^{-1}(U)//K,R,\rho|_{p^{-1}(U)},\omega)) = 0$ holds for $n \le -1$.  This
    finishes the proof of
    assertion~\ref{the:passage_from_COP_to_CVcyc_in_characteristic_N:Artinian}.
    $K \subseteq Q$.
    \\[1mm]~\ref{the:passage_from_COP_to_CVcyc_in_characteristic_N:general} The proof that
    the assembly~\eqref{eq:COP-assembly-homolgy-theory-R} is bijective is analogous to the
    proof of
    assertion~\ref{the:passage_from_COP_to_CVcyc_in_characteristic_N:general_char_p} but
    now using Lemma~\ref{lem:desired_calp(cala,G)-equivalence_calp(calb;G)_is_empty}
    instead of Lemma~\ref{lem:desired_calp(cala,G)-equivalence_in_characteristic_N}.
\end{proof}

\begin{remark}\label{rem:possible_integral_results}
  It is conceivable that
  assertion~\ref{the:passage_from_COP_to_CVcyc_in_characteristic_N:Artinian} appearing in
  Theorem~\ref{the:passage_from_COP_to_CVcyc_in_characteristic_N} is still true if we do not invert $N$.
  (We have no proof.)
  For instance  it may be true that, for a prime $p$, an Artinian ring $R$ for which $p$ is invertible in $R$,
  and a subgroup $G$ of a reductive $p$-adic group, 
  \[
    K_n \big(\calh(G;R)\big) = 0 \quad \text{for}\; n \le -1
  \]
  holds and the map
  \[
    \colimunder_{U \in \Sub_\COP(G)} \Kgroup_0 (\calh(U;R)) \to \Kgroup_0 (\calh(G;R))
  \]
  is bijective, or, equivalently, that the assembly
  map~\eqref{eq:COP-assembly-homolgy-theory-R} is bijective for $n \le 0$.  (The latter is
  not true in general under the assumptions above for $n \ge 1$.)
\end{remark}



\begin{thebibliography}{10}

\bibitem{Almkvist(1974)}
G.~Almkvist.
\newblock The {G}rothendieck ring of the category of endomorphisms.
\newblock {\em J. Algebra}, 28:375--388, 1974.

\bibitem{Bartels-Echterhoff-Lueck(2008colim)}
A.~Bartels, S.~Echterhoff, and W.~L\"uck.
\newblock Inheritance of isomorphism conjectures under colimits.
\newblock In Cortinaz, Cuntz, Karoubi, Nest, and Weibel, editors, {\em
  {K}-Theory and noncommutative geometry}, EMS-Series of Congress Reports,
  pages 41--70. European Mathematical Society, 2008.

\bibitem{Bartels-Lueck(2007ind)}
A.~Bartels and W.~L{\"u}ck.
\newblock Induction theorems and isomorphism conjectures for {$K$}- and
  {$L$}-theory.
\newblock {\em Forum Math.}, 19:379--406, 2007.

\bibitem{Bartels-Lueck(2020additive)}
A.~Bartels and W.~L{\"u}ck.
\newblock Vanishing of {N}il-terms and negative {$K$}-theory for additive
  categories.
\newblock Preprint, arXiv:2002.03412 [math.KT], 2020.

\bibitem{Bartels-Lueck(2023K-theory_red_p-adic_groups)}
A.~Bartels and W.~L{\"u}ck.
\newblock Algebraic {$K$}-theory of reductive $p$-adic groups.
\newblock Preprint, arXiv:2306.03452 [math.KT], 2023.

\bibitem{Bartels-Lueck(2023foundations)}
A.~Bartels and W.~L{\"u}ck.
\newblock Inheritance properties of the {$K$}-theoretic {F}arrell-{J}ones
  {C}onjecture for totally disconnected groups.
\newblock Preprint arXiv:2306.01518 [math.KT], 2023.

\bibitem{Bartels-Lueck(2023forward)}
A.~Bartels and W.~L\"{u}ck.
\newblock On the {A}lgebraic {K}-{T}heory of {H}ecke {A}lgebras.
\newblock In {\em Mathematics {G}oing {F}orward}, volume 2313 of {\em Lecture
  Notes in Math.}, pages 241--277. Springer, Cham, 2023.

\bibitem{Bartels-Lueck(2023recipes)}
A.~Bartels and W.~L{\"u}ck.
\newblock Recipes to compute the algebraic {$K$}-theory of {H}ecke algebras.
\newblock Preprint, arXiv:2306.01510 [math.KT], to appear in Algebraic and Geometric Topology, 2023.

\bibitem{Bartels-Lueck-Reich(2008appl)}
A.~Bartels, W.~L\"uck, and H.~Reich.
\newblock On the {F}arrell-{J}ones {C}onjecture and its applications.
\newblock {\em Journal of Topology}, 1:57--86, 2008.

\bibitem{Bartels(2003b)}
A.~C. Bartels.
\newblock On the domain of the assembly map in algebraic {$K$}-theory.
\newblock {\em Algebr. Geom. Topol.}, 3:1037--1050 (electronic), 2003.

\bibitem{Bass(1968)}
H.~Bass.
\newblock {\em Algebraic ${K}$-theory}.
\newblock W. A. Benjamin, Inc., New York-Amsterdam, 1968.

\bibitem{Bass-Murthy(1967)}
H.~Bass and M.~P. Murthy.
\newblock Grothendieck groups and {P}icard groups of abelian group rings.
\newblock {\em Ann. of Math. (2)}, 86:16--73, 1967.

\bibitem{Bloch(1977)}
S.~Bloch.
\newblock Algebraic {$K$}-theory and crystalline cohomology.
\newblock {\em Inst. Hautes \'{E}tudes Sci. Publ. Math.}, 47:187--268 (1978),
  1977.

\bibitem{Blondel(2011)}
C.~Blondel.
\newblock Basic representation theory of reductive $p$-adic groups.
\newblock unpublished notes,
  https://webusers.imj-prg.fr/\~corinne.blondel/Blondel\_Beijin.pdf, 2011.

\bibitem{Bunke-Kasprowski-Winges(2021)}
U.~Bunke, D.~Kasprowski, and C.~Winges.
\newblock On the {F}arrell-{J}ones conjecture for localising invariants.
\newblock Preprint, arXiv:2111.02490 [math.KT], 2021.

\bibitem{Davis-Quinn-Reich(2011)}
J.~F. Davis, F.~Quinn, and H.~Reich.
\newblock {Algebraic $K$-theory over the infinite dihedral group: a controlled
  topology approach.}
\newblock {\em J. Topol.}, 4(3):505--528, 2011.

\bibitem{Dress(1973)}
A.~W.~M. Dress.
\newblock Contributions to the theory of induced representations.
\newblock In {\em Algebraic $K$-theory, II: ``Classical'' algebraic $K$-theory
  and connections with arithmetic (Proc. Conf., Battelle Memorial Inst.,
  Seattle, Wash., 1972)}, pages 183--240. Lecture Notes in Math., Vol. 342.
  Springer, Berlin, 1973.

\bibitem{Dress(1975)}
A.~W.~M. Dress.
\newblock Induction and structure theorems for orthogonal representations of
  finite groups.
\newblock {\em Ann. of Math. (2)}, 102(2):291--325, 1975.

\bibitem{Farrell(1977)}
F.~T. Farrell.
\newblock The nonfiniteness of {N}il.
\newblock {\em Proc. Amer. Math. Soc.}, 65(2):215--216, 1977.

\bibitem{Grunewald(2008Nil)}
J.~Grunewald.
\newblock The behavior of {N}il-groups under localization and the relative
  assembly map.
\newblock {\em Topology}, 47(3):160--202, 2008.

\bibitem{Hambleton-Lueck(2012)}
I.~Hambleton and W.~L\"uck.
\newblock Induction and computation of {B}ass {N}il groups for finite groups.
\newblock {\em PAMQ}, 8(1):199--219, 2012.

\bibitem{Harmon(1987)}
D.~R. Harmon.
\newblock {$NK\sb 1$} of finite groups.
\newblock {\em Proc. Amer. Math. Soc.}, 100(2):229--232, 1987.

\bibitem{Lueck(2005s)}
W.~L{\"u}ck.
\newblock Survey on classifying spaces for families of subgroups.
\newblock In {\em Infinite groups: geometric, combinatorial and dynamical
  aspects}, volume 248 of {\em Progr. Math.}, pages 269--322. Birkh\"auser,
  Basel, 2005.

\bibitem{Lueck(2019handbook)}
W.~L\"{u}ck.
\newblock Assembly maps.
\newblock In H.~Miller, editor, {\em Handbook of Homotopy Theory}, Handbook in
  Mathematics Series, pages 853--892. CRC Press/Chapman and Hall, Boca Raton,
  FL, 2019.

\bibitem{Lueck(2022book)}
W.~L\"uck.
\newblock {I}somorphism {C}onjectures in {$K$}- and {$L$}-theory.
\newblock in preparation, see http://him-lueck.uni-bonn.de/data/ic.pdf, 2024.

\bibitem{Lueck-Reich(2005)}
W.~L{\"u}ck and H.~Reich.
\newblock The {B}aum-{C}onnes and the {F}arrell-{J}ones conjectures in {$K$}-
  and {$L$}-theory.
\newblock In {\em Handbook of $K$-theory. Vol. 1, 2}, pages 703--842. Springer,
  Berlin, 2005.

\bibitem{Lueck-Steimle(2014delooping)}
W.~L{\"u}ck and W.~Steimle.
\newblock Non-connective {$K$}- and {N}il-spectra of additive categories.
\newblock In {\em An alpine expedition through algebraic topology}, volume 617
  of {\em Contemp. Math.}, pages 205--236. Amer. Math. Soc., Providence, RI,
  2014.

\bibitem{Lueck-Steimle(2016splitasmb)}
W.~L{\"u}ck and W.~Steimle.
\newblock Splitting the relative assembly map, {N}il-terms and involutions.
\newblock {\em Ann. K-Theory}, 1(4):339--377, 2016.

\bibitem{Lueck-Steimle(2016BHS)}
W.~L{\"u}ck and W.~Steimle.
\newblock A twisted {B}ass--{H}eller--{S}wan decomposition for the algebraic
  {$K$}-theory of additive categories.
\newblock {\em Forum Math.}, 28(1):129--174, 2016.

\bibitem{Meyer-Solleveld(2010)}
R.~Meyer and M.~Solleveld.
\newblock Resolutions for representations of reductive {$p$}-adic groups via
  their buildings.
\newblock {\em J. Reine Angew. Math.}, 647:115--150, 2010.

\bibitem{Moody(1987)}
J.~A. Moody.
\newblock Induction theorems for infinite groups.
\newblock {\em Bull. Amer. Math. Soc. (N.S.)}, 17(1):113--116, 1987.

\bibitem{Pedersen-Taylor(1978)}
E.~K. Pedersen and L.~R. Taylor.
\newblock The {W}all finiteness obstruction for a fibration.
\newblock {\em Amer. J. Math.}, 100(4):887--896, 1978.

\bibitem{Saunier(2023)}
V.~Saunier.
\newblock The fundamental theorem of localizing invariants.
\newblock {\em Ann. K-Theory}, 8(4):609--643, 2023.

\bibitem{Schlichting(2006)}
M.~Schlichting.
\newblock Negative {$K$}-theory of derived categories.
\newblock {\em Math. Z.}, 253(1):97--134, 2006.

\bibitem{Stienstra(1982)}
J.~Stienstra.
\newblock Operations in the higher {$K$}-theory of endomorphisms.
\newblock In {\em Current trends in algebraic topology, Part 1 (London, Ont.,
  1981)}, volume~2 of {\em CMS Conf. Proc.}, pages 59--115. Amer. Math. Soc.,
  Providence, R.I., 1982.

\bibitem{Swan(1960a)}
R.~G. Swan.
\newblock Induced representations and projective modules.
\newblock {\em Ann. of Math. (2)}, 71:552--578, 1960.

\bibitem{Swan(1963)}
R.~G. Swan.
\newblock The {G}rothendieck ring of a finite group.
\newblock {\em Topology}, 2:85--110, 1963.

\bibitem{Dieck(1979)}
T.~tom Dieck.
\newblock {\em Transformation groups and representation theory}.
\newblock Springer-Verlag, Berlin, 1979.

\bibitem{Ullmann-Winges(2019)}
M.~Ullmann and C.~Winges.
\newblock On the {F}arrell--{J}ones conjecture for algebraic {K}-theory of
  spaces: the {F}arrell--{H}siang method.
\newblock {\em Ann. K-Theory}, 4(1):57--138, 2019.

\bibitem{Kallen(1971)}
W.~van~der Kallen.
\newblock Le {$K_{2}$} des nombres duaux.
\newblock {\em C. R. Acad. Sci. Paris S\'{e}r. A-B}, 273:A1204--A1207, 1971.

\bibitem{Weibel(1981)}
C.~A. Weibel.
\newblock Mayer-{V}ietoris sequences and module structures on {$NK\sb\ast $}.
\newblock In {\em Algebraic $K$-theory, Evanston 1980 (Proc. Conf.,
  Northwestern Univ., Evanston, Ill., 1980)}, volume 854 of {\em Lecture Notes
  in Math.}, pages 466--493. Springer, Berlin, 1981.

\end{thebibliography}


\end{document}